\documentclass[final,leqno,a4paper]{amsart}
\usepackage{amssymb,amsfonts,enumerate}
\usepackage{stmaryrd,xcolor}
\usepackage{cite}
\usepackage{graphicx}
\usepackage{geometry}
\usepackage{amsaddr}
\usepackage[hidelinks]{hyperref}
\usepackage[capitalize,nameinlink,noabbrev]{cleveref}

\newtheorem{theorem}{Theorem}[section]
\newtheorem{lemma}[theorem]{Lemma}
\newtheorem{proposition}[theorem]{Proposition}
\newtheorem{corollary}[theorem]{Corollary}

\theoremstyle{remark}
\newtheorem{remark}[theorem]{Remark}

\numberwithin{equation}{section}
\numberwithin{figure}{section}
\numberwithin{table}{section}

\crefname{example}{Example}{Examples}
\crefname{hypothesis}{Hypothesis}{Hypotheses}
\crefname{conj}{Conjecture}{Conjectures}
\crefformat{equation}{\textup{#2(#1)#3}}
\crefrangeformat{equation}{\textup{#3(#1)#4--#5(#2)#6}}
\crefmultiformat{equation}{\textup{#2(#1)#3}}{ and \textup{#2(#1)#3}}
{, \textup{#2(#1)#3}}{, and \textup{#2(#1)#3}}
\crefrangemultiformat{equation}{\textup{#3(#1)#4--#5(#2)#6}}%
{ and \textup{#3(#1)#4--#5(#2)#6}}{, \textup{#3(#1)#4--#5(#2)#6}}{, and \textup{#3(#1)#4--#5(#2)#6}}

\Crefformat{equation}{#2Equation~\textup{(#1)}#3}
\Crefrangeformat{equation}{Equations~\textup{#3(#1)#4--#5(#2)#6}}
\Crefmultiformat{equation}{Equations~\textup{#2(#1)#3}}{ and \textup{#2(#1)#3}}
{, \textup{#2(#1)#3}}{, and \textup{#2(#1)#3}}
\Crefrangemultiformat{equation}{Equations~\textup{#3(#1)#4--#5(#2)#6}}%
{ and \textup{#3(#1)#4--#5(#2)#6}}{, \textup{#3(#1)#4--#5(#2)#6}}{, and \textup{#3(#1)#4--#5(#2)#6}}

\newcommand{\R}{\mathbb{R}}
\newcommand{\C}{\mathbb{C}}
\newcommand{\Z}{\mathbb{Z}}
\newcommand{\vsigma}{\varsigma}
\newcommand{\mcalL}{\mathcal{L}}

\newcommand{\vx}{\mathbf{x}}
\newcommand{\rmd}{\mathrm{d}}
\newcommand{\vep}{\varepsilon}
\newcommand{\vphi}{\varphi}
\newcommand{\phiu}[1]{\vphi{(#1)}}

\newcommand{\psihn}[1]{\psi^{\langle #1 \rangle}}

\begin{document}

\title[Optimal error bounds on time-splitting methods for the NLSE]{Optimal error bounds on time-splitting methods for the nonlinear Schr\"odinger equation with low regularity potential and nonlinearity}
 \author[W. Bao]{Weizhu Bao}
 \address{Department of Mathematics, \\
	National University of Singapore, Singapore 119076, Singapore}
\email{matbaowz@nus.edu.sg}

\author[Y. Ma]{Ying Ma}
\address{Department of Mathematics, \\
	School of Mathematics, Statistics and Mechanics, \\ 
	Beijing University of Technology, Beijing 100124, China}
\email{maying@bjut.edu.cn}

\author[C. Wang]{Chushan Wang}
\address{Department of Mathematics, \\
	National University of Singapore, Singapore 119076, Singapore}
\email{e0546091@u.nus.edu}

\begin{abstract}
We establish optimal error bounds on time-splitting methods for the nonlinear Schr\"odinger equation with low regularity potential and typical power-type nonlinearity $ f(\rho) = \rho^\sigma $, where $ \rho:=|\psi|^2 $ is the density with $ \psi $ the wave function and $ \sigma > 0 $ the exponent of the nonlinearity. For the first-order Lie-Trotter time-splitting method, optimal $ L^2 $-norm error bound is proved for $L^\infty$-potential and $ \sigma > 0 $, and optimal $H^1$-norm error bound is obtained for $ W^{1, 4} $-potential and $ \sigma \geq 1/2 $. For the second-order Strang time-splitting method, optimal $ L^2 $-norm error bound is established for $H^2$-potential and $ \sigma \geq 1 $, and optimal $H^1$-norm error bound is proved for $H^3$-potential and $ \sigma \geq 3/2 $ (or $\sigma = 1$). Compared to those error estimates of time-splitting methods in the literature, our optimal error bounds either improve the convergence rates under the same regularity assumptions or significantly relax the regularity requirements on potential and nonlinearity for optimal convergence orders. A key ingredient in our proof is to adopt a new technique called \textit{regularity compensation oscillation} (RCO), where low frequency modes are analyzed by phase cancellation, and high frequency modes are estimated by regularity of the solution. Extensive numerical results are reported to confirm our error estimates and to demonstrate that they are sharp. 
\end{abstract}



\thanks{The work of the first author was partially supported by the Ministry of Education of Singapore under its AcRF Tier 2 funding MOE-T2EP20122-0002 (A-8000962-00-00). } 

\maketitle

\section{Introduction}\label{sec1}

In this paper, we consider the following nonlinear Schr\"odinger equation (NLSE)
\begin{equation}\label{NLSE}
	\left\{
	\begin{aligned}
		&i \partial_t \psi(\vx, t) = -\Delta \psi(\vx, t) + V(\vx) \psi(\vx, t) + f(|\psi(\vx, t)|^2) \psi(\vx, t), \quad  \vx \in \Omega, \  t>0, \\
		&\psi(\vx, 0) = \psi_0(\vx), \quad \vx \in \overline{\Omega},
	\end{aligned}
	\right.
\end{equation}
where $t$ is time, $\vx\in \R^d$ ($d=1, 2, 3$) is the spatial coordinate, $ \psi:=\psi(\vx, t) $ is a complex-valued wave function, $ \Omega = \Pi_{i=1}^d (a_i, b_i) \subset \R^d $ is a bounded domain equipped with periodic boundary condition. Here, $ V: \Omega \rightarrow \R $ is a time-independent real-valued potential which is assumed to be purely bounded, and $f$ is the power-type nonlinearity given as 
\begin{equation}\label{eq:f}
	f(\rho) = \beta \rho^\sigma, \quad \rho:=|\psi|^2 \geq 0,
\end{equation}
where $\beta \in \R$ is a given constant and $\sigma > 0$ is the exponent of the nonlinearity.
%The NLSE \cref{NLSE} conserves the mass
%\begin{equation}
%	M(\psi(\cdot, t)) = \int_{\Omega} |\psi(\vx, t)|^2 \rmd \vx \equiv M(\psi_0), \quad t \geq 0,
%\end{equation}
%and the energy
%	\begin{align}
	%		E(\psi(\cdot, t))
	%		&= \int_{\Omega} \left[ |\nabla \psi(\vx, t)|^2 + V(\vx)|\psi(\vx, t)|^2 + F(|\psi(\vx, t)|^2) \right] \rmd \vx \notag \\
	%		&\equiv E(\psi_0), \quad t \geq 0,
	%	\end{align}
%where the interaction energy density $F(\rho)$ is given as
%\begin{equation}\label{eq:F}
%	F(\rho) = \int_0^\rho f(s) \rmd s = \frac{\beta}{\sigma+1} \rho^{\sigma+1}, \quad \rho \geq 0.
%\end{equation}

When $V(\vx) = |\vx|^2 / 2$ and $f(\rho) = \beta \rho$, the NLSE reduces to the cubic NLSE with harmonic potential (or the NLSE with smooth potential and nonlinearity), also known as the Gross-Pitaevskii equation (GPE), which is widely adopted for modeling and simulation in quantum mechanics, nonlinear optics, and Bose-Einstein condensation (BEC) \cite{review_2013, ESY, NLS}. For the GPE, many accurate and efficient numerical methods have been proposed and analyzed in last two decades, including finite difference time domain (FDTD) methods \cite{FD,review_2013,Ant,henning2017}, exponential wave integrators (EWI) \cite{ExpInt,SymEWI,bao2023_EWI}, time-splitting methods \cite{BBD,lubich2008,review_2013,schratz2016,Ant,splitting_low_reg,ji2023,RCO_SE}, 
%the finite element method \cite{FEM1,FEM2,FEM3,FEM4,henning2017}, 
and, recently, low regularity (resonance based Fourier) integrators (LRI) \cite{LRI,LRI_sinum,LRI_error,LRI_general,LRI_fulldisc,LRI_sec,tree1,tree2} designed for the cubic NLSE with low regularity initial data. 

While the cubic NLSE or GPE is prevalent, diverse physics applications require the incorporation of low regularity potential and nonlinearity into the NLSE \cref{NLSE}. Typical examples include the square-well potential, which is discontinuous, the disorder potential considered in the study of Anderson localization \cite{poten_Josephson,poten_anderson}, and the non-integer power nonlinearity present in the Lee-Huang-Yang correction \cite{LHY} for modeling quantum droplets \cite{QD1,QD2,QD3,QD4}. For more applications involving low regularity potential and nonlinearity, we refer the readers to \cite{bao2023_semi_smooth,bao2023_EWI,henning2017,zhao2021} and references therein. 

Most numerical methods for the cubic NLSE with smooth potential can be extended straightforwardly to solve the NLSE \cref{NLSE} with aforementioned low regularity potential and/or nonlinearity, e.g., the FDTD method \cite{henning2017}, the time-splitting method \cite{ignat2011,choi2021,bao2023_semi_smooth}, the EWI \cite{bao2023_EWI} and the LRI \cite{zhao2021,tree2,bronsard2022}. However, the error analysis of them with low regularity potential and/or nonlinearity is a very subtle and challenging question, which requires new techniques and in-depth analysis. For low regularity potential, the first error estimate concerning $L^\infty$-potential was obtained in \cite{henning2017} for the Crank-Nicolson Galerkin method. Recently, some low regularity integrators \cite{zhao2021,tree2,bronsard2022} are designed, aiming at reducing the regularity requirements on potential and the exact solution at the same time. In these works, low regularity nonlinearity is not taken into account. Very recently, a first-order EWI was analyzed in \cite{bao2023_EWI} for both low regularity potential and nonlinearity.

In terms of time-splitting methods, the error estimates for the NLSE with smooth potential and nonlinearity have been well-established, and we refer to \cite{BBD,lubich2008,review_2013,schratz2016,splitting_low_reg,ji2023} for the details. In the presence of low regularity potential and nonlinearity (especially the non-integer power nonlinearity), the first-order Lie-Trotter time-splitting method was analyzed in \cite{ignat2011,choi2021,bao2023_semi_smooth}. In particular, under the assumption of $H^2$-solution, the first-order $L^2$-norm error bound in time requires $\sigma \geq 1/2$ and $H^2$-potential \cite{bao2023_semi_smooth}. This is essentially due to the Laplacian operator in the commutator bound, which takes two additional derivatives on potential and nonlinearity (as well as the exact solution). When considering low regularity nonlinearity with $0<\sigma<1/2$, only $1/2+\sigma$ order convergence in $L^2$-norm can be proved \cite{bao2023_semi_smooth}. Moreover, for low regularity $L^\infty$-potential, there are no convergence results at any order available. However, for the NLSE to be well-posed in $H^2$, it suffices to assume $L^\infty$-potential and $\sigma>0$ (see \cref{rem:regularity}). Then the expected ``optimal" error bound should be able to provide first-order temporal convergence (and second-order spatial convergence when considering full discretizations) in $L^2$-norm under the assumptions of $L^\infty$-potential, $\sigma>0$ and $H^2$-solution. In other words, the optimal error bound needs to satisfy: (i) the convergence order is optimal with respect to the order of the numerical method and the regularity of the exact solution, and (ii) the regularity requirements on potential and nonlinearity for the optimal convergence order should be optimally weak, i.e., in line with the regularity needed for the well-posedness of the equation. In this aspect, all the aforementioned error estimates of time-splitting methods for low regularity potential and nonlinearity are certainly not ``optimal". However, it is worth noting that the first-order $L^2$-norm error bound in time is observed numerically for $\sigma > 0$ in \cite{bao2023_semi_smooth}, which motivates us to improve the error estimates of time-splitting methods and to establish ``optimal" error bounds. Surprisingly, our improved error bounds are also valid for low regularity $L^\infty$-potential in addition to low regularity nonlinearity. Moreover, the new analysis techniques can also be extended to relax the regularity requirements on potential and nonlinearity for higher-order error bounds on time-splitting methods in the literature.

The main aim of this paper is to improve the error estimates of time-splitting methods and establish optimal error bounds under much weaker regularity assumptions on potential and nonlinearity. Main results are summarized in \cref{sec:main_results}. Compared to the error estimates of time-splitting methods in the literature \cite{BBD,lubich2008,review_2013,choi2021,ignat2011,bao2023_semi_smooth}, to obtain optimal convergence rates, our results significantly relax the regularity requirements on both potential and nonlinearity: 
\begin{enumerate}[(i)]
	\item for the first-order Lie-Trotter time-splitting method (\cref{thm:LT}): 
	\begin{itemize}
		\item optimal $L^2$-norm error bound is proved for $L^\infty$-potential and $\sigma > 0$, which relaxes the assumption of $H^2$-potential and $\sigma \geq 1/2$ in \cite{bao2023_semi_smooth,ignat2011};
		\item optimal $H^1$-norm error bound is obtained for $W^{1,4}$-potential and $\sigma \geq 1/2$, which weakens the requirement of $H^3$-potential and $\sigma \geq 1$ in \cite{bao2023_semi_smooth}. 
	\end{itemize}
	\item for the second-order Strang time-splitting method (\cref{thm:ST}): 
	\begin{itemize}
		\item optimal $L^2$-norm error bound is established for $H^2$-potential and $ \sigma \geq 1 $, which improved the assumption of $H^4$-potential and $ \sigma \geq 3/2$ (or $\sigma=1$) in \cite{BBD,lubich2008,review_2013};
		\item optimal $H^1$-norm error bound is proved for $H^3$-potential and $\sigma \geq 3/2$ (or $\sigma = 1$), which relaxes the need of $H^5$-potential and $\sigma \geq 2$ (or $\sigma=1$) in \cite{review_2013}. 
	\end{itemize}
\end{enumerate}
Roughly speaking, the differentiability requirement on potential is reduced by two orders and that on nonlinearity is reduced by one order (or $1/2$ in terms of $\sigma$). The price to pay is introducing a CFL-type time step size restriction $\tau \lesssim h^2$ where $\tau$ and $h$ are the time step size and mesh size, respectively. Such time step size restriction is natural in terms of the balance between temporal and spatial errors. Moreover, it can be observed numerically when potential is of low regularity, indicating that it cannot be removed or improved! 

Here, we briefly explain the idea of our analysis. As highlighted in \cite{bao2023_semi_smooth}, to obtain the optimal $L^2$-norm error bound for the Lie-Trotter time-splitting method for any $\sigma >0$, one must be able to capture the error cancellation between different steps. To analyze the error cancellation, inspired by a recently developed technique called regularity compensation oscillation (RCO) \cite{RCO_SE}, we truncate the frequency according to the time step size and use summation by parts formula in the estimate of accumulation of dominant local errors. As a result, the Laplacian operator in the commutator, which requires the strongest regularity on potential and nonlinearity, is replaced by a first-order temporal derivative, thereby reducing the differentiability requirement on (time-independent) potential by two orders and on nonlinearity by one order. The same analysis can be naturally generalized to relax the regularity requirements of optimal $H^1$-norm error bound for the first-order Lie-Trotter splitting and the optimal $L^2$- and $H^1$-norm error bounds for the second-order Strang splitting. Actually, the idea of substituting higher order spatial derivatives with lower order temporal derivatives can be traced back to Kato \cite{kato1987}, and has been extensively employed in the analysis of the NLSE from PDE perspectives \cite{cazenave2003,sinum2019}. More recently, this approach has also been adopted in the numerical analysis of the NLSE \cite{bao2023_EWI}. 

The rest of the paper is organized as follows. In Section 2, we present the time-splitting methods and state our main results. Sections 3 and 4 are devoted to the error estimates of the first-order Lie-Trotter time-splitting method and the second-order Strang time-splitting method, respectively. Numerical results are reported in Section 5 to confirm our error estimates. Finally, some conclusions are drawn in Section 6. Throughout the paper, we adopt standard Sobolev spaces as well as their corresponding norms, and denote by $ C $ a generic positive constant independent of the mesh size $ h $ and time step size $ \tau $, and by $ C(\alpha) $ a generic positive constant depending only on the parameter $ \alpha $. The notation $ A \lesssim B $ is used to represent that there exists a generic constant $ C>0 $, such that $ |A| \leq CB $. 

\section{Time-splitting methods and main results}
In this section, we introduce the first-order Lie-Trotter and second-order Strang time-splitting methods to solve the NLSE with low regularity potential and nonlinearity. We also state our main results here. For simplicity of the presentation and to avoid heavy notations, we only carry out the analysis in one dimension (1D) and take $ \Omega = (a, b) $. Generalizations to two dimension (2D) and three dimension (3D) are straightforward. In fact, the only dimension sensitive estimates are the Sobolev embedding and inverse inequalities. In our analysis, we only use the embeddings that hold for $d=1, 2, 3$, and we explicitly show the dependence of dimension in inverse inequalities. 

We define periodic Sobolev spaces as (see, e.g. \cite{bronsard2022}, for the definition in phase space)
\begin{equation*}
	H_\text{per}^m(\Omega) := \{\phi \in H^m(\Omega) : \phi^{(k)}(a) = \phi^{(k)}(b), \ k=0, \cdots, m-1\}, \quad m \geq 1, \quad m \in \mathbb{N}. 
\end{equation*}

The operator splitting techniques are based on a decomposition of the flow of \cref{NLSE}: 
\begin{equation*}
	\partial_t \psi = A(\psi) + B(\psi),
\end{equation*}
where $A (\psi) = i \Delta \psi$ and 
\begin{equation}\label{eq:B}
	B(\psi) = -i\phiu{\psi}\psi := -i (V + f(|\psi|^2)) \psi \quad \text{with} \quad \phiu{\phi} = V + f(|\phi|^2). 
\end{equation}
Then the NLSE \cref{NLSE} can be decomposed into two sub-problems. The first one is
\begin{equation}
	\left\{
	\begin{aligned}
		&\partial_t \psi(x, t) = A(\psi) = i \Delta \psi(x, t), \quad x \in \Omega, \quad t>0, \\
		&\psi(x, 0) = \psi_0(x), \quad x \in \overline{\Omega},
	\end{aligned}
	\right.
\end{equation}
which can be formally integrated exactly in time as
\begin{equation}\label{eq:linear_step}
	\psi(\cdot, t) = e^{i t \Delta} \psi_0(\cdot), \qquad t \geq 0.
\end{equation}
The second one is to solve
\begin{equation}
	\left\{
	\begin{aligned}
		&\partial_t \psi(x, t) = B(\psi) = -iV(\vx)\psi(x, t) -i f(|\psi(x, t)|^2)\psi(x, t), \quad  t>0, \\
		&\psi(x, 0) = \psi_0(x), \quad x \in \overline{\Omega},
	\end{aligned}
	\right.
\end{equation}
which, by noting
$|\psi(x, t)|=|\psi_0(x)|$ for $t \geq 0$, can be integrated exactly in time as
\begin{equation}\label{eq:nonl_step}
	\psi(x, t) = \Phi_B^t (\psi_0)(x) := \psi_0(x)e^{- i t (V(x) + f(|\psi_0(x)|^2))}, \quad x \in \overline{\Omega}, \quad t \geq 0. 
\end{equation}
Different combinations of the linear step \cref{eq:linear_step} and the nonlinear step \cref{eq:nonl_step} will yield different time-splitting schemes. 

\subsection{Lie-Trotter time-splitting Fourier spectral (LTFS) method}
Choose a time step size  $ \tau > 0 $, denote time steps as $ t_n = n \tau $ for $ n = 0, 1, ... $, and let $ \psi^{[n]}(\cdot) $ be the approximation of $ \psi(\cdot, t_n) $ for $ n \geq 0 $. Then a first-order semi-discretization of the NLSE \cref{NLSE} via the Lie-Trotter splitting is given as:
\begin{equation}\label{eq:LT}
	\psi^{[n+1]} = e^{i \tau \Delta} \Phi_B^\tau\left (\psi^{[n]} \right ), \quad n \geq 0, 
\end{equation}
with $ \psi^{[0]}(x) = \psi_0(x) $ for $x\in\overline{\Omega}$. 

Then we further discretize the semi-discrete scheme \cref{eq:LT} in space by the Fourier spectral method to obtain a fully discrete scheme. 

We remark that usually, the Fourier pseudospectral method is used for spatial discretization, which can be efficiently implemented with FFT. However, due to the low regularity of potential and/or nonlinearity, it is very hard to establish error estimates of the Fourier pseudospectral method, and it is impossible to obtain optimal error bounds in space as order reduction can be observed numerically \cite{bao2023_EWI}. 

Choose a mesh size $ h=(b-a)/N $ with $ N $ being a positive even integer and denote grid points as
\begin{equation*}
	x_j = a + jh, \quad j = 0, 1, \cdots, N.
\end{equation*}
Define the index sets
\begin{equation*}
	\mathcal{T}_N = \left\{-\frac{N}{2}, \cdots, \frac{N}{2}-1 \right\}, %\qquad \mathcal{T}_N^0 = \{0, 1, \cdots, N\},
\end{equation*}
and denote
\begin{equation}
	X_N = \text{span}\left\{e^{i \mu_l(x - a)}: l \in \mathcal{T}_N\right\}, \quad \mu_l = \frac{2 \pi l}{b-a}.  
\end{equation}
Let $ P_N:L^2(\Omega) \rightarrow X_N $ be the standard $ L^2 $-projection onto $ X_N $ as 
%and $ I_N: Y_N  \rightarrow X_N $ be the standard Fourier interpolation operator as
\begin{equation}
	(P_N u)(x) = \sum_{l \in \mathcal{T}_N} \widehat u_l e^{i \mu_l(x - a)}, 
	%\quad (I_N v)(x) = \sum_{l \in \mathcal{T}_N} \widetilde v_l e^{i \mu_l(x - a)}, \qquad x \in \overline{\Omega} = [a, b],
\end{equation}
where $ u \in L^2(\Omega) $, and $\widehat{u}_l$ are the Fourier coefficients of $u$ defined as 
\begin{equation}
	\widehat{u}_l = \frac{1}{b-a} \int_a^b u(x) e^{-i \mu_l (x-a)} \rmd x, \qquad l \in \Z. 
	%\quad \widetilde{v}_l = \frac{1}{N} \sum_{j=0}^{N-1} v_j e^{-i \mu_l (x_j - a)}, \qquad l \in \mathcal{T}_N.
\end{equation}
%Sometimes, we shall identify a continuous function $v \in C(\overline{\Omega})$ satisfying $v(a) = v(b)$ with a vector $v = (v_0, v_1, \cdots, v_N)^T \in Y_N$ with $v_j = v(x_j)$ for $j \in \mathcal{T}_N^0$. 

Let $ \psi^n(\cdot) $ be the numerical approximations of $ \psi(\cdot, t_n) $ for $ n \geq 0 $. Then the first-order Lie-Trotter time-splitting Fourier spectral method (\textit{LTFS}) reads
\begin{equation}\label{full_discrete_LT}
	\begin{aligned}
		&\psi^{(1)}(x) = e^{- i \tau (V(x) + f(|\psi^n(x)|^2))} \psi^n(x), \\
		&\psi^{n+1}(x) = \sum_{l \in \mathcal{T}_N} e^{- i \tau \mu_l^2} \widehat{(\psi^{(1)})}_l e^{i \mu_l(x-a)}, 
	\end{aligned}
	\quad x \in \Omega, \qquad n \geq 0, 
\end{equation}
where $ \psi^0 = P_N \psi_0 $ in \cref{full_discrete_LT}. 

Let $\mathcal{S}_1^\tau:X_N \rightarrow X_N$ be the numerical integrator associated with the LTFS method: 
\begin{equation}\label{eq:S1S2_S1}
	\mathcal{S}_1^\tau(\phi) := e^{i \tau \Delta} P_N \Phi_B^\tau(\phi), \qquad \phi \in X_N. 
\end{equation}
Then $\psi^n \ (n \geq 0)$ obtained from the LTFS \cref{full_discrete_LT} satisfy
\begin{equation}\label{LTFS}
	\begin{aligned}
		&\psi^{n+1} = \mathcal{S}_1^\tau(\psi^n) = e^{i \tau \Delta} P_N \Phi_B^\tau(\psi^n), \quad n \geq 0, \\
		&\psi^0 = P_N \psi_0. 
	\end{aligned}
\end{equation}

\subsection{Strang time-splitting Fourier spectral (STFS) method}
Similar to the first-order semi-discretization \cref{eq:LT}, we can obtain a second-order semi-discretization via the Strang splitting: 
\begin{equation}\label{eq:ST}
	\psi^{[n+1]} = e^{i \frac{\tau}{2} \Delta} \Phi_B^\tau \left (e^{i \frac{\tau}{2} \Delta} \psi^{[n]} \right ), \quad n \geq 0, 
\end{equation}
with $ \psi^{[0]}(x) = \psi_0(x) $ for $x\in\overline{\Omega}$. 

By using the Fourier spectral method to further discretize \cref{eq:ST}, we get the second-order Strang time-splitting Fourier spectral method (\textit{STFS}): 
\begin{equation}\label{full_discrete_ST}
	\begin{aligned}
		&\psi^{(1)}(x) = \sum_{l \in \mathcal{T}_N} e^{- i \frac{\tau}{2} \mu_l^2} \widehat{(\psi^{n})}_l e^{i \mu_l(x-a)}, \\
		&\psi^{(2)}(x) = e^{- i \tau (V(x) + f(|\psi^{(1)}(x)|^2))} \psi^{(1)}(x), \\
		&\psi^{n+1}(x) = \sum_{l \in \mathcal{T}_N} e^{- i \frac{\tau}{2} \mu_l^2} \widehat{(\psi^{(2)})}_l e^{i \mu_l(x-a)},  
	\end{aligned}
	\quad x \in \Omega, \qquad n \geq 0, 
\end{equation}
where $\psi^0 = P_N \psi_0$ in \cref{full_discrete_ST}. 

Introduce a numerical flow $\mathcal{S}_2^\tau:X_N \rightarrow X_N$ associated with the STFS scheme as 
\begin{equation}\label{eq:S1S2_S2}
	\mathcal{S}_2^\tau(\phi) := e^{i \frac{\tau}{2} \Delta} P_N \Phi_B^\tau \left (e^{i \frac{\tau}{2} \Delta}\phi \right ), \qquad \phi \in X_N. 
\end{equation}
Note that $\psi^n \ (n \geq 0)$ obtained from the STFS \cref{full_discrete_ST} satisfy
\begin{equation}\label{STFS}
	\begin{aligned}
		&\psi^{n+1} = \mathcal{S}_2^\tau(\psi^n) = e^{i \frac{\tau}{2} \Delta} P_N \Phi_B^\tau \left (e^{i \frac{\tau}{2} \Delta}\psi^n \right ), \quad n \geq 0, \\
		&\psi^0 = P_N \psi_0. 
	\end{aligned}
\end{equation}

\subsection{Main results}\label{sec:main_results}
For $ \psi^n \ (n \geq 0) $ obtained from the first-order LTFS method \cref{full_discrete_LT}, we have
\begin{theorem}\label{thm:LT}
	Under the assumptions that $ V \in L^\infty(\Omega) $, $ \sigma > 0 $ and the exact solution $ \psi \in C([0, T]; H^2_\text{\rm per}(\Omega)) \cap C^1([0, T]; L^2(\Omega)) $, there exists $ h_0 > 0 $ sufficiently small such that when $ 0 < h < h_0 $ and $ \tau < h^2/\pi $, we have
	\begin{equation}\label{eq:LT_L2}
		\| \psi(\cdot, t_n) - \psi^n  \|_{L^2} \lesssim \tau + h^2, \quad \| \psi(\cdot, t_n) - \psi^n  \|_{H^1} \lesssim \sqrt{\tau} + h, \quad 0 \leq n \leq T/\tau. 
	\end{equation}
	In addition, if $ V \in W^{1, 4}(\Omega) \cap H^1_\text{\rm per}(\Omega) $, $ \sigma \geq 1/2 $ and the solution $ \psi \in C([0, T]; H^3_\text{\rm per}(\Omega)) \cap C^1([0, T]; H^1(\Omega)) $, we have
	\begin{equation}\label{eq:LT_H1}
		\| \psi(\cdot, t_n) - \psi^n \|_{H^1} \lesssim \tau + h^2, \quad 0 \leq n \leq T/\tau. 
	\end{equation}
\end{theorem}

For $ \psi^n \ (n \geq 0) $ obtained from the second-order STFS method \cref{full_discrete_ST}, we have
\begin{theorem}\label{thm:ST}
	Under the assumptions that $ V \in H_\text{\rm per}^2(\Omega) $, $ \sigma \geq 1 $ and the exact solution $ \psi \in C([0, T]; H^4_\text{\rm per}(\Omega)) \cap C^1([0, T]; H^2(\Omega)) \cap C^2([0, T]; L^2(\Omega)) $, there exists $ h_0 > 0 $ sufficiently small such that when $ 0 < h < h_0 $ and $ \tau < h^2/\pi $, we have
	\begin{equation}\label{eq:ST_L2}
		\| \psi(\cdot, t_n) - \psi^n \|_{L^2} \lesssim \tau^2 + h^4, \quad \| \psi(\cdot, t_n) - \psi^n \|_{H^1} \lesssim \tau^\frac{3}{2} + h^3, \quad 0 \leq n \leq T/\tau. 
	\end{equation}
	In addition, if $ V \in H^3_\text{\rm per}(\Omega) $, $ \sigma \geq \frac{3}{2} $ (or $\sigma = 1$) and the solution $ \psi \in C([0, T]; H^5_\text{\rm per}(\Omega)) \cap C^1([0, T]; H^3(\Omega)) \cap C^2([0, T]; H^1(\Omega)) $, we have
	\begin{equation}\label{eq:ST_H1}
		\| \psi(\cdot, t_n) - \psi^n \|_{H^1} \lesssim \tau^2 + h^4, \quad 0 \leq n \leq T/\tau. 
	\end{equation}
\end{theorem}

\begin{remark}\label{rem:regularity}
	Our regularity assumptions on the exact solution $\psi$ are compatible with the assumptions on potential and nonlinearity. By Corollary 4.8.6 of \cite{cazenave2003} (see also \cite{kato1987,bao2023_semi_smooth,bao2023_EWI}), the $H^2$-regularity can be propagated by the NLSE \cref{NLSE} with $ V \in L^\infty(\Omega) $ and $\sigma > 0$ (corresponding to the case of \cref{eq:LT_L2}), i.e. it can be expected that $\psi \in C([0, T]; H^2_\text{per}(\Omega)) \cap C^1([0, T]; L^2(\Omega))$ if $\psi_0 \in H^2_\text{per}(\Omega)$. The assumptions for \cref{eq:LT_H1,eq:ST_L2,eq:ST_H1} are compatible with the assumption for \cref{eq:LT_L2} in the sense that the increment of the Sobolev exponent of $\psi$ is the same as the increment of the differentiability order of potential and nonlinearity. 
\end{remark}

\begin{remark}
	According to our error estimates in \cref{thm:LT,thm:ST}, the time step size restriction $\tau \lesssim h^2$ is natural in terms of the balance of spatial errors and temporal errors. Moreover, we can clearly observe this time step size restriction in the numerical results in \cref{sec:numerical experiments}, indicating that it is necessary and optimal. 
\end{remark}

\begin{remark}\label{rem:H1_bound}
	The $H^1$-norm error bounds in \cref{eq:LT_L2,eq:ST_L2} follow directly from the corresponding $L^2$-norm error bounds with $ \tau \lesssim h^2 $, standard projection error estimates of $P_N$, and the inverse estimate $\| \phi \|_{H^1} \lesssim h^{-1} \| \phi \|_{L^2}$ for all $\phi \in X_N$ \cite{book_spectral}. 
\end{remark}

\section{Proof of \cref{thm:LT} for the LTFS \cref{full_discrete_LT}}
In this section, we shall show the optimal error bounds for the LTFS method \cref{full_discrete_LT}. We start with the optimal $L^2$-norm error bound and present the proof of \cref{eq:LT_L2} in \cref{thm:LT}. In the rest of this section, we assume that $ V \in L^\infty(\Omega) $, $ \sigma > 0 $ and $\psi \in C([0, T]; H_\text{per}^2(\Omega)) \cap C^1([0, T]; L^2(\Omega))$, and define a constant 
\begin{equation*}
	M_2 = \max\left \{ \| V \|_{L^\infty}, \| \psi \|_{L^\infty([0, T]; H^2(\Omega))}, \| \partial_t \psi \|_{L^\infty([0, T]; L^2(\Omega))}, \| \psi \|_{L^\infty([0, T]; L^\infty(\Omega))} \right \}. 
\end{equation*}
%	According to the known results \cite{kato1987,cazenave2003}, we make the following assumptions on the NLSE ??: for the potential and the nonlinearity, we assume that
%	\begin{equation}\label{A1}
	%		V \in L^\infty(\Omega), \quad \sigma > 0, \tag{A1} 
	%	\end{equation}
%	and for the exact solution, we assume that, for some $ 0 < T < T_\text{max} $ with $ T_\text{max} $ being the maximal existing time,
%	\begin{equation}\label{B1}
	%		\psi \in C([0, T]; H_\text{per}^2(\Omega)) \cap C^1([0, T]; L^2(\Omega)). \tag{B1}
	%	\end{equation}
\subsection{Some estimates for the operator $ B $}
For the operator $B$ defined in \cref{eq:B}, we have
%	\begin{lemma}
	%		Under the assumptions $ V \in L^\infty(\Omega) $ and $ \sigma > 0 $, for any $ v \in L^2(\Omega) $ satisfying $ \| v \|_{L^\infty} \leq M $, we have
	%		\begin{equation*}
		%			\| B(v) \|_{L^2} \leq C(\| V \|_{L^\infty}, M) \| v \|_{L^2}. 
		%		\end{equation*}
	%	\end{lemma}

\begin{lemma}\label{lem:B}
	Under the assumptions $ V \in L^\infty(\Omega) $ and $ \sigma > 0 $, for any $ v, w \in L^2(\Omega) $ satisfying $ \| v \|_{L^\infty} \leq M $ and $ \| w \|_{L^\infty} \leq M $, we have
	\begin{equation}\label{lem:diff_B1}
		\| B(v) - B(w) \|_{L^2} \leq C(\| V \|_{L^\infty}, M) \| v - w \|_{L^2}. 
	\end{equation}
	In particular, when $w = 0$, we have
	\begin{equation}\label{lem:B1}
		\| B(v) \|_{L^2} \leq C(\| V \|_{L^\infty}, M) \| v \|_{L^2}. 
	\end{equation}
\end{lemma}
Let $ dB(\cdot)[\cdot] $ be the G\^ateaux derivative defined as
\begin{equation}\label{eq:Gateaux}
	dB(v)[w]:= \lim_{\varepsilon \rightarrow 0} \frac{B(v + \varepsilon w) - B(v)}{\vep}, 
\end{equation}
where the limit is taken for real $\vep$ (see also \cite{bao2023_semi_smooth}). 
Introduce a continuous function $G:\C \rightarrow \C$ as 
\begin{equation}\label{eq:G_def}
	G(z) = \left\{
	\begin{aligned}
		& f^\prime(|z|^2) z^2=\beta \sigma |z|^{2\sigma-2} z^2, &z \neq 0, \\
		&0, &z=0,
	\end{aligned}
	\right. \qquad z \in \C.
\end{equation}
Plugging the expression of $B$ \cref{eq:B} into \cref{eq:Gateaux}, we obtain
\begin{align}\label{eq:dB_def}
	dB(v)[w]=-i \left[ Vw + (1+\sigma)f(|v|^2) w + G(v) \overline{w} \right], 
\end{align}
where we use $f^\prime(|z|^2)|z|^2 = \sigma f(|z|^2) $ for all $ z \in \C $ and define $G(v)(x) := G(v(x))$. 
%	define $ G: L^\infty(\Omega) \rightarrow L^\infty(\Omega) $ as
%	\begin{equation}\label{eq:G_def}
	%		G(v)(x) = \left\{\begin{aligned}
		%			&f^\prime(|v(x)|^2)[v(x)]^2, & v(x) \neq 0, \\
		%			&0, &v(x)=0,  
		%		\end{aligned}
	%		\qquad x \in \Omega. 
	%		\right. 
	%	\end{equation}
Then we have
\begin{lemma}\label{lem:dB1}
	Under the assumptions $ V \in L^\infty(\Omega) $ and $ \sigma > 0 $, for any $ v, w \in L^2(\Omega) $ satisfying $ \| v \|_{L^\infty} \leq M $, we have
	\begin{equation*}
		\| dB(v)[w] \|_{L^2} \leq C(\| V \|_{L^\infty}, M) \| w \|_{L^2}. 
	\end{equation*}
\end{lemma}
The proofs of \cref{lem:B,lem:dB1} can be found in \cite{bao2023_semi_smooth} (Lemmas 3.2 and 3.3) and we shall omit them for brevity. From the definition of the nonlinear flow $\Phi_B^\tau$ in \cref{eq:nonl_step}, we immediately have
\begin{lemma}\label{lem:Phi_B^tau1}
	Under the assumptions $ V \in L^\infty(\Omega) $ and $ \sigma > 0 $, for any $ v \in L^\infty(\Omega) $ and $ w \in L^\infty(\Omega) $, we have
	\begin{equation*}
		\| \Phi_B^\tau(v) \|_{L^2} = \| v \|_{L^2}, \quad \| \Phi_B^\tau(w) \|_{L^\infty} = \| w \|_{L^\infty}. 
	\end{equation*}
\end{lemma}

\subsection{Local truncation error and stability estimates}
We shall establish the local truncation error and stability estimates for the first-order Lie-Trotter splitting. In the rest of this paper, we always abbreviate $\psi(\cdot, t)$ by $\psi(t)$ for simplicity of notations when there is no confusion. 

Define the local truncation error of the Lie-Trotter time-splitting method as
\begin{equation}\label{eq:E^n_def}
	\mathcal{E}^n = P_N \psi(t_{n+1}) - \mathcal{S}_1^\tau(P_N\psi(t_n)), \quad 0 \leq n \leq T/\tau-1. 
\end{equation}
Then the local truncation error can be decomposed into two parts based on different regularity requirements on potential and nonlinearity. 
\begin{proposition}\label{prop:dominant_error_LT_L2}
	Assuming that $ V \in L^\infty(\Omega) $, $ \sigma > 0 $ and $\psi \in C([0, T]; H_\text{per}^2(\Omega)) \cap C^1([0, T]; L^2(\Omega))$, we have
	\begin{equation*}
		\mathcal{E}^n = \mathcal{E}_1^n + \mathcal{E}_2^n, 
	\end{equation*}
	where
	\begin{equation*}
		\| \mathcal{E}_1^n \|_{L^2} \lesssim \tau^2 + \tau h^2, \quad \mathcal{E}_2^n = -e^{i\tau\Delta}\int_0^\tau (I-e^{-is\Delta}) P_N B(P_N \psi(t_n)) \rmd s.
	\end{equation*}
\end{proposition}

\begin{proof}
	The proof proceeds similarly to the proof of Proposition 3.6 in \cite{bao2023_semi_smooth} and we only sketch it here for the convenience of the reader. By Duhamel's formula, one has
	\begin{align}\label{eq:duhamel}
		&P_N\psi(t_{n+1}) 
		= e^{i \tau \Delta} P_N \psi(t_n) + \int_0^\tau e^{i (\tau - s )\Delta} P_N B(e^{is \Delta} \psi(t_n)) \rmd s \notag\\
		&\quad + \int_0^\tau \int_0^s e^{i (\tau - s) \Delta} P_N\left( dB(e^{i(s-\vsigma)\Delta}\psi(t_n+\vsigma)) [e^{i(s - \vsigma)\Delta} B(\psi(t_n+\vsigma))] \right) \rmd \vsigma \rmd s. 
	\end{align}
	%		Applying the first order Taylor expansion
	%		\begin{equation*}
		%			e^{\tau w} = 1 + \tau w + \tau^2 \int_0^1 (1-\theta) w^2e^{\theta \tau w} \rmd \theta
		%		\end{equation*}
	%		to $e^{- i \tau \phiu{P_N \psi(t_n)}} $, we have
	%		\begin{align}\label{eq:S1}
		%			\mathcal{S}_1^\tau(P_N\psi(t_n)) 
		%			&= e^{i \tau \Delta} P_N \psi(t_n) + \tau e^{i \tau \Delta} P_N B(P_N \psi(t_n)) \notag\\
		%			&\quad + \tau^2 \int_0^1 (1-\theta) e^{i \tau \Delta} P_N \left( B(P_N \psi(t_n)) \phiu{P_N \psi(t_n)} e^{-i \theta \tau \phiu{P_N \psi(t_n)}}  \right) \rmd \theta.  
		%		\end{align}
	Applying the first-order Taylor expansion
	\begin{equation}\label{eq:taylor}
		\Phi_B^\tau(w) = w + \tau B(w) + \tau^2 \int_0^1 (1-\theta) dB(\Phi_B^{\theta \tau}(w)) [B(\Phi_B^{\theta \tau}(w))] \rmd \theta
	\end{equation}
	to $\mathcal{S}_1^\tau(P_N\psi(t_n))$, we have
	\begin{align}\label{eq:S1}
		&\mathcal{S}_1^\tau(P_N\psi(t_n)) 
		= e^{i \tau \Delta} P_N \psi(t_n) + \tau e^{i \tau \Delta} P_N B(P_N \psi(t_n)) \notag\\
		&\quad + \tau^2 \int_0^1 (1-\theta) e^{i \tau \Delta} P_N \left( dB(\Phi_B^{\theta\tau}(P_N \psi(t_n))) [B(\Phi_B^{\theta\tau}(P_N \psi(t_n)))]  \right) \rmd \theta.  
	\end{align}
	%		By the Taylor expansion of $ \Phi_B^\tau $, we have
	%		\begin{equation}\label{eq:S1}
		%			\begin{aligned}
			%				\mathcal{S}_1^\tau(P_N\psi(t_n)) 
			%				&= e^{i \tau \Delta} P_N \psi(t_n) + \tau e^{i \tau \Delta} P_N B(P_N \psi(t_n)) \\
			%				&\quad + \tau^2 \int_0^1 (1-\theta) e^{i \tau \Delta} P_N \left( dB(\Phi_B^{\theta \tau}(P_N \psi(t_n))) [B(\Phi_B^{\theta \tau}(P_N \psi(t_n)))] \right) \rmd \theta.  
			%			\end{aligned}
		%		\end{equation}
	Subtracting \cref{eq:S1} from \cref{eq:duhamel}, recalling \cref{eq:E^n_def}, one obtain
	\begin{equation}\label{eq:error_decomp}
		\mathcal{E}^n = P_N\psi(t_{n+1}) - \mathcal{S}_1^\tau(P_N\psi(t_n)) = e_1 + e_2 + e_3, 
	\end{equation}
	where
	\begin{equation}\label{e3}
		\begin{aligned}
			&e_1 = \int_0^\tau \int_0^s e^{i (\tau - s) \Delta} P_N\left( dB(e^{i(s-\vsigma)\Delta}\psi(t_n+\vsigma)) [e^{i(s - \vsigma)\Delta} B(\psi(t_n+\vsigma))] \right) \rmd \vsigma \rmd s, \\
			&e_2 = -\tau^2 \int_0^1 (1-\theta) e^{i \tau \Delta} P_N \left( dB(\Phi_B^{\theta\tau}(P_N \psi(t_n))) [B(\Phi_B^{\theta\tau}(P_N \psi(t_n)))]  \right) \rmd \theta, \\
			&e_3 = \int_0^\tau e^{i (\tau - s )\Delta} P_N B(e^{is \Delta} \psi(t_n)) \rmd s - \tau e^{i \tau \Delta} P_N B(P_N \psi(t_n)). 
		\end{aligned}
	\end{equation}
	By the boundedness of $e^{i t \Delta}$ and $P_N$, and using \cref{lem:B,lem:dB1,lem:Phi_B^tau1} and Sobolev embedding, we have
	\begin{equation}
		\| e_1 \|_{L^2} \lesssim \tau^2, \quad \| e_2 \|_{L^2} \lesssim \tau^2. 
	\end{equation}
	Then we shall estimate $ e_3 $. From $e_3$ in \cref{e3}, we have
	\begin{align}\label{eq:decomp_e3}
		e_3 
		&= \int_0^\tau \left[ e^{i (\tau - s )\Delta} P_N B(e^{is \Delta} \psi(t_n)) - e^{i \tau \Delta} P_N B(P_N \psi(t_n)) \right] \rmd s \notag\\
		&= \int_0^\tau e^{i (\tau - s )\Delta} P_N (B(e^{is \Delta} \psi(t_n)) - B(\psi(t_n))) \rmd s \notag \\
		&\quad + \int_0^\tau e^{i (\tau - s )\Delta} P_N (B(\psi(t_n)) - B(P_N \psi(t_n))) \rmd s \notag\\
		&\quad - e^{i \tau \Delta} \int_0^\tau (I - e^{- i s \Delta}) P_N B(P_N \psi(t_n)) \rmd s =: e_3^1 + e_3^2 + e_3^3. 
	\end{align}
	By \cref{lem:diff_B1}, the boundedness of $ P_N $ and $ e^{i t \Delta} $, standard Fourier projection error estimates of $P_N$ \cite{book_spectral} and 
	$\| (I - e^{i t \Delta}) \phi \|_{L^2} \lesssim t \| \phi \|_{H^2}$ for all $ \phi \in H^2_\text{per}(\Omega) $ \cite{BBD,bao2022}, we have
	\begin{equation*}
		\| e_3^1 \|_{L^2} \lesssim \tau^2, \quad \| e_3^2 \|_{L^2} \lesssim \tau h^2. 
	\end{equation*}
	The conclusion follows by letting $ \mathcal{E}_1^n = e_1 + e_2 + e_3^1 + e_3^2 $ and $ \mathcal{E}_2^n = e_3^3 $. 
\end{proof}

For the nonlinear flow, we have the following $ L^\infty $-conditional $ L^2 $-stability estimate. 
\begin{proposition}\label{prop:stability_L2}
	Let $ v, w \in X_N $ such that $ \| v \|_{L^\infty} \leq M $, $ \| w \|_{L^\infty} \leq M $. Then 
	\begin{equation*}
		\| P_N (\Phi_B^\tau(v) - v) - P_N (\Phi_B^\tau(w) - w) \|_{L^2} \leq C(\| V \|_{L^\infty}, M) \tau \| v - w \|_{L^2}. 
	\end{equation*}
\end{proposition}

\begin{proof}
	Let $ u^\theta = (1-\theta)w + \theta v $ for $ 0 \leq \theta \leq 1 $. Then we have, by recalling $\phiu{u} = V + f(|u|^2)$, 
	\begin{align}\label{eq:diff_Phi_B}
		&(\Phi_B^\tau(v) - v) - (\Phi_B^\tau(w) - w)
		= \int_0^1 \partial_\theta \left ( \Phi_B^\tau(u^\theta) - u^\theta \right ) \rmd \theta \notag\\
		&= \int_0^1 (v-w) \left ( e^{-i\tau \phiu{u^\theta}} - 1 \right ) \rmd \theta \notag \\
		&\quad - i \tau \int_0^1 \left ( \sigma f(|u^\theta|^2) (v-w) + G(u^\theta) \overline{(v - w)} \right ) e^{-i\tau \phiu{u^\theta}} \rmd \theta. 
	\end{align} 
	From \cref{eq:diff_Phi_B}, noting that $| e^{i \theta} - 1| \leq \theta$ for all $\theta \in \R$, $ \| \phiu{u^\theta} \|_{L^\infty} \lesssim \| V \|_{L^\infty} + M^{2\sigma}$ and
	\begin{equation}\label{eq:f_G_bound}
		\left |f(|z|^2) \right | + \left |G(z) \right | \lesssim |z|^{2\sigma}, \quad z \in \C, \quad \sigma > 0, 
	\end{equation}
	we have
	\begin{equation}\label{stability_L2}
		\| \Phi_B^\tau(v) - v) - (\Phi_B^\tau(w) - w) \|_{L^2} \leq C(\| V \|_{L^\infty}, M) \tau \| v - w \|_{L^2} + C(M) \tau \| v - w \|_{L^2}. 
	\end{equation}
	The conclusion then follows from \cref{stability_L2} and the boundedness of $ P_N $ immediately. 
	%		\begin{equation}\label{eq:Phi_Bu-u}
		%			\begin{aligned}
			%				\Phi_B^\tau(u) - u = u \left ( e^{-i\tau (V + f(|u|^2))} - 1 \right )  = -i \tau u (V + f(|u|^2)) \int_0^1 e^{- i \theta \tau (V + f(|u|^2))} \rmd \theta. 
			%			\end{aligned}
		%		\end{equation}
	%		Taking $ u = v $ and $ u = w $ in \cref{eq:Phi_Bu-u} and subtracting one from the other, one has
	%		\begin{align*}
		%			&\| \Phi_B^\tau(v) - v) - (\Phi_B^\tau(w) - w) \|_{L^2} \\
		%			&\leq \tau \sup_{0 \leq \theta \leq 1} \left \|  v (V + f(|v|^2)) e^{- i \theta \tau (V + f(|v|^2))} - w (V + f(|w|^2)) e^{- i \theta \tau (V + f(|w|^2))} \right \|_{L^2} \\
		%			&= \tau \sup_{0 \leq \theta \leq 1} \left \|  v (V + f(|v|^2)) e^{- i \theta \tau f(|v|^2)} - w (V + f(|w|^2)) e^{- i \theta \tau f(|w|^2)} \right \|_{L^2} \\
		%			&\lesssim \tau 
		%		\end{align*} 
	%		By ?? and \cref{eq:nonl_step}, one has
	%		\begin{equation}
		%			\begin{aligned}
			%				&\| P_N (\Phi_B^\tau(v) - v) - P_N (\Phi_B^\tau(w) - w) \|_{L^2} \\
			%				&\leq \| (\Phi_B^\tau(v) - v) - (\Phi_B^\tau(w) - w) \|_{L^2} \\
			%				&\leq \| P_N (\Phi_B^\tau(v) - v) - P_N (\Phi_B^\tau(w) - w) \|_{L^2} 
			%			\end{aligned}
		%		\end{equation}		
\end{proof}

%	Then we show a discrete counterpart of Lemma 3.5 in \cite{bao2023_EWI}. 
%	\begin{lemma}
	%		Assume that $w \in C^1([0, n\tau]; L^2()\Omega)$. Let $ \tau N^2 \lesssim 1$ and
	%		\begin{equation}
		%			v = \tau \sum_{k=0}^{n-1} e^{i (n\tau - k \tau) \Delta} P_N \phi(k \tau). 
		%		\end{equation}
	%		Then we have
	%		\begin{equation*}
		%			\| v \|_{H^2} \lesssim \tau. 
		%		\end{equation*}
	%	\end{lemma}
%	
%	\begin{proof}
	%		content...
	%	\end{proof}

\begin{remark}
	We cannot expect the constant $C$ in \cref{prop:stability_L2} to depend exclusively on the minimal of $ \| v \|_{L^\infty} $ and $ \| w \|_{L^\infty} $ as is the case of Proposition 3.8 in \cite{bao2023_semi_smooth} since 
	\begin{equation*}
		\lim_{\tau \rightarrow 0} \left ( \frac{\Phi_B^\tau(v)(x) - v(x)}{\tau} - \frac{\Phi_B^\tau(w)(x) - w(x)}{\tau} \right) = B(v)(x) - B(w)(x), \quad x \in \Omega, 
	\end{equation*} 
	and the constant in \cref{lem:diff_B1} is already optimal. This also prevents us from generalizing the results in this paper to the LogSE considered in \cite{bao2019,bao2022}. 
\end{remark}

\subsection{Optimal $ L^2 $-norm error bound}\label{sec:proof_S1_L2}
We shall establish the optimal $L^2$-norm error bound \cref{eq:LT_L2} for the LTFS method \cref{LTFS}. As mentioned before, the main idea is to replace the Laplacian $\Delta$ with a temporal derivative in the dominant local error $\mathcal{E}^n_2$. 
%The techniques used in the proof are inspired by the Regularity Compensation Oscillation (RCO) technique introduced in \cite{RCO_SE,RCO_KG,RCO_Dirac}. 
\begin{proof}[Proof of \cref{eq:LT_L2} in \cref{thm:LT}]
	Let $ e^n = P_N \psi(t_n) - \psi^n $ for $ 0 \leq n \leq T/\tau $. We start with the $L^2$-error bound. By standard projection error estimates of $P_N$, recalling that $ \psi \in C([0, T]; H^2_\text{\rm per}(\Omega)) $, we have
	\begin{equation}\label{eq:reduce}
		\| \psi(t_n) - P_N \psi(t_n) \|_{L^2} \lesssim h^2, \quad 0 \leq n \leq T/\tau.  
	\end{equation}
	Then the proof reduces to the estimate of $ e^n $. For $  0 \leq n \leq T/\tau -1 $, we have
	\begin{align}\label{eq:error_eq_improv}
		e^{n+1} 
		&= P_N \psi(t_{n+1}) - \psi^{n+1} \notag\\
		&= P_N \psi(t_{n+1}) - \mathcal{S}_1^\tau(P_N\psi(t_n)) + \mathcal{S}_1^\tau(P_N\psi(t_n)) - \mathcal{S}_1^\tau(\psi^n) \notag \\
		&= e^{i \tau \Delta} e^n + Q^n + \mathcal{E}^n, 
	\end{align}
	where $\mathcal{E}^n$ is defined in \cref{eq:E^n_def} and
	\begin{align}
		Q^n &= e^{i \tau \Delta} P_N \left( (\Phi_B^\tau(P_N \psi(t_n)) - P_N \psi(t_n)) - (\Phi_B^\tau(\psi^n) - \psi^n)  \right). \label{eq:Qn}
%		\mathcal{E}^n &= P_N \psi(t_{n+1}) - \mathcal{S}_1^\tau(P_N\psi(t_n)). \label{eq:En}
	\end{align}
	Iterating \cref{eq:error_eq_improv}, we have
	\begin{equation}\label{eq:error_eq_iter}
		e^{n+1} = e^{i(n+1)\tau\Delta} e^0 + \sum_{k=0}^{n} e^{i (n-k)\tau \Delta} \left( Q^k + \mathcal{E}^k \right), \quad 0 \leq n \leq T/\tau-1. 
	\end{equation}
	For $ Q^n $ in \cref{eq:Qn}, applying \cref{prop:stability_L2} and the isometry property of $e^{it\Delta}$, we have
	\begin{equation*}
		\| Q^n \|_{L^2} \leq C(\| \psi^n \|_{L^\infty}, M_2) \tau \| e^n \|_{L^2}, \quad 0 \leq n \leq T/\tau - 1, 
	\end{equation*}
	which, together with the isometry property of $e^{i t \Delta}$ and the triangle inequality, implies
	\begin{equation}\label{eq:est_Zn}
		\left \| \sum_{k=0}^{n} e^{i (n-k)\tau \Delta} Q^k \right \|_{L^2} \leq C_1 \tau \sum_{k=0}^{n} \| e^k \|_{L^2}, \quad 0 \leq n \leq T/\tau -1, 
	\end{equation}
	where the constant $C_1$ depends on $\max_{0 \leq k \leq n} \| \psi^k \|_{L^\infty}$ and $M_2$. 
	By \cref{prop:dominant_error_LT_L2} and the isometry property of $e^{i t \Delta}$, we have
	\begin{align}\label{eq:est_En}
		\left\| \sum_{k=0}^{n} e^{i (n-k)\tau \Delta} \mathcal{E}^k \right\|_{L^2} 
		&\lesssim n \tau (\tau + h^2) + \left\| \sum_{k=0}^{n} e^{i (n-k)\tau \Delta} \mathcal{E}_2^k \right\|_{L^2} \notag\\
		& = n \tau (\tau + h^2) + \left\| \sum_{k=0}^{n} e^{-ik\tau \Delta} \int_0^\tau (I-e^{-is\Delta}) \rmd s P_N B(P_N \psi(t_k)) \right\|_{L^2} \notag\\
		& = n \tau (\tau + h^2) + \left\| \mathcal{J}^n \right\|_{L^2}, \quad \quad  0 \leq n \leq T/\tau - 1, 
	\end{align}
	where
	\begin{equation}\label{eq:Jn}
		\mathcal{J}^n = \sum_{k=0}^{n} e^{-ik\tau \Delta} \int_0^\tau (I-e^{-is\Delta}) \rmd s P_N B(P_N \psi(t_k)), \quad  0 \leq n \leq T/\tau - 1. 
	\end{equation}
	From \cref{eq:error_eq_iter}, using \cref{eq:est_Zn,eq:est_En}, we have
	\begin{equation}\label{eq:error_eq_LT_L2}
		\| e^{n+1} \|_{L^2} \lesssim \tau + h^2 + C_1\tau \sum_{k=0}^n \| e^k \|_{L^2} + \left\| \mathcal{J}^n \right\|_{L^2}, \quad 0 \leq n \leq T/\tau -1. 
	\end{equation}
	
	We shall estimate $ \mathcal{J}^n $ carefully to capture the phase cancellation. Set
	\begin{equation}\label{eq:phin}
		\phi^n = P_N B(P_N \psi(t_n)) \in X_N, \quad 0 \leq n \leq T/\tau - 1, 
	\end{equation}
	and define
	\begin{equation}\label{eq:r_S_def}
		\delta_l = \int_0^\tau (1-e^{is\mu_l^2}) \rmd s, \quad S_{n, l} = \sum_{k=0}^n e^{ i k \tau \mu_l^2}, \qquad l \in \mathcal{T}_N, \quad 0 \leq n \leq T/\tau -1. 
	\end{equation}
	For $\delta_l $ defined in \cref{eq:r_S_def}, we have
	\begin{equation}\label{eq:rl}
		|\delta_l| \leq \int_0^\tau |1-e^{is\mu_l^2}| \rmd s \leq \int_0^\tau 2 |\sin(s \mu_l^2 /2)| \rmd s \leq \int_0^\tau s \mu_l^2 \rmd s = \frac{\tau^2 \mu_l^2}{2}, \quad l \in \mathcal{T}_N. 
	\end{equation}
	When $ \tau < h^2 / \pi $, using the fact that $|\sin(x)| \geq 2|x|/\pi$ when $ x \in [0, \pi/2]$, we have
	\begin{equation}\label{eq:Snl}
		|S_{n, l}| = \frac{|1-e^{i (n+1) \tau \mu_l^2}|}{|1-e^{i \tau \mu_l^2}|} \leq \frac{2}{2\sin(\tau \mu_l^2 /2)} \lesssim \frac{1}{\tau \mu_l^2}, \quad 0 \neq l \in \mathcal{T}_N. 
	\end{equation}
	Combining \cref{eq:rl,eq:Snl} and noting that $\delta_l = 0$ when $l=0$, we have
	\begin{equation}\label{eq:rS}
		| \delta_l S_{n, l} | \lesssim \tau, \qquad l \in \mathcal{T}_N, \quad 0 \leq n \leq T/\tau-1, 
	\end{equation}
	where the constant is independent of $n$ and $l$. Inserting \cref{eq:phin} into \cref{eq:Jn} and recalling \cref{eq:r_S_def}, we have, for $0 \leq n \leq T/\tau-1$, 
	\begin{align}\label{eq:Jn_phase}
		\mathcal{J}^n 
		&= \sum_{k=0}^n \sum_{l \in \mathcal{T}_N} e^{ i k \tau \mu_l^2} \int_0^\tau (1-e^{is\mu_l^2}) \rmd s \widehat{\phi^k_l} e^{i\mu_l(x-a)} \notag \\
		&= \sum_{k=0}^n \sum_{l \in \mathcal{T}_N} e^{ i k \tau \mu_l^2} \delta_l \widehat{\phi^k_l} e^{i\mu_l(x-a)}. 
	\end{align}
	From \cref{eq:Jn_phase}, exchanging the order of summation, using summation by parts and recalling \cref{eq:r_S_def}, we obtain, for $0 \leq n \leq T/\tau-1$,  
	\begin{align}\label{eq:summation_by_parts}
		\mathcal{J}^n
		&= \sum_{l \in \mathcal{T}_N} \delta_l e^{i\mu_l(x-a)} \sum_{k=0}^n e^{i k \tau \mu_l^2}  \widehat{\phi^k_l} \notag \\
		&= \sum_{l \in \mathcal{T}_N} \delta_l e^{i\mu_l(x-a)} \left( S_{n, l}  \widehat{\phi^n_l} - \sum_{k=0}^{n-1} S_{k, l} \left(\widehat{\phi^{k+1}_l} - \widehat{\phi^k_l} \right) \right).  
	\end{align}
	From \cref{eq:summation_by_parts}, using Parseval's identity, Cauchy inequality and \cref{eq:rS}, we have, for $0 \leq n \leq T/\tau-1$, 
	\begin{align}\label{eq:est_J_LT_L2}
		\| \mathcal{J}^n \|_{L^2}^2 
		&=(b-a) \sum_{l \in \mathcal{T}_N} \delta_l^2 \left| S_{n, l}  \widehat{\phi^n_l} - \sum_{k=0}^{n-1} S_{k, l} \left(\widehat{\phi^{k+1}_l} - \widehat{\phi^k_l} \right) \right|^2 \notag\\
		&\lesssim \sum_{l \in \mathcal{T}_N} \delta_l^2 S_{n, l}^2 \left|\widehat{\phi^n_l} \right|^2 + \sum_{l \in \mathcal{T}_N} \delta_l^2 \left|\sum_{k=0}^{n-1} S_{k, l} \left(\widehat{\phi^{k+1}_l} - \widehat{\phi^k_l} \right) \right|^2 \notag\\
		&\lesssim \tau^2 \sum_{l \in \mathcal{T}_N} \left|\widehat{\phi^n_l} \right|^2 + \sum_{l \in \mathcal{T}_N} \delta_l^2 \sum_{k_1=0}^{n-1} \left|S_{k_1, l}\right|^2 \sum_{k_2=0}^{n-1} \left|\widehat{\phi^{k_2+1}_l} - \widehat{\phi^{k_2}_l}\right|^2 \notag\\
		&\lesssim \tau^2 \| \phi^n \|_{L^2}^2 + \sum_{k_1=0}^{n-1} \sum_{k_2=0}^{n-1} \sum_{l \in \mathcal{T}_N} \delta_l^2 \left|S_{k_1, l}\right|^2 \left|\widehat{\phi^{k_2+1}_l} - \widehat{\phi^{k_2}_l} \right|^2 \notag\\
		&\lesssim \tau^2 \| \phi^n \|_{L^2}^2 + n\tau^2 \sum_{k=0}^{n-1} \| \phi^{k+1} - \phi^{k} \|_{L^2}^2. 
	\end{align}
	Recalling \cref{eq:phin}, using the $L^2$-projection property of $ P_N$ and \cref{lem:B}, we have
	\begin{equation*}
		\| \phi^n \|_{L^2} \leq \| B(P_N \psi(t_n)) \|_{L^2} \leq C(M_2)\| \psi(t_n) \|_{L^2} \lesssim C(M_2), \quad 0 \leq n \leq T/\tau-1,
	\end{equation*}
	and, for $0 \leq k \leq T/\tau-2$, 
	\begin{align*}
		\| \phi^{k+1} - \phi^{k} \|_{L^2} 
		&= \| P_N B(P_N \psi(t_{k+1})) - P_N B(P_N \psi(t_k)) \|_{L^2} \\
		&\leq \| B(P_N \psi(t_{k+1})) - B(P_N \psi(t_k)) \|_{L^2} \\
		&\lesssim \| \psi(t_{k+1}) - \psi(t_{k}) \|_{L^2} \lesssim \tau \| \partial_t \psi \|_{L^\infty([t_k, t_{k+1}]; L^2)}, 
	\end{align*}
	which, inserted into \cref{eq:est_J_LT_L2}, yield
	\begin{equation}\label{eq:est_J_final}
		\| \mathcal{J}^n \|^2_{L^2} \lesssim \tau^2 + n^2 \tau^4 \lesssim \tau^2.   
	\end{equation}
	Inserting \cref{eq:est_J_final} into \cref{eq:error_eq_LT_L2}, we have
	\begin{equation}\label{eq:error_final_L2}
		\| e^{n+1} \|_{L^2} \lesssim \tau + h^2 + C_1 \tau \sum_{k=0}^{n} \| e^k \|_{L^2}, 
	\end{equation}
	where $ C_1 $ depends on $M_2$ and $ \max_{0 \leq k \leq n} \| \psi^k \|_{L^\infty} $, and can be controlled by discrete Gronwall's inequality and the standard argument of mathematical induction with the inverse equality $ \| \phi \|_{L^\infty} \leq C_\text{inv} h^{-\frac{d}{2}} \| \phi \|_{L^2} $ for all $\phi  \in X_N $ \cite{book_spectral}: 
	\begin{align}\label{eq:inverse}
		\| \psi^k \|_{L^\infty} 
		&\leq \| \psi^{k} - P_N \psi(t_k) \|_{L^\infty} + \| P_N \psi(t_k) \|_{L^\infty} \notag \\
		&\leq C_\text{inv} h^{-\frac{d}{2}} \| e^k \|_{L^2} + C \| P_N \psi(t_k) \|_{H^2} \leq C_\text{inv} h^{-\frac{d}{2}} \| e^k \|_{L^2} + C(M_2), 
	\end{align}
	where $ d $ is the dimension of the space, i.e. $ d=1 $ in the current case. To estimate $h^{-\frac{d}{2}} \| e^k \|_{L^2}$ in \cref{eq:inverse}, we also need to use the time step size restriction $\tau \lesssim h^2$ imposed in \cref{eq:Snl}. As a result, we obtain
	\begin{equation*}
		\| e^n \|_{L^2} \lesssim \tau + h^2, \quad 0 \leq n \leq T/\tau, 
	\end{equation*}
	which directly yields the optimal $L^2$-norm error bound in \cref{eq:LT_L2} by recalling \cref{eq:reduce}. To obtain the $H^1$-norm error bound in \cref{eq:LT_L2}, as mentioned in \cref{rem:H1_bound}, by the inverse inequality, the step size restriction $\tau \lesssim h^2$ and standard projection error estimates of $P_N$, we have, for $0 \leq n \leq T/\tau$, 
	\begin{equation}\label{eq:H1_bound_inverse}
		\| \psi(t_n) - \psi^n \|_{H^1} \leq \| \psi(t_n) - P_N \psi(t_n) \|_{H^1} + \| e^n \|_{H^1} \lesssim h + h^{-1} \| e^n \|_{L^2} \lesssim \tau^\frac{1}{2} + h,  
	\end{equation}
	which completes the proof. 
	%		Finally, from \cref{eq:error_final_L2}, using discrete Gronwall's inequality, we have
	%		\begin{equation*}
		%			\| e^{n} \|_{L^2} \lesssim \tau + h^2, \quad 0 \leq n \leq T/\tau, 
		%		\end{equation*} 
	%		which completes the proof. 
	%		By Sobolev embedding theorem, one has, for $ 0 \leq n \leq T/\tau $, 
	%		\begin{equation*}
		%			\| P_N \psi(t_{n}) \|_{L^\infty} \leq c \| P_N \psi(t_{n}) \|_{H^2} \leq c \| \psi(t_{n}) \|_{H^2} \leq c M_2=:M_0. 
		%		\end{equation*}
	%		By the inverse inequality, there exists $ h_0>0 $ such that when $ 0 < h < h_0 $ and $ \tau \lesssim h^2 $, one has
	%		\begin{equation}\label{eq:inverse_estimate_1}
		%			\begin{aligned}
			%				\| \psi^{n+1} \|_{L^\infty} 
			%				&\leq \| \psi^{n+1} - P_N \psi(t_{n+1}) \|_{L^\infty} + \| P_N \psi(t_{n+1}) \|_{L^\infty} \\
			%				&\leq C_\text{inv} h^{-\frac{d}{2}} \| e^{n+1} \|_{L^2} + M_0 \leq 1+M_0. 
			%			\end{aligned}
		%		\end{equation}
	%		The rest of the proof follows from the standard mathematical induction procedure. 
\end{proof}
\begin{remark}
	The key estimate \cref{eq:est_J_LT_L2} of $\mathcal{J}^n$ \cref{eq:Jn} can be understood as a time-discrete and global version of Lemma 3.6 in \cite{bao2023_EWI}. 
\end{remark}

\subsection{Optimal $ H^1 $-norm error bound}
In this subsection, we shall show the optimal $ H^1 $-norm error bound \cref{eq:LT_H1} for the LTFS method \cref{full_discrete_LT} under the assumptions that $ V \in W^{1, 4}(\Omega) \cap H^1_\text{per}(\Omega) $, $ \sigma \geq 1/2 $ and $\psi \in C([0, T]; H_\text{per}^3(\Omega)) \cap C^1([0, T]; H^1(\Omega))$. We define a constant 
\begin{equation*}
	M_3 = \max\left \{ \| V \|_{W^{1, 4}}, \| \psi \|_{L^\infty([0, T]; H^3(\Omega))}, \| \partial_t \psi \|_{L^\infty([0, T]; H^1(\Omega))}, \| \psi \|_{L^\infty([0, T]; L^\infty(\Omega))} \right \}. 
\end{equation*}
Note that $ W^{1, 4}(\Omega) \hookrightarrow L^\infty(\Omega) $ when $ 1 \leq d \leq 3 $. The proof follows exactly the same framework as the proof of the optimal $L^2$-norm error bound \cref{eq:LT_L2}. 

We start with the estimates for the operator $B$ and present the $H^1$-norm counterparts of \cref{lem:B,lem:dB1,lem:Phi_B^tau1} under the above higher regularity assumptions. The proof of these lemmas can be found in \cite{bao2023_semi_smooth} (see Lemmas 4.2, 4.3 and 4.5 there). 

\begin{lemma}\label{lem:B_2}
	Under the assumptions that $ V \in W^{1, 4}(\Omega) $ and $ \sigma \geq 1/2 $, for any $ v, w \in H^2(\Omega) $ such that $ \| v \|_{H^2} \leq M $, $ \| w \|_{H^2} \leq M $, we have
	\begin{equation}\label{lem:diff_B2}
		\| B(v) - B(w) \|_{H^1} \leq C(\| V \|_{W^{1, 4}}, M) \| v - w \|_{H^1}. 
	\end{equation}
	In particular, when $w = 0$, we have
	\begin{equation}\label{lem:B2}
		\| B(v) \|_{H^1} \leq C(\| V \|_{W^{1, 4}}, M) \| v \|_{H^1}. 
	\end{equation}
\end{lemma}

\begin{lemma}\label{lem:dB2}
	Under the assumptions that $ V \in W^{1, 4}(\Omega) \cap H^1_\text{per}(\Omega) $ and $ \sigma \geq 1/2 $, for any $ v, w \in H^1_\text{per}(\Omega) $ satisfying $ \| v \|_{L^\infty} \leq M $, $ \| w \|_{L^\infty} \leq M $, we have $dB(v)[w] \in H^1_\text{per}(\Omega)$ and
	\begin{equation*}
		\| dB(v)[w] \|_{H^1} \leq C(\| V \|_{W^{1, 4}}, M) (\| v \|_{H^1} + \| w \|_{H^1}). 
	\end{equation*}
\end{lemma}

\begin{lemma}\label{lem:Phi_B^tau2}
	Under the assumptions that $ V \in W^{1, 4}(\Omega) $ and $ \sigma \geq 1/2 $, for any $ v \in H^1(\Omega) $ satisfying $ \| v \|_{L^\infty} \leq M $, we have
	\begin{equation*}
		\| \Phi_B^\tau(v) \|_{H^1} \leq (1 + C(\| V \|_{W^{1, 4}}, M) \tau) \| v \|_{H^1}. 
	\end{equation*}
\end{lemma}

Similar to \cref{prop:dominant_error_LT_L2}, we have
\begin{proposition}\label{prop:dominant_error_LT_H1}
	Under the assumptions $ V \in W^{1, 4}(\Omega) \cap H^1_\text{per}(\Omega) $, $ \sigma \geq 1/2 $ and $\psi \in C([0, T]; H_\text{per}^3(\Omega)) \cap C^1([0, T]; H^1(\Omega))$, for the local error $ \mathcal{E}^n $ defined in \cref{eq:E^n_def}, we have
	\begin{equation*}
		\mathcal{E}^n = \mathcal{E}_1^n + \mathcal{E}_2^n, 
	\end{equation*}
	where
	\begin{equation*}
		\| \mathcal{E}_1^n \|_{H^1} \lesssim \tau^2 + \tau h^2, \quad \mathcal{E}_2^n = -e^{i\tau\Delta}\int_0^\tau (I-e^{-is\Delta}) P_N B(P_N \psi(t_n)) \rmd s. 
	\end{equation*}
\end{proposition}

\begin{proof}
	Again, we only sketch the proof here and the details can be found in \cite{bao2023_semi_smooth}. We have the same error decomposition as \cref{eq:error_decomp} in the proof of \cref{prop:dominant_error_LT_L2}. By \cref{lem:B_2,lem:dB2,lem:Phi_B^tau2}, Sobolev embeddings and the boundedness of $e^{it\Delta}$ and $P_N$ on $H^1_\text{per}(\Omega)$, we have
	\begin{equation*}
		\| e_1 \|_{H^1} \lesssim \tau^2, \quad \| e_2 \|_{H^1} \lesssim \tau^2. 
	\end{equation*}
	For $ e_3^1 $ and $ e_3^2 $ defined in \cref{eq:decomp_e3}, using \cref{lem:diff_B2}, the boundedness of $ P_N $ and $ e^{it\Delta} $, the standard projection error estimates of $P_N$ and 	$\| (I - e^{i t \Delta}) \phi \|_{H^1} \lesssim t \| \phi \|_{H^3}$ for all $ \phi \in H^3_\text{per}(\Omega) $, we have 
	\begin{equation*}
		\| e_3^1 \|_{H^1} \lesssim \tau^2, \quad \| e_3^2 \|_{H^1} \lesssim \tau h^2. 
	\end{equation*}
	The conclusion follows again by letting $ \mathcal{E}_1^n = e_1 + e_2 + e_3^1 + e_3^2 $ and $ \mathcal{E}_2^n = e_3^3 $. 
\end{proof}

We can also establish the following $L^\infty$-conditional $H^1$-stability estimate of the nonlinear flow. 
\begin{proposition}\label{prop:stability_LT_H1}
	Let $ v, w \in X_N $ such that $ \| v \|_{L^\infty} \leq M $, $ \| w \|_{L^\infty} \leq M $ and $ \| v \|_{H^2} \leq M_1 $. Then we have
	\begin{equation*}
		\| P_N (\Phi_B^\tau(v) - v) - P_N (\Phi_B^\tau(w) - w) \|_{H^1} \leq C(\| V \|_{W^{1, 4}}, M, M_1) \tau \| v - w \|_{H^1}. 
	\end{equation*}
\end{proposition}

\begin{proof}
	%		By \cref{prop:stability_L2} and the standard property of $P_N$, it suffices to show that
	%		\begin{equation}\label{reduce}
		%			\| \partial_x ((\Phi_B^\tau(v) - v) - (\Phi_B^\tau(w) - w)) \|_{L^2} \leq C(\| V \|_{W^{1, 4}}, M, M_1) \tau \| v - w \|_{H^1}. 
		%		\end{equation}
	From \cref{eq:diff_Phi_B}, we have
	\begin{equation}\label{H1_stab_decomp}
		\partial_x (\Phi_B^\tau(v) - v) - \partial_x (\Phi_B^\tau(w) - w) = \int_0^1 (W_1 -i\tau W_2 -i\tau W_3) \rmd \theta, 
	\end{equation}
	where, recalling $\phiu{u^\theta} = V + f(|u^\theta|^2)$ in \cref{eq:B} and $u^\theta = (1-\theta)w + \theta v \ (0 \leq  \theta \leq 1)$,  
	\begin{equation}\label{W_def}
		\begin{aligned}
			&W_1 = \partial_x \left [ (v-w) \left ( e^{-i\tau \phiu{u^\theta}} - 1 \right )\right ], \quad  W_2 = \partial_x \left ( \sigma f(|u^\theta|^2) (v-w) e^{-i\tau \phiu{u^\theta}} \right), \\
			&W_3 =\partial_x \left ( G(u^\theta) \overline{(v - w)} e^{-i\tau \phiu{u^\theta}} \right). 
		\end{aligned}
	\end{equation}
	%		we have
	%		\begin{align*}
		%			\left \| (\Phi_B^\tau(v) - v) - (\Phi_B^\tau(w) - w) \right \|_{H^1} 
		%			&\leq \int_0^1 \left \| (v-w) \left ( e^{-i\tau \phiu{u^\theta}} - 1 \right ) \right \|_{H^1} \rmd \theta \\
		%			&\quad - i \tau \int_0^1 \left \| \left ( \sigma f(|u^\theta|^2) (v-w) + G(u^\theta) \overline{(v - w)} \right ) e^{-i\tau \phiu{u^\theta}} \right \|_{H^1} \rmd \theta. 
		%		\end{align*} 
	For $W_1$ in \cref{W_def}, by direct calculation, we have
	\begin{align}\label{W1}
		W_1 
		&= \partial_x (v-w) \left ( e^{-i\tau \phiu{u^\theta} } - 1 \right ) \notag \\
		&\quad - i \tau (v-w) ( \partial_x V + f^\prime(|u^\theta|^2) (u^\theta \partial_x \overline{u^\theta} + \overline{u^\theta} \partial_x u^\theta ) ) e^{-i\tau \phiu{u^\theta}},  
	\end{align}
	where $ f'(|z|^2)z $ and $ f'(|z|^2)\overline{z} $ with $z \in \C$ are defined to be $ 0 $ when $ z = 0 $ and $\sigma \geq 1/2$. From \cref{W1}, recalling $|e^{i \theta} - 1| \leq \theta$ for all $\theta \in \R$ and (by the convention $|z|^0 \equiv 1, z \in \C$)
	\begin{equation}\label{fbound}
		|f(|z|^2)| \lesssim |z|^{2\sigma}, \quad |f^\prime(|z|^2)z| + |f^\prime(|z|^2)\overline{z}| \lesssim |z|^{2\sigma - 1}, \quad \sigma \geq \frac{1}{2}, \quad z \in \C, 
	\end{equation}
	we have, by H\"older's inequality and Sobolev embedding, 
	\begin{align}\label{W1_est}
		\| W_1 \|_{L^2}
		&\leq \tau \| \phiu{u^\theta} \|_{L^\infty} \| \partial_x (v - w) \|_{L^2} + \tau \left \|  (v-w) \partial_x V \right \|_{L^2} \notag \\
		&\quad + C(M) \tau \| (v-w) \partial_x u^\theta \|_{L^2} \notag\\
		&\leq C(\| V \|_{L^\infty}, M) \tau \| v - w \|_{H^1} + \tau \| V \|_{W^{1, 4}} \| v-w \|_{L^4}  \notag\\
		&\quad + C(M) \tau \| (v-w) \partial_x v \|_{L^2} + C(M) \theta \tau \| (v-w) \partial_x (v-w) \|_{L^2} \notag\\
		&\leq C(\| V \|_{W^{1, 4}}, M) \tau \| v - w \|_{H^1} + C(M) \tau \| v \|_{W^{1, 4}} \| v-w \|_{L^4} \notag \\
		&\quad + C(M) \tau \| \partial_x (v-w) \|_{L^2} \notag\\
		&\leq C(\| V \|_{W^{1, 4}}, M, M_1) \tau \| v - w \|_{H^1}. 
	\end{align}
	For $W_2$ in \cref{W_def}, direct calculation yields
	\begin{align}\label{W2}
		W_2
		&= \sigma f^\prime(|u^\theta|^2) (u^\theta \partial_x \overline{u^\theta} + \overline{u^\theta} \partial_x u^\theta ) (v-w) e^{-i\tau \phiu{u^\theta}} \notag\\
		&\quad + \sigma f(|u^\theta|^2) \partial_x (v-w) e^{-i\tau \phiu{u^\theta}} -i\tau \sigma f(|u^\theta|^2) (v-w) \partial_x \phiu{u^\theta} e^{-i\tau \phiu{u^\theta}}. 
	\end{align}
	From \cref{W2}, using H\"older's inequality and Sobolev embedding and noting \cref{fbound}, we have
	\begin{align}\label{W2_est}
		\| W_2 \|_{L^2} 
		&\lesssim C(M) \| (v-w) \partial_x u^\theta \|_{L^2} + C(\| V \|_{L^\infty}, M) \| \partial_x ( v- w) \|_{L^2} \notag\\
		&\quad  + \tau C(M) \| (v-w) \partial_x V \|_{L^2} + \tau \| (v-w) f(|u^\theta|^2) \partial_x f(|u^\theta|^2) \|_{L^2} \notag \\
		&\leq C(M) \| (v-w) \partial_x v \|_{L^2} + \theta C(M) \| (v-w) \partial_x (v-w) \|_{L^2} \notag \\
		&\quad + C(\| V \|_{L^\infty}, M) \| v-w \|_{H^1} +\tau C(M) \| \partial_x V \|_{L^4} \| v - w \|_{H^1}  \notag \\
		&\quad + \tau \| (v-w) \partial_x [f(|u^\theta|^2)]^2 \|_{L^2} \notag \\
		&\leq C(M) \| \partial_x v \|_{L^4} \| v-w \|_{L^4} + C(M) \| \partial_x (v-w) \|_{L^2} \notag \\
		&\quad + C(\| V \|_{W^{1, 4}}, M) \| v-w \|_{H^1} +  C(M) \| (v-w) \partial_x u^\theta \|_{L^2} \notag \\
		&\leq C(\| V \|_{W^{1, 4}}, M, M_1)\| v- w \|_{H^1}, 
	\end{align}
	where we use the identity  
	\begin{equation*}
		\partial_x [f(|u^\theta|^2)]^2 = \beta^2 \partial_x (|u^\theta|^2)^{2\sigma} = 2\sigma \beta^2 (|u^\theta|^2)^{2\sigma-1} (u^\theta \partial_x \overline{u^\theta} + \overline{u^\theta} \partial_x u^\theta), \quad \sigma \geq 1/2. 
	\end{equation*}
	Similarly, we have
	\begin{equation}\label{W3_est}
		\|  W_3 \|_{L^2} \leq C(\| V \|_{W^{1, 4}}, M, M_1)\| v- w \|_{H^1}. 
	\end{equation}
	From \cref{H1_stab_decomp}, using \cref{W1_est,W2_est,W3_est} and recalling \cref{prop:stability_L2}, we obtain the desired result. 
\end{proof}

\begin{proof}[Proof of \cref{eq:LT_H1} in \cref{thm:LT}]
	The proof is similar to the proof of \cref{eq:LT_L2}. From \cref{eq:error_eq_iter}, using \cref{prop:dominant_error_LT_H1,prop:stability_LT_H1}, the isometry property of $e^{it\Delta}$ and triangle inequality, we have
	\begin{equation}\label{eq:error_eq_LT_H1}
		\| e^{n+1} \|_{H^1} \lesssim \tau + h^2 + C_2 \tau \sum_{k=0}^n \| e^k \|_{H^1} + \left\| \mathcal{J}^n \right\|_{H^1},  \quad 0 \leq n \leq T/\tau-1, 
	\end{equation}
	where $\mathcal{J}^n$ is defined in \cref{eq:Jn} and $C_2$ depends on $M_3$ and $\max_{0 \leq k \leq n} \| \psi^k \|_{L^\infty}$. Recalling \cref{eq:Jn_phase,eq:summation_by_parts} and using Parseval's identity, similar to \cref{eq:est_J_LT_L2}, we have
	\begin{align}\label{Jn_est_H1}
		\| \mathcal{J}^n \|_{H^1}^2 
		&= (b-a)\sum_{l \in \mathcal{T}_N} (1+\mu_l^2) \delta_l^2 \left| S_{k, l}  \widehat{\phi}^n_l - \sum_{k=0}^{n-1} S_{k, l} \left(\widehat{\phi}^{k+1}_l - \widehat{\phi}^k_l\right) \right|^2 \notag\\
		&\lesssim \sum_{l \in \mathcal{T}_N} (1+\mu_l^2) \delta_l^2 S_{k, l}^2 \left|\widehat{\phi}^n_l \right|^2 + \sum_{l \in \mathcal{T}_N} (1+\mu_l^2) \delta_l^2 \left|\sum_{k=0}^{n-1} S_{k, l} \left(\widehat{\phi}^{k+1}_l - \widehat{\phi}^k_l\right) \right|^2 \notag\\
		&\lesssim \tau^2 \sum_{l \in \mathcal{T}_N} (1+\mu_l^2) \left|\widehat{\phi}^n_l \right|^2 + \sum_{l \in \mathcal{T}_N} (1+\mu_l^2) \delta_l^2 \sum_{k_1=0}^{n-1} \left|S_{k_1, l}\right|^2 \sum_{k_2=0}^{n-1} \left|\widehat{\phi}^{k_2+1}_l - \widehat{\phi}^{k_2}_l\right|^2 \notag\\
		&\lesssim \tau^2 \| \phi^n \|_{H^1}^2 + \sum_{k_1=0}^{n-1} \sum_{k_2=0}^{n-1} \sum_{l \in \mathcal{T}_N} (1+\mu_l^2) \delta_l^2 \left|S_{k_1, l}\right|^2 \left|\widehat{\phi}^{k_2+1}_l - \widehat{\phi}^{k_2}_l\right|^2 \notag\\
		&\lesssim \tau^2 \| \phi^n \|_{H^1}^2 + n\tau^2 \sum_{k=0}^{n-1} \| \phi^{k+1} - \phi^{k} \|_{H^1}^2, \quad 0 \leq n \leq T/\tau-1. 
	\end{align}
	Recalling \cref{eq:phin}, using the projection property of $P_N$ and \cref{lem:B_2}, we have
	\begin{align*}
		&\| \phi^n \|_{H^1} \leq \| B(P_N \psi(t_n)) \|_{H^1} \leq C(M_3)\| \psi(t_n) \|_{H^1} \leq C(M_3), \quad 0 \leq n \leq 	\frac{T}{\tau}-1, \\
		& 
		\begin{aligned}
			\| \phi^{k+1} - \phi^{k} \|_{H^1}
			&= \| P_N B(P_N \psi(t_{k+1})) - P_N B(P_N \psi(t_k)) \|_{H^1} \\
			&\leq \| B(P_N \psi(t_{k+1})) - B(P_N \psi(t_k)) \|_{H^1} \\
			&\lesssim \| \psi(t_{k+1}) - \psi(t_{k}) \|_{H^1} \lesssim \tau \| \partial_t \psi \|_{L^\infty([t_k, t_{k+1}]; H^1)}, \quad 0 \leq k \leq \frac{T}{\tau}-2,
		\end{aligned} 
	\end{align*}
	which, inserted into \cref{Jn_est_H1}, yield
	\begin{equation*}
		\| \mathcal{J}^n \|^2_{H^1} \lesssim \tau^2 + n^2 \tau^4 \lesssim \tau^2, \quad 0 \leq n \leq T/\tau -1. 
	\end{equation*}
	With the above estimate of $\| \mathcal{J}^n \|_{H^1}$, noting the uniform $ L^\infty $-bound of $ \psi^n \ (0 \leq n \leq T/\tau) $ established in the proof of \cref{eq:LT_L2} in \cref{sec:proof_S1_L2}, the proof can be completed by applying discrete Gronwall's inequality to \cref{eq:error_eq_LT_H1}. 
\end{proof}

\section{Proof of \cref{thm:ST} for the STFS \cref{full_discrete_ST}}
In this section, we shall show the optimal error bounds for the STFS method \cref{full_discrete_ST}. Again, we start with the optimal $L^2$-norm error bounds and show the proof of \cref{eq:ST_L2} in \cref{thm:ST}. In the rest of this section, we assume that $ V \in H^2_\text{per}(\Omega) $, $ \sigma \geq 1 $ and $\psi \in C([0, T]; H_\text{per}^4(\Omega)) \cap C^1([0, T]; H^2(\Omega)) \cap C^2([0, T]; L^2(\Omega))$, and define a constant 
\begin{equation*}
	M_4 = \max\left \{ \| V \|_{H^2}, \| \psi \|_{L^\infty([0, T]; H^4(\Omega))}, \| \partial_t \psi \|_{L^\infty([0, T]; H^2(\Omega))}, \| \partial_{tt} \psi \|_{L^\infty([0, T]; L^2(\Omega))} \right \}. 
\end{equation*}
\subsection{Some estimates for the operator B}
\begin{lemma}\label{lem:f_H2}
	When $ \sigma \geq 1 $, for any $ v \in H^2(\Omega) $, we have
	\begin{equation*}
		\| f(|v|^2) \|_{H^2} \leq C(\| v \|_{H^2}), \quad \| G(v) \|_{H^2} \leq C(\| v \|_{H^2}). 
	\end{equation*}
\end{lemma}

\begin{proof}
	By direct calculation, we have
	\begin{equation}\label{eq:par_f}
		\begin{aligned}
			&\partial_x f(|v|^2) = f^\prime(|v|^2) (v \partial_x \overline{v} + \overline{v} \partial_x v), \\
			&\partial_{xx} f(|v|^2) = f^{\prime\prime}(|v|^2) (v \partial_x \overline{v} + \overline{v} \partial_x v)^2 + f^\prime(|v|^2) (v \partial_{xx} \overline{v} + \overline{v} \partial_{xx} v + 2 |\partial_x v|^2 ). 
		\end{aligned}
	\end{equation}
	From \cref{eq:par_f}, using Sobolev embedding and noting \cref{fbound} and
	\begin{equation}
		|f^\prime(|z|^2)| + | f^{\prime\prime}(|z|^2)z^2 | + | f^{\prime\prime}(|z|^2)|z|^2 | + | f^{\prime\prime}(|z|^2)\overline{z}^2 | \lesssim |z|^{2\sigma - 2}, \quad  \sigma \geq 1, \quad z \in \C, 
	\end{equation}
	we have
	\begin{align*}
		&\| f(|v|^2) \|_{L^2} \lesssim \| v \|_{L^\infty}^{2\sigma} \leq C(\| v \|_{H^2}), \quad \| \partial_x f(|v|^2) \|_{L^2} \lesssim \| v \|_{L^\infty}^{2\sigma - 1} \| \partial_x v \|_{L^2} \leq C(\| v \|_{H^2}), \\
		&\| \partial_{xx} f(|v|^2) \|_{L^2} \lesssim \| v \|_{L^\infty}^{2\sigma - 2} \| \partial_x v \|_{L^4}^2 + \| v \|_{L^\infty}^{2\sigma -1} \| \partial_{xx} v \|_{L^2} \leq C(\| v \|_{H^2}), 
	\end{align*}
	which implies $ \| f(|v|^2) \|_{H^2} \leq C(\| v \|_{H^2}) $. Similarly, we can show that $  \| G(v) \|_{H^2} \leq C(\| v \|_{H^2}) $ and the proof is completed. 
\end{proof}

%	\begin{proof}
	%		Note that when $ \sigma \geq 1 $, $ f(|\cdot|^2) $ and $ G(\cdot) $ are $ C^2 $ in the real sense. Then the conclusion is a direct consequence of Sobolev embedding and H\"older's inequality. 
	%	\end{proof}

\begin{lemma}\label{lem:B3}
	Under the assumptions $ V \in H^2(\Omega) $ and $ \sigma \geq 1 $, for any $ v, w \in H^2(\Omega) $ such that $ \| v \|_{H^2} \leq M $, $ \| w \|_{H^2} \leq M $, we have
	\begin{equation}\label{lem:diffB_H2}
		\| B(v) - B(w) \|_{H^2} \leq C(\| V \|_{H^2}, M) \| v - w \|_{H^2}. 
	\end{equation}
	In particular, when $w = 0$, we have
	\begin{equation}\label{lem:B_H2}
		\| B(v) \|_{H^2} \leq C(\| V \|_{H^2}, M). 
	\end{equation}
\end{lemma}

\begin{proof}
	Noting that $ H^2(\Omega) $ is an algebra when $ 1 \leq d \leq 3 $, we get
	\begin{equation}\label{diff_V_H2}
		\| V v - V w \|_{H^2} \leq \| V \|_{H^2} \| v - w \|_{H^2}. 
	\end{equation}
	Let $ u^\theta = (1-\theta)w + \theta v $ for $ 0 \leq \theta \leq 1 $ and define
	\begin{equation*}
		\gamma(\theta) = f(|u^\theta|^2)u^\theta, \quad 0 \leq \theta \leq 1. 
	\end{equation*}
	Then one has
	\begin{equation}\label{diff_fvv}
		f(|v|^2)v - f(|w|^2)w = \gamma(1) - \gamma(0) = \int_0^1 \gamma'(\theta) \rmd \theta , 
	\end{equation}
	where
	\begin{equation}\label{gamma_prime}
		\gamma'(\theta) = (1+\sigma) f(|u^\theta|^2) (v-w) + G(u^\theta) \overline{(v-w)}, \quad 0 \leq \theta \leq 1. 
	\end{equation}
	Inserting \cref{gamma_prime} into \cref{diff_fvv}, by \cref{lem:f_H2} and the algebra property of $H^2(\Omega)$, we have
	\begin{equation*}
		\| f(|v|^2)v - f(|w|^2)w \|_{H^2} \leq \int_0^1 \| \gamma'(\theta) \|_{H^2} \rmd \theta \leq C(M) \| v - w \|_{H^2}, 
	\end{equation*}
	which, combined with \cref{diff_V_H2}, completes the proof. 
\end{proof}

\begin{lemma}\label{lem:diff_dB}
	Under the assumptions $ V \in L^\infty(\Omega) $ and $ \sigma \geq 1/2 $, for any $ v_j, w_j \in L^\infty(\Omega) $ satisfying $ \| v_j \|_{L^\infty} \leq M $ and $ \| w_j \|_{L^\infty} \leq M $ with $ j = 1, 2 $, we have
	\begin{equation*}
		\| dB(v_1)[w_1] - dB(v_2)[w_2] \|_{L^2} \leq C(\| V \|_{L^\infty}, M) \left ( \| v_1 - v_2 \|_{L^2} + \| w_1 - w_2 \|_{L^2} \right ). 
	\end{equation*}
\end{lemma}

\begin{proof}
	Recalling \cref{eq:dB_def}, we have
	\begin{align}\label{diff_dB}
		\| dB(v_1)[w_1] - dB(v_2)[w_2] \|_{L^2} 
		&\lesssim \| V \|_{L^\infty} \| w_1 -  w_2 \|_{L^2} + \| f(|v_1|^2)w_1 - f(|v_2|^2)w_2 \|_{L^2} \notag\\
		&\quad + \| G(v_1)\overline{w_1} - G(v_2)\overline{w_2} \|_{L^2}. 
	\end{align}
	When $\sigma \geq 1/2 $, we have (see (4.9) in \cite{bao2023_semi_smooth})
	\begin{equation}\label{eq:diff_f}
		| f(|z_1|^2) - f(|z_2|^2) | \lesssim \max\{|z_1|, |z_2|\}^{2\sigma -1} |z_1 - z_2|, \quad z_1, z_2 \in \C, 
	\end{equation}
	which implies
	\begin{align}\label{diff_dB_est_1}
		&\| f(|v_1|^2)w_1 - f(|v_2|^2)w_2 \|_{L^2} \notag \\
		&\leq \| ( f(|v_1|^2) - f(|v_2|^2) )w_1 \|_{L^2} + \| f(|v_2|^2) (w_1 - w_2) \|_{L^2} \notag\\
		&\leq \| f(|v_1|^2) - f(|v_2|^2) \|_{L^2} \| w_1 \|_{L^\infty} + \| f(|v_2|^2) \|_{L^\infty} \| w_1 - w_2 \|_{L^2} \notag\\
		&\leq C(M) \left(\| v_1 - v_2 \|_{L^2} + \| w_1 - w_2 \|_{L^2} \right). 
	\end{align}
	Similarly, we have 
	\begin{equation}\label{diff_dB_est_2}
		\| G(v_1)\overline{w_1} - G(v_2)\overline{w_2} \|_{L^2} \leq C(M) \left( \| v_1 - v_2 \|_{L^2} + \| w_1 - w_2 \|_{L^2} \right). 
	\end{equation}
	Inserting \cref{diff_dB_est_1,diff_dB_est_2} into \cref{diff_dB} yields the desired estimate. 
\end{proof}

%	\begin{proof}
	%		We define $ f^\vep $ as in \cite{bao2023semi_smooth} and let $ w^\vep = f^\vep(|v|^2) $ and $ w = f(|v|^2) $. Then one has $ w^\vep \rightarrow w $ in $ L^2(\Omega) $ as $ \vep \rightarrow 0 $. Moreover, one has
	%		\begin{equation*}
		%			\sup_{0<\vep<1} \| w^\vep \|_{H^2} \leq C(\| v \|_{H^2}) < \infty. 
		%		\end{equation*}
	%		By the weak compactness of $ H^2(\Omega) $ and the compact embedding $ H^2(\Omega) \subset \subset L^2(\Omega) $, one has $ w \in H^2(\Omega) $ and
	%		\begin{equation*}
		%			\| w \|_{H^2} \leq C(\| v \|_{H^2}). 
		%		\end{equation*}
	%	\end{proof}

\begin{lemma}\label{lem:dB_H2}
	Under the assumptions $ V \in H^2(\Omega) $ and $ \sigma \geq 1 $, for any $ v, w \in H^2(\Omega) $ satisfying $ \| v \|_{H^2} \leq M $, $ \| w \|_{H^2} \leq M $, we have
	\begin{equation*}
		\| dB(v)[w] \|_{H^2} \leq C(\| V \|_{H^2}, M). 
	\end{equation*}
\end{lemma}

\begin{proof}
	Recalling \cref{eq:dB_def} and using \cref{lem:f_H2} and the algebra property of $H^2(\Omega)$, we obtain the desired result. 
\end{proof}

\begin{lemma}\label{lem:dG_L2}
	Under the assumption that $\sigma \geq 1/2$, for any $ v \in L^\infty(\Omega) $ and $w \in L^2(\Omega)$ satisfying $ \| v \|_{L^\infty} \leq M $, we have
	\begin{equation*}
		\| dG(v)[w] \|_{L^2} \leq C(M)\| w \|_{L^2}. 
	\end{equation*}
\end{lemma}

\begin{proof}
	Recalling \cref{eq:Gateaux,eq:G_def}, we have
	\begin{equation}\label{eq:dG_def}
		dG(v)[w] = f''(|v|^2) v^3 \overline{w} + f''(|v|^2) v^2 \overline{v} w + 2 f'(|v|^2)v w, 
	\end{equation}
	where $ f''(|z|^2) z^3 $ and $ f''(|z|^2) z^2 \overline{z} $ with $z \in \C$ are defined as $ 0 $ when $ z=0 $. From \cref{eq:dG_def}, using the fact that 
	\begin{equation}\label{f_bound2}
		\left | f''(|z|^2) z^3 \right | + \left | f''(|z|^2) z^2 \overline{z} \right | + \left | f'(|z|^2)z \right | 
		\lesssim |z|^{2\sigma-1}, \quad \sigma \geq 1/2, \quad z \in \C, 
	\end{equation}
	and Sobolev embedding, we obtain
	\begin{equation*}
		\| dG(v)[w] \|_{L^2} \lesssim \| v \|_{L^\infty}^{2\sigma -1} \| w \|_{L^2} \leq C(M) \| w \|_{L^2}, 
	\end{equation*}
	which completes the proof. 
\end{proof}

\subsection{Local truncation error and stability estimates}
Similar to the estimate of the first-order Lie-Trotter time-splitting method, we first establish the local truncation error and stability estimates for the second-order Strang splitting. 

Define the local truncation error of the Strang time-splitting method as
\begin{equation}\label{eq:barE_def}
	\mcalL^n = P_N \psi(t_{n+1}) - \mathcal{S}_2^\tau (P_N \psi(t_n)), \quad 0 \leq n \leq T/\tau - 1. 
\end{equation}
\begin{proposition}\label{prop:dominant_error_ST_L2}
	Assuming that $ V \in H^2_\text{per}(\Omega) $, $ \sigma \geq 1 $ and $\psi \in C([0, T]; H_\text{per}^4(\Omega)) \cap C^1([0, T]; H^2(\Omega)) \cap C^2([0, T]; L^2(\Omega))$, we have
	\begin{equation*}
		\mcalL^n = \mcalL_1^n + \mcalL_2^n, 
	\end{equation*}
	where 
	\begin{equation*}
		\| \mcalL_1^n \|_{L^2} \lesssim \tau^3 + \tau h^4, \quad \mcalL_2^n = \tau^3 \Delta^2 \int_0^1 \ker(\theta) e^{i (1-\theta)\tau\Delta} P_N B(e^{i\theta \tau \Delta} P_N \psi(t_n)) \rmd \theta, 
	\end{equation*}
	with $\ker(\theta)$ the peano kernel for the mid-point rule. 
	%and $ u^n(s) = e^{i (\tau - s )\Delta} P_N B(e^{is \Delta} P_N \psi(t_n)) \  (0 \leq s \leq \tau)$ defined in \cref{eq:f_and_phi_def}. 
\end{proposition}

\begin{proof}
	From \cref{eq:duhamel}, using \cref{,lem:diff_B1,lem:diff_dB}, the isometry property of $e^{it\Delta}$ and the standard projection error estimates of $P_N$, we have
	\begin{align}\label{eq:duhamel_modify}
		&P_N\psi(t_{n+1}) = e^{i \tau \Delta} P_N \psi(t_n) + \int_0^\tau e^{i (\tau - s )\Delta} P_N B(e^{is \Delta} P_N \psi(t_n)) \rmd s \\
		& + \int_0^\tau \int_0^s e^{i (\tau - s) \Delta} P_N\left( dB(e^{i(s-\vsigma)\Delta} P_N \psi(t_n+\vsigma)) [e^{i(s - \vsigma)\Delta} B(P_N \psi(t_n+\vsigma))] \right) \rmd \vsigma \rmd s + e_h, \notag
	\end{align}
	where $ \| e_h \|_{L^2} \lesssim \tau h^4 $. By the Taylor expansion for $ \Phi_B^\tau $ in \cref{eq:taylor}, we have
	\begin{align}\label{eq:S2_expan}
		\mathcal{S}_2^\tau(P_N \psi(t_n)) 
		&= e^{i \tau \Delta} P_N \psi(t_n) + \tau e^{i \frac{\tau}{2} \Delta} P_N B(e^{i \frac{\tau}{2} \Delta} P_N \psi(t_n)) \notag \\
		&\quad + \tau^2 \int_0^1 (1-\theta) \gamma(\theta) \rmd \theta, 
	\end{align}
	where
	\begin{equation}\label{eq:gamma_def}
		\gamma(\theta) =  e^{i \frac{\tau}{2} \Delta} P_N dB(\Phi_B^{\theta \tau}(e^{i \frac{\tau}{2} \Delta} P_N \psi(t_n) )) [B(\Phi_B^{\theta \tau}(e^{i \frac{\tau}{2} \Delta} P_N \psi(t_n)))], \quad 0 \leq \theta \leq 1. 
	\end{equation}
	Subtracting \cref{eq:S2_expan} from \cref{eq:duhamel_modify}, we have 
	\begin{equation*}
		\mcalL^n = P_N\psi(t_{n+1}) - \mathcal{S}_2^\tau(P_N \psi(t_n)) = e_1 + e_2 + e_3 + e_h, 
	\end{equation*}
	where
	\begin{equation}\label{eq:error_decomp_S2}
		\begin{aligned}
			&e_1 = \int_0^\tau \int_0^s e^{i (\tau - s) \Delta} P_N\left( dB(e^{i(s-\vsigma)\Delta} P_N \psi(t_n+\vsigma)) [e^{i(s - \vsigma)\Delta} B(P_N \psi(t_n+\vsigma))] \right) \rmd \vsigma \rmd s, \\
			&e_2 = -\tau^2 \int_0^1 (1-\theta) \gamma(\theta) \rmd \theta, \\
			&e_3 = \int_0^\tau e^{i (\tau - s )\Delta} P_N B(e^{is \Delta} P_N \psi(t_n)) \rmd s - \tau e^{i \frac{\tau}{2} \Delta} P_N B(e^{i \frac{\tau}{2} \Delta} P_N \psi(t_n)). 
		\end{aligned}
	\end{equation}
	Set, for $0 \leq \vsigma \leq s \leq \tau$, 
	\begin{equation}\label{eq:D}
		D(s, \vsigma) = e^{i (\tau - s) \Delta} P_N \left( dB(e^{is\Delta} P_N \psi(t_n) )[e^{i(s - \vsigma)\Delta} B( e^{i \vsigma \Delta} P_N \psi(t_n))] \right).
	\end{equation}
	Noting that
	\begin{equation*}
		D\left( \frac{\tau}{2}, \frac{\tau}{2} \right) = \gamma(0) = e^{i \frac{\tau}{2} \Delta} P_N dB(e^{i\frac{\tau}{2}\Delta} P_N \psi(t_n) )[B( e^{i \frac{\tau}{2} \Delta} P_N \psi(t_n))], 
	\end{equation*}
	we have
	\begin{equation*}
		e_1 + e_2 = r_1 + r_2 + r_3, 
	\end{equation*}
	where
	\begin{equation}\label{eq:r_j}
		\begin{aligned}
			&r_1 =\int_0^\tau \int_0^s \left[e^{i (\tau - s) \Delta} P_N \left(dB(e^{i(s-\vsigma)\Delta} P_N \psi(t_n+\vsigma)) [e^{i(s - \vsigma)\Delta} B(P_N \psi(t_n+\vsigma))] \right) \right. \\
			&\qquad \qquad \qquad \left. \vphantom{e^{i (\tau - s) \Delta}}- D(s, \vsigma) \right] \rmd \vsigma \rmd s,\\
			&r_2 = \int_0^\tau \int_0^s \left[ D(s, \vsigma) - D\left( \frac{\tau}{2}, \frac{\tau}{2} \right) \right] \rmd \vsigma \rmd s, \qquad r_3 = \tau^2 \int_0^1 (1-\theta) [\gamma(0) - \gamma(\theta)] \rmd \theta.
		\end{aligned} 
	\end{equation}
	Then one has (the proof is postponed to \cref{lem:r_j})
	\begin{equation}\label{eq:r_j_L_2}
		\| r_j \|_{L^2} \lesssim \tau^3 \text{ for } j = 1, 2, 3. 
	\end{equation}
	To estimate $ e_3 $ in \cref{eq:error_decomp_S2}, let
	\begin{equation}\label{eq:f_and_phi_def}
		g^n(s) = e^{i (\tau - s )\Delta} P_N B(e^{is \Delta} P_N \psi(t_n)), \quad 0 \leq s \leq \tau. 
	\end{equation}
	Then we have
	\begin{equation}\label{eq:e3_S2}
		e_3 = \int_0^\tau \left[ g^n(s) - g^n(\tau/2) \right] \rmd s = \tau^3 \int_0^1 \ker(\theta) \partial_{ss} {g^n}(\theta \tau) \rmd \theta, 
	\end{equation}
	where $ \ker(\theta) $ is the Peano kernel for the mid-point rule (see \cite{lubich2008,review_2013}). 
	For simplicity of the presentation, we define
	\begin{equation}\label{eq:v_w_def}
		v(s) = P_N B(e^{is \Delta} P_N \psi(t_n)), \quad w(s) = e^{is \Delta} P_N \psi(t_n), \quad 0 \leq s \leq \tau. 
	\end{equation}
	Note that we have $g^n(s) = e^{i (\tau - s )\Delta} v(s) $ for $0 \leq s \leq \tau$ and 
	\begin{equation}\label{eq:parpar_g}
		\partial_{ss}{g^n}(s) = e^{i (\tau - s) \Delta} \left( \partial_{ss} {v}(s) - i\Delta \partial_s v(s) - \Delta^2 v(s) \right), \quad 0 \leq s \leq \tau. 
	\end{equation}
	Then $e_3$ \cref{eq:e3_S2} can be decomposed as
	\begin{align}\label{eq:e3_decomp_S2}
		e_3 
		&=  \tau^3 \int_0^1 \ker(\theta) e^{i (1-\theta)\tau \Delta} \left( \partial_{ss} {v}(\theta\tau) - i\Delta \partial_s v(\theta \tau) \right) \rmd \theta \notag \\
		&\quad - \tau^3 \int_0^1 \ker(\theta) e^{i (1-\theta)\tau \Delta} \Delta^2 v(\theta \tau) \rmd \theta =: e_3^1 + e_3^2. 
	\end{align}
	From \cref{eq:v_w_def}, by direct calculation, we have
	\begin{align}
		\partial_s v(s) 
		&= P_N dB(w(s)) [i\Delta w(s)] \notag \\
		&= P_N \left[ V \Delta w(s) + (1+\sigma)f(|w(s)|^2) \Delta w(s) + G(w(s)) \Delta \overline{w(s)} \right], \label{eq:par_v} \\
		\partial_{ss} v(s) 
		&= P_N \left[ iV \Delta^2 w(s) + (1+\sigma)f'(|w(s)|^2) (i\Delta w(s) \overline{w(s)} - i w(s) \Delta \overline{w(s)}) \Delta w(s) \right. \notag\\
		&\qquad \quad + (1+\sigma)f(|w(s)|^2) i \Delta^2 w(s) + dG(w(s))[i\Delta w(s)] \Delta \overline{w(s)} \notag \\
		&\qquad \quad \left. - i G(w(s)) \Delta^2 \overline{w(s)} \right]. \label{eq:parpar_v}
	\end{align}
	From \cref{eq:par_v}, recalling that $ V \in H_\text{per}^2(\Omega) $ and $ \psi \in C([0, T]; H^4_\text{per}(\Omega)) $ and using the algebra property of $H^2(\Omega)$, \cref{lem:f_H2} and the boundedness of $P_N$ on $H^2_\text{per}(\Omega)$, we have, for $0 \leq s \leq \tau$, 
	\begin{equation}\label{est_1}
		\| \Delta \partial_s v(s) \|_{L^2} \lesssim \| \partial_s v(s) \|_{H^2} \lesssim \| V \|_{H^2} \| w \|_{H^4} + C(\| w \|_{H^2})\| w \|_{H^4} \leq C(M_4). 
	\end{equation}
	From \cref{eq:parpar_v}, recalling that $ V \in H_\text{per}^2(\Omega) $ and $ \psi \in C([0, T]; H^4_\text{per}(\Omega)) $, using \cref{lem:dG_L2}, Sobolev embedding, and the boundedness of $P_N$, and noting \cref{eq:f_G_bound}, we have, 
	\begin{align}\label{est_2}
		\| \partial_{ss} v(s) \|_{L^2} 
		&\lesssim \| V \|_{L^\infty} \| w \|_{H^4} + \| w \|_{L^\infty}^{2\sigma -1} \| \Delta w \|_{L^4}^2 + \| w \|_{L^\infty}^{2\sigma} \| w \|_{H^4} \notag \\
		&\quad + C(\| w \|_{H^4}) \| \Delta w \|_{L^\infty} \leq C(M_4), \qquad 0 \leq s \leq \tau. 
	\end{align}
	Combing \cref{est_1,est_2}, recalling \cref{eq:e3_decomp_S2}, we have, by the isometry property of $e^{it\Delta}$, 
	\begin{equation}
		\| e_3^1 \|_{L^2} \lesssim \tau^3. 
	\end{equation}
	The conclusion then follows from letting $ \mcalL^n_1 = e_1 + e_2 + e_3^1 + e_h$ and $\mcalL^n_2 = e_3^2$.  
\end{proof}

\begin{lemma}\label{lem:r_j}
	Under the assumption of \cref{prop:dominant_error_ST_L2}, for $r_j \ (j = 1, 2, 3)$ defined in \cref{eq:r_j}, we have
	\begin{equation*}
		\| r_j\|_{L^2} \lesssim \tau^3, \quad j =1, 2, 3. 
	\end{equation*}
\end{lemma}

\begin{proof}
	Recalling \cref{eq:D}, by the boundedness of $e^{it\Delta}$ and $P_N$, \cref{lem:diff_dB,lem:B}, we have
	\begin{align}\label{eq:r_1_est}
		\| r_1 \|_{L^2} 
		&\leq \int_0^\tau \int_0^s \left\| dB(e^{i(s-\vsigma)\Delta} P_N \psi(t_n+\vsigma)) [e^{i(s - \vsigma)\Delta} B(P_N \psi(t_n+\vsigma))] \right. \notag\\
		&\left. \qquad\qquad - dB(e^{is\Delta} P_N \psi(t_n) )[e^{i(s - \vsigma)\Delta} B( e^{i \vsigma \Delta} P_N \psi(t_n))] \right\|_{L^2} \rmd \vsigma \rmd s \notag\\
		&\lesssim  \int_0^\tau \int_0^s \left( \| e^{i(s-\vsigma)\Delta} P_N \psi(t_n+\vsigma) - e^{is\Delta} P_N \psi(t_n) \|_{L^2} \right. \notag \\
		&\qquad \qquad \left. + \| e^{i(s - \vsigma)\Delta} B(P_N \psi(t_n+\vsigma)) - e^{i(s - \vsigma)\Delta} B( e^{i \vsigma \Delta} P_N \psi(t_n)) \|_{L^2} \right) \rmd \vsigma \rmd s \notag \\
		&\leq \tau^2 \sup_{0 \leq \sigma \leq \tau} \| \psi(t_n+\vsigma) - e^{i\vsigma\Delta} \psi(t_n) \|_{L^2} \notag \\ 
		&\leq \tau^2 \sup_{0 \leq \vsigma \leq \tau} \int_0^\vsigma \| e^{i(\vsigma - s)\Delta} B(\psi(t_n + s)) \|_{L^2} \rmd s \notag \\
		&\leq \tau^3 \sup_{0 \leq s \leq \tau} \| B(\psi(t_n + s)) \|_{L^2} \lesssim \tau^3, 
	\end{align}
	where we also use the Duhamel's formula for $\psi(t_n + \varsigma)$. To estimate $ r_2 $, recalling \cref{eq:D} and $w(t) = e^{it\Delta}P_N \psi(t_n)$ for $0 \leq t \leq \tau$ in \cref{eq:v_w_def}, we have
	\begin{equation}\label{par_s_D}
		\partial_s D(s, \vsigma) 
		= -i\Delta D(s, \vsigma) + e^{i (\tau - s) \Delta} P_N \partial_s \left( dB(w(s) )[e^{i(s - \vsigma)\Delta} B(w(\vsigma))] \right). 
	\end{equation}
	Recalling \cref{eq:dB_def}, by direct calculation, we have
	\begin{align*}
		&\partial_s \left( dB(w(s) )[e^{i(s - \vsigma)\Delta} B(w(\vsigma))] \right) \\
		&=-i \partial_s \left[ Ve^{i(s - \vsigma)\Delta} B(w(\vsigma)) + (1+\sigma)f(|w(s)|^2) e^{i(s - \vsigma)\Delta} B(w(\vsigma)) \right. \\
		&\qquad \qquad \left. + G(w(s)) \overline{e^{i(s - \vsigma)\Delta} B(w(\vsigma))} \right] \\
		&=V e^{i(s - \vsigma)\Delta} \Delta B(w(\vsigma)) -i(1+\sigma) f^\prime(|w(s)|^2) (w(s)\partial_s \overline{w(s)} + \overline{w(s)} \partial_s w(s))e^{i(s - \vsigma)\Delta} B(w(\vsigma)) \\
		&\quad + (1+\sigma)f(|w(s)|^2) e^{i(s - \vsigma)\Delta} \Delta B(w(\vsigma)) - i dG(w(s))[\partial_s w(s)] e^{-i(s - \vsigma)\Delta}\overline{B(w(\vsigma))} \\
		&\quad - G(w(s)) e^{-i(s - \vsigma)\Delta} \Delta \overline{B(w(\vsigma))}, 
	\end{align*}
	from which, noting \cref{eq:f_G_bound,fbound} and  
	\begin{equation}\label{v_Linfty}
		\| w(s) \|_{L^\infty} \lesssim \| w(s) \|_{H^2} \leq \| \psi(t_n) \|_{H^2} \leq M_4, \quad 0 \leq s \leq \tau, 
	\end{equation}
	we have
	\begin{align}\label{est_partial_s_dB}
		&\left \| \partial_s \left( dB(w(s) )[e^{i(s - \vsigma)\Delta} B(w(\vsigma))] \right) \right \|_{L^2} \notag \\
		&\lesssim \| B(w(\vsigma)) \|_{H^2} + \| \partial_s w(s) \|_{L^2} \| e^{i(s-\vsigma)\Delta} B(w(\vsigma)) \|_{L^\infty} \notag\\
		&\quad + \| dG(w(s))[\partial_s w(s)] \|_{L^2} \| e^{-i(s - \vsigma)\Delta}\overline{B(w(\vsigma))} \|_{L^\infty}. 
	\end{align}
	From \cref{est_partial_s_dB}, using \cref{lem:B_H2}, \cref{lem:dG_L2}, \cref{v_Linfty}, Sobolev embedding and the isometry property of $e^{it\Delta}$, we have
	\begin{align}\label{est_partial_s_dB_end}
		&\left \| \partial_s \left( dB(w(s) )[e^{i(s - \vsigma)\Delta} B(w(\vsigma))] \right) \right \|_{L^2} \notag \\
		&\lesssim C(\| V \|_{H^2}, \| \psi(t_n) \|_{H^2}) + \| \psi(t_n) \|_{H^2} \| B(w(\vsigma)) \|_{H^2} \notag\\
		&\quad + C(\| w(s) \|_{L^\infty})\| \partial_s w(s) \|_{L^2} \| B(w(\vsigma)) \|_{H^2} \leq C(M_4). 
	\end{align}
	From \cref{par_s_D}, using \cref{lem:dB_H2} and \cref{est_partial_s_dB_end} and the boundedness of $e^{it\Delta}$ and $P_N$, we obtain 
	\begin{align}\label{par_s_D_est}
		\| \partial_s D(s, \vsigma) \|_{L^2} 
		&\leq \| \Delta D(s, \vsigma) \|_{L^2} + \left \| \partial_s \left( dB(w(s) )[e^{i(s - \vsigma)\Delta} B(w(\vsigma))] \right) \right \|_{L^2} \notag \\
		&\leq C(M_4). 
	\end{align}
	%	Recalling \cref{eq:D} again, we have
	%	\begin{align*}
		%		\partial_\vsigma D(s, \vsigma) 
		%		&= -i e^{i (\tau -s )\Delta} P_N \partial_s \left[ \left (V+(1+\sigma)f(|w(s)|^2) \right ) e^{i(s - \vsigma)\Delta} B(w(\vsigma)) \right.\\
		%		&\qquad\qquad\qquad\qquad\quad \left. + G(w(s)) \overline{e^{i(s - \vsigma)\Delta} B(w(\vsigma))} \right] \\
		%		&=
		%	\end{align*}
	Similarly, using \cref{lem:B_H2,lem:dB1}, we have $ \| \partial_\vsigma D(s, \vsigma) \|_{L^2} \leq C(M_4) $, which, combined with \cref{par_s_D_est}, yields
	\begin{equation}
		\| \nabla D(s, \vsigma) \|_{L^2} \leq C(M_4), \quad 0 \leq \vsigma \leq s \leq \tau, 
	\end{equation}
	which further implies
	\begin{equation}\label{eq:r_2_est}
		\| r_2 \|_{L^2} \lesssim \tau^3 \sup_{0 \leq \vsigma \leq s \leq \tau} \| \nabla D(s, \vsigma) \|_{L^2} \lesssim \tau^3. 
	\end{equation}
	For $r_3$ in \cref{eq:r_j}, recalling \cref{eq:gamma_def}, \cref{eq:nonl_step} and $w(t) = e^{it\Delta}P_N\psi(t_n)$, by \cref{lem:diff_dB,lem:diff_B1}, we have
	\begin{align}\label{eq:r_3_est}
		&\| r_3 \|_{L^2} \leq \tau^2 \int_0^1 \| \gamma(\theta) - \gamma(0) \|_{L^2} \rmd \theta \notag\\
		&\leq \tau^2 \int_0^1 \left \| dB(\Phi_B^{\theta \tau}(w(\tau/2))) [B(\Phi_B^{\theta \tau}(w(\tau/2)))] - dB(w(\tau/2))[B(w(\tau/2))] \right \|_{L^2} \rmd \theta \notag\\
		&\lesssim \tau^2 \sup_{0 \leq \theta \leq 1} \left ( \| \Phi_B^{\theta \tau}(w(\tau/2)) - w(\tau/2) \|_{L^2} + \| B(\Phi_B^{\theta \tau}(w(\tau/2))) - B(w(\tau/2)) \|_{L^2} \right ) \notag\\
		&\lesssim \tau^2 \sup_{0 \leq \theta \leq 1} \left \| w(\tau/2) \left (e^{-i\tau (V + f(|w(\tau/2)|^2))} -1 \right ) \right \|_{L^2} \lesssim \tau^3. 
	\end{align}
	Combing \cref{eq:r_1_est,eq:r_2_est,eq:r_3_est}, we complete the proof. 
\end{proof}

\begin{proposition}\label{prop:stability_ST_L2}
	Let $ v, w \in X_N $ such that $ \| e^{i \frac{\tau}{2}\Delta} w \|_{L^\infty} \leq M $ and $ \| v \|_{H^2} \leq M_1 $, then we have
	\begin{equation*}
		\left \| \left( \mathcal{S}_2^\tau(v) - e^{i \tau \Delta} v \right) - \left(\mathcal{S}_2^\tau (w) -  e^{i \tau \Delta} w \right) \right \|_{L^2} \leq C(\| V \|_{L^\infty}, M, M_1) \tau \| v - w \|_{L^2}. 
	\end{equation*}
\end{proposition}

\begin{proof}
	Recalling \cref{eq:S1S2_S2} and noting that $P_N$ is an identity on $X_N$, we have
	\begin{align}\label{stab_S2}
		&\left( \mathcal{S}_2^\tau(v) - e^{i \tau \Delta} v \right) - \left(\mathcal{S}_2^\tau (w) -  e^{i \tau \Delta} w \right) \notag\\
		&= e^{i \frac{\tau}{2}\Delta} P_N \left (\Phi_B^\tau( e^{i \frac{\tau}{2}\Delta} v) - e^{i \frac{\tau}{2}\Delta} v \right) - e^{i \frac{\tau}{2}\Delta} P_N \left (\Phi_B^\tau( e^{i \frac{\tau}{2}\Delta} w) - e^{i \frac{\tau}{2}\Delta} w \right). 
	\end{align}
	By \cref{stab_S2}, using \cref{prop:stability_L2}, the isometry property of $e^{i t \Delta}$ and Sobolev embedding, we have
	\begin{align*}
		&\left \| \left( \mathcal{S}_2^\tau(v) - e^{i \tau \Delta} v \right) - \left(\mathcal{S}_2^\tau (w) -  e^{i \tau \Delta} w \right) \right \|_{L^2} \\
		&\leq \left \| P_N \left (\Phi_B^\tau( e^{i \frac{\tau}{2}\Delta} v) - e^{i \frac{\tau}{2}\Delta} v \right) - P_N \left (\Phi_B^\tau( e^{i \frac{\tau}{2}\Delta} w) - e^{i \frac{\tau}{2}\Delta} w \right) \right \|_{L^2} \\
		&\leq C(\| V \|_{L^\infty}, \| e^{i \frac{\tau}{2}\Delta} v \|_{L^\infty}, \| e^{i \frac{\tau}{2}\Delta} w \|_{L^\infty}) \| e^{i \frac{\tau}{2}\Delta} v - e^{i \frac{\tau}{2}\Delta} w \|_{L^2} \\
		&\leq C(\| V \|_{L^\infty}, \| v \|_{H^2}, \| e^{i \frac{\tau}{2}\Delta} w \|_{L^\infty}) \| v -  w \|_{L^2}, 
	\end{align*}
	which completes the proof. 
\end{proof}

\subsection{Optimal $L^2$-norm error bound}\label{sec:proof_S2_L2}
\begin{proof}[Proof of \cref{eq:ST_L2} in \cref{thm:ST}]
	Let $ e^n = P_N \psi(t_n) - \psi^n $ for $ 0 \leq n \leq T/\tau $. By the standard Fourier projection error estimates and noting that $ \psi \in C([0, T]; H^4_\text{\rm per}(\Omega)) $, we have
	\begin{equation*}
		\| \psi(t_n) - P_N \psi(t_n) \|_{L^2} \lesssim h^4. 
	\end{equation*}
	Then the proof reduces to the estimate of $ e^n $. For $  0 \leq n \leq T/\tau -1 $, 
	\begin{align}\label{eq:error_eq_S2}
		e^{n+1} 
		&= P_N \psi(t_{n+1}) - \psi^{n+1} \notag \\
		&= P_N \psi(t_{n+1}) - \mathcal{S}_2^\tau (P_N\psi(t_n)) + \mathcal{S}_2^\tau (P_N\psi(t_n)) - \mathcal{S}_2^\tau(\psi^n)\notag \\
		&= e^{i \tau \Delta} e^n + Z^n + \mcalL^n, 
	\end{align}
	where $\mcalL^n$ is defined in \cref{eq:barE_def} and
	\begin{align}
		Z^n &= \left( \mathcal{S}_2^\tau(P_N \psi(t_n)) - e^{i \tau \Delta} P_N \psi(t_n) \right) - \left(\mathcal{S}_2^\tau (\psi^n) -  e^{i \tau \Delta} \psi^n \right). \label{eq:Zn}
%		\mcalL^n &= P_N \psi(t_{n+1}) - \mathcal{S}_2^\tau(P_N\psi(t_n)). \label{eq:Ln}
	\end{align}
	Iterating \cref{eq:error_eq_S2}, we have, for $ 0 \leq n \leq T/\tau -1 $, 
	\begin{equation}\label{eq:error_eq}
		e^{n+1} = e^{i(n+1)\tau\Delta} e^0 + \sum_{k=0}^{n} e^{i (n-k)\tau \Delta} \left( Z^k + \mcalL^k \right). 
	\end{equation}
	For $ Z^n $ in \cref{eq:Zn}, by \cref{prop:stability_ST_L2} and the boundedness of $e^{it\Delta}$ and $P_N$, we have
	\begin{align*}
		\| Z^n \|_{L^2} 
		&\leq C(\| V \|_{L^\infty}, \| \psi(t_n) \|_{H^2}, \| e^{i \frac{\tau}{2}\Delta} \psi^n \|_{L^\infty}) \tau \| P_N \psi(t_n) - \psi^n \|_{L^2} \\
		&\leq C(M_4, \| e^{i \frac{\tau}{2}\Delta} \psi^n \|_{L^\infty}) \tau \| e^n \|_{L^2}, \quad 0 \leq n \leq T/\tau -1, 
	\end{align*}
	which, together with the isometry property of $e^{it\Delta}$ and the triangle inequality, yields
	\begin{equation}\label{eq:Zn_sum}
		\left \| \sum_{k=0}^{n} e^{i (n-k)\tau \Delta} Z^k \right \|_{L^2} \leq \sum_{k=0}^n \| Z^k \|_{L^2} \leq C_3 \tau \sum_{k=0}^n \| e^k \|_{L^2}, \quad 0 \leq n \leq T/\tau -1, 
	\end{equation}
	where $C_3$ depends on $\max_{0 \leq k \leq n} \| e^{i \frac{\tau}{2}\Delta} \psi^k \|_{L^\infty} $ and $M_4$. It follows from \cref{prop:dominant_error_ST_L2} and the isometry property of $e^{it\Delta}$ that
	\begin{align}\label{eq:Ln_sum}
		\left\| \sum_{k=0}^{n} e^{i (n-k)\tau \Delta} \mcalL^k \right\|_{L^2} 
		&\lesssim n \tau (\tau^2 + h^4) + \left\| \sum_{k=0}^{n} e^{i (n-k)\tau \Delta} \mcalL^k_2 \right\|_{L^2} \notag \\
		&= n \tau (\tau^2 + h^4) + \left\| \tau^3 \sum_{k=0}^{n} e^{-ik \tau \Delta} \Delta^2 \eta^k  \right\|_{L^2}, 
	\end{align}
	where
	\begin{equation}\label{eq:eta}
		\eta^n = \int_0^1 \ker(\theta) e^{i (1-\theta)\tau\Delta} P_N B(e^{i\theta \tau \Delta} P_N \psi(t_n)) \rmd \theta \in X_N, \quad 0 \leq 0 \leq n \leq T/\tau -1. 
	\end{equation}
	Moreover, define
	\begin{equation}\label{eq:Gn}
		\mathcal{G}^n = \tau^3 \sum_{k=0}^{n} e^{-ik \tau \Delta} \Delta^2 \eta^k , \quad 0 \leq n \leq T/\tau -1. 
	\end{equation}
	From \cref{eq:error_eq}, using \cref{eq:Zn_sum,eq:Ln_sum} and recalling \cref{eq:Gn}, we have
	\begin{equation}\label{eq:error_eq_L2_S2}
		\| e^{n+1} \|_{L^2} \lesssim \tau^2 + h^4 + C_3\tau \sum_{k=0}^n \| e^k \|_{L^2} +\left\| \mathcal{G}^n \right\|_{L^2}.
	\end{equation}
	%	Let
	%	\begin{equation}\label{eq:def_v}
		%		v^n = \int_0^1 \ker(\theta) f^n(\theta \tau) \rmd \theta \in X_N, \quad 0 \leq n \leq T/\tau, 
		%	\end{equation}
	%	and let
	%	\begin{equation}
		%		\mathcal{J}^n = \sum_{k=0}^{n} e^{-ik\tau \Delta} \mcalL^k_2 = \tau^3 \sum_{k=0}^{n} e^{-ik\tau \Delta} \Delta^2 v^k. 
		%	\end{equation}
	We shall use similar techniques as before to analyze $\mathcal{G}^n$. Note that
	\begin{equation}
		\mathcal{G}^n = \tau ^3\sum_{k=0}^n \sum_{l \in \mathcal{T}_N} e^{ i k \tau \mu_l^2} \mu_l^4 \widehat{\eta^k_l} e^{i\mu_l(x-a)}, 
	\end{equation}
	%	where
	%	\begin{equation}
		%		\widehat{v^k_l} = \int_0^1 \ker(\theta \tau) e^{-i (1-\theta)\tau \mu_l^2} \widehat{\phi^k_l}(\theta \tau) \rmd \theta, \quad l \in \mathcal{T}_N, \quad 0 \leq k \leq n \leq T/\tau -1.  
		%	\end{equation}
	from which, exchanging the order of summation, using summation by parts and recalling \cref{eq:Snl}, we obtain
	\begin{align}\label{eq:summation_by_parts_S2}
		\mathcal{G}^n
		&= \tau^3 \sum_{l \in \mathcal{T}_N} \mu_l^4 e^{i\mu_l(x-a)} \sum_{k=0}^n e^{i k \tau \mu_l^2} \widehat{\eta^k_l} \notag \\
		&= \tau^3 \sum_{l \in \mathcal{T}_N} \mu_l^4 e^{i\mu_l(x-a)} \left( S_{n, l}  \eta^n_l - \sum_{k=0}^{n-1} S_{k, l} \left(\widehat{\eta^{k+1}_l} - \widehat{\eta^k_l}\right) \right). 
	\end{align}
	Noting that $\mu_l = 0$ when $l=0$, similar to \cref{eq:rS}, we have
	\begin{equation}\label{eq:muS}
		| \tau \mu_l^2 S_{n, l} | \lesssim 1, \quad l \in \mathcal{T}_n, \quad 0 \leq n \leq T/\tau -1, 
	\end{equation}
	where the constant is independent of $n$ and $l$. 
	From \cref{eq:summation_by_parts_S2}, using Parseval's identity and \cref{eq:muS}, we get 
	\begin{align}\label{eq:Gn_est}
		\| \mathcal{G}^n \|_{L^2}^2 
		&= (b-a)\tau^6 \sum_{l \in \mathcal{T}_N} \mu_l^8 \left| S_{n, l}  \eta^n_l - \sum_{k=0}^{n-1} S_{k, l} \left(\widehat{\eta^{k+1}_l} - \widehat{\eta^k_l}\right) \right|^2 \notag\\
		&\lesssim \tau^6 \sum_{l \in \mathcal{T}_N} \mu_l^8 S_{n, l}^2 \left|\widehat{\eta^n_l} \right|^2 +\tau^6 \sum_{l \in \mathcal{T}_N} \mu_l^8 \left|\sum_{k=0}^{n-1} S_{k, l} \left(\widehat{\eta^{k+1}_l} - \widehat{\eta^k_l}\right) \right|^2 \notag\\
		&\lesssim \tau^4 \sum_{l \in \mathcal{T}_N} \mu_l^4 \left|\widehat{\eta^n_l} \right|^2 + \tau^6 \sum_{l \in \mathcal{T}_N} \mu_l^8 \sum_{k_1=0}^{n-1} \left|S_{k_1, l}\right|^2 \sum_{k_2=0}^{n-1} \left|\widehat{\eta^{k_2+1}_l} - \widehat{\eta^{k_2}_l}\right|^2 \notag\\
		&\lesssim \tau^4 \| \eta^n \|_{H^2}^2 + \tau^4 \sum_{k_1=0}^{n-1} \sum_{k_2=0}^{n-1} \sum_{l \in \mathcal{T}_N} \mu_l^4 \left|\widehat{\eta^{k_2+1}_l} - \widehat{\eta^{k_2}_l}\right|^2 \notag\\
		&\lesssim \tau^4 \| \eta^n \|_{H^2}^2 + n\tau^4 \sum_{k=0}^{n-1} \| \eta^{k+1} - \eta^{k} \|_{H^2}^2. 
	\end{align}
	Recalling \cref{eq:eta} and using \cref{lem:B_H2} and the boundedness of $e^{it\Delta}$ and $P_N$, we have  
	\begin{equation}\label{eq:eta_n_H2}
		\| \eta^n \|_{H^2} \lesssim \sup_{0 \leq \theta \leq 1} \| B(e^{i \theta \tau \Delta}P_N \psi(t_n)) \|_{H^2} \leq C(M_4), \quad 0 \leq n \leq T/\tau - 1. 
	\end{equation}
	Moreover, for $0 \leq k \leq T/\tau -2$, we have
	\begin{equation}\label{eq:diff_eta}
		\eta^{k+1} - \eta^{k} =  \int_0^1 \ker(\theta) e^{i (1-\theta)\tau\Delta} P_N \left( B(e^{i \theta \tau \Delta} P_N \psi(t_{k+1})) - B(e^{i \theta \tau \Delta} P_N \psi(t_{k})) \right) \rmd \theta. 
	\end{equation}
	From \cref{eq:diff_eta}, using \cref{lem:diffB_H2} and the boundedness of $e^{it\Delta}$ and $P_N$, we have
	\begin{align*}
		\| v^{k+1} - v^{k} \|_{H^2} 
		&\lesssim \int_0^1 \| B(e^{i \theta \tau \Delta} P_N \psi(t_{k+1})) - B(e^{i \theta \tau \Delta} P_N \psi(t_{k})) \|_{H^2} \rmd \theta \\
		&\lesssim \int_0^1 \| e^{i \theta \tau \Delta} P_N (\psi(t_{k+1}) - \psi(t_{k})) \|_{H^2} \rmd \theta \\
		&\lesssim \| \psi(t_{k+1}) - \psi(t_{k}) \|_{H^2} \lesssim \tau \| \partial_t \psi \|_{L^\infty([t_k, t_{k+1}]; H^2)}, \quad 0 \leq k \leq \frac{T}{\tau} -2,  
	\end{align*}
	which together with \cref{eq:eta_n_H2}, inserted into \cref{eq:Gn_est}, yields
	\begin{equation*}
		\| \mathcal{G}^n \|_{L^2}^2 \lesssim \tau^4 + n^2 \tau^6 \lesssim \tau^4, \quad 0 \leq n \leq T/\tau -1. 
	\end{equation*}
	Recalling \cref{eq:error_eq_L2_S2}, we have 
	\begin{equation}\label{eq:error_eq_3}
		\| e^{n+1} \|_{L^2} \lesssim \tau^2 + h^4 + C_3 \tau \sum_{k=0}^{n} \| e^k \|_{L^2}, \quad 0 \leq n \leq T/\tau -1. 
	\end{equation}
	where $ C_3 $ depends on $ \max_{0 \leq k \leq n} \| e^{i \frac{\tau}{2}\Delta} \psi^k \|_{L^\infty} $ and $M_4$. The proof can be competed by applying discrete Gronwall's inequality to \cref{eq:error_eq_3} and using the standard induction argument with the inverse estimate 
	\begin{align}\label{eq:inverse_estimate}
		\| e^{i \frac{\tau}{2} \Delta} \psi^{k} \|_{L^\infty} 
		&\leq \|e^{i \frac{\tau}{2} \Delta} (\psi^{k} - P_N \psi(t_{k}) )\|_{L^\infty} + \| e^{i \frac{\tau}{2}\Delta} P_N \psi(t_{k}) \|_{L^\infty} \notag \\
		&\leq C_\text{inv} h^{-\frac{d}{2}} \|e^{i \frac{\tau}{2} \Delta} (\psi^{k} - P_N \psi(t_{k}) )\|_{L^2} +  C\| e^{i \frac{\tau}{2}\Delta} P_N \psi(t_{k}) \|_{H^2} \notag \\
		&= C_\text{inv} h^{-\frac{d}{2}} \| e^{k} \|_{L^2} + C(M_4)
	\end{align}
	to control the constant $C_3$ in \cref{eq:error_eq_3}. The $H^1$-norm error bound in \cref{eq:ST_L2} can be obtained similarly to \cref{eq:H1_bound_inverse} by using inverse inequalities and the time step size restriction $\tau \lesssim h^2$, and we omit the details here. Thus the proof is completed. 
	%	To complete the proof, we adopt the standard induction arguments to prove, for $0 \leq n \leq T/\tau$, 
	%	\begin{equation}\label{eq:assumption}
		%		\| P_N \psi(t_n) - \psi^n \|_{L^2} \lesssim \tau^2 + h^4, \quad \| e^{i \frac{\tau}{2}\Delta} \psi^n \|_{L^\infty} \leq 1+M_0, 
		%	\end{equation}
	%	where $M_0$ is defined by 
	%	\begin{equation}
		%		\| e^{i \frac{\tau}{2}\Delta} P_N \psi(t_{n}) \|_{L^\infty} \leq c \| e^{i \frac{\tau}{2}\Delta} P_N \psi(t_{n}) \|_{H^2} \leq c\| \psi(t_{n}) \|_{H^2} \leq cM_4 =:M_0, \quad 0 \leq n \leq T/\tau. 
		%	\end{equation}
	%	Assume that \cref{eq:assumption} holds for $ 0 \leq n \leq m \leq T/\tau-1$. Applying discrete Gronwall's inequality to \cref{eq:error_eq_3} with $n = m$, we have
	%	\begin{equation}\label{eq:induction_1}
		%		\| e^{m+1} \|_{L^2} \lesssim \tau^2 + h^4, 
		%	\end{equation}
	%	which implies, when $0 \leq h \leq h_0$ with $h_0>0$ sufficiently small and $\tau \lesssim h^2$,  
	%	\begin{align}\label{eq:induction_2}
		%		\| e^{i \frac{\tau}{2} \Delta} \psi^{m+1} \|_{L^\infty} 
		%		&\leq \|e^{i \frac{\tau}{2} \Delta} (\psi^{m+1} - P_N \psi(t_{m+1}) )\|_{L^\infty} + \| e^{i \frac{\tau}{2}\Delta} P_N \psi(t_{m+1}) \|_{L^\infty} \notag \\
		%		&\leq C_\text{inv} h^{-\frac{d}{2}} \|e^{i \frac{\tau}{2} \Delta} (\psi^{m+1} - P_N \psi(t_{m+1}) )\|_{L^2} + M_0 \notag\\
		%		&= C_\text{inv} h^{-\frac{d}{2}} \| e^{m+1} \|_{L^2} + M_0 \leq 1+M_0. 
		%	\end{align}
	%	Combing \cref{eq:induction_1,eq:induction_2}, we prove \cref{eq:assumption} for $n = m+1$ and thus complete the proof by the induction arguments. 
\end{proof}

\subsection{Optimal $ H^1 $-norm error bound}
Then we shall show the optimal $H^1$-norm error bound \cref{eq:ST_H1} for the STFS method under the assumptions that $ V \in H^3_\text{per}(\Omega) $, $ \sigma \geq 3/2 $ and $\psi \in C([0, T]; H_\text{per}^5(\Omega)) \cap C^1([0, T]; H^3(\Omega)) \cap C^2([0, T]; H^1(\Omega))$. All the results in this subsection hold trivially for $\sigma = 1$, and we shall omit this case for simplicity. We define a constant 
\begin{equation*}
	M_5 = \max\left \{ \| V \|_{H^3}, \| \psi \|_{L^\infty([0, T]; H^5(\Omega))}, \| \partial_t \psi \|_{L^\infty([0, T]; H^3(\Omega))}, \| \partial_{tt} \psi \|_{L^\infty([0, T]; H^1(\Omega))} \right \}. 
\end{equation*}

We first show the higher order counterparts of \cref{lem:f_H2,lem:B3,lem:diff_dB,lem:dB_H2,lem:dG_L2}. 
\begin{lemma}\label{lem:f_H3} 
	When $ \sigma \geq 3/2 $, for any $ v \in H^3(\Omega) $, we have
	\begin{equation*}
		\| f(|v|^2) \|_{H^3} \leq C(\| v \|_{H^3}), \quad \| G(v) \|_{H^3} \leq C(\| v \|_{H^3}). 
	\end{equation*}
\end{lemma}

\begin{proof}
	By some elementary calculation, we have the point-wise estimate
	\begin{align}\label{par_xxx_f}
		| \partial_{xxx} f(|v|^2) | 
		&\lesssim \left (\left |f^{\prime\prime\prime}(|v|^2)|v|^3\right | \left |f^{\prime\prime}(|v|^2) |v| \right | \right ) |\partial_x v|^3 \notag \\
		&\quad + \left (\left |f^{\prime\prime}(|v|^2)|v|^2 \right | + \left |f^\prime(|v|^2)\right |\right ) |\partial_x v| |\partial_{xx} v| \notag \\
		&\quad + \left |f^\prime(|v|^2) |v| \right | |\partial_{xxx} v|,  
	\end{align}
	where $ f^{\prime\prime\prime}(|z|^2)|z|^3 $, $f^{\prime\prime}(|z|^2)|z|$ and $f^{\prime\prime}(|z|^2)|z|^2$ with $z \in \C$ are defined as $0$ when $z = 0$. From \cref{par_xxx_f}, using Sobolev embedding and H\"older's inequality and noting \cref{fbound}, \cref{f_bound2} and 
	\begin{equation}
		\left |f^{\prime\prime\prime}(|z|^2)|z|^3\right | + \left |f^{\prime\prime}(|z|^2) |z| \right| \lesssim |z|^{2\sigma-3}, \quad \sigma \geq \frac{3}{2}, \quad  z \in \C, 
	\end{equation}
	we have
	\begin{align*}
		\| \partial_{xxx} f(|v|^2) \|_{L^2} 
		&\lesssim \| v \|_{L^\infty}^{2\sigma - 3} \| \partial_x v\|_{L^6}^3 + \| v \|_{L^\infty}^{2\sigma - 2} \| \partial_x \|_{L^4} \| \partial_{xx} v\|_{L^4} + \| v \|_{L^\infty}^{2\sigma - 1}\| \partial_{xxx} v \|_{L^2} \\
		&\leq C(\| v \|_{H^3}),  
	\end{align*}
	which, combined with \cref{lem:f_H2}, yields $\| f(|v|^2) \|_{H^3} \leq C(\| v \|_{H^3}) $. Similarly, we have $ \| G(v) \|_{H^3} \leq C(\| v \|_{H^3}) $ and complete the proof. 
\end{proof}
%	\begin{proof}
	%		When $ \sigma > 3/2 $, $ f(|z|^2) $ and $ G(z) $ are $ C^3 $ in the real sense and the conclusion can be proved by elementary calculation. When $ \sigma = 3/2 $, the conclusion can be proved by the standard regularization technique, which can be found in many textbooks about Sobolev spaces (see, e.g., \cite{lieb2001}). 
	%	\end{proof}

\begin{lemma}\label{lem:B4}
	Under the assumptions $ V \in H^3(\Omega) $ and $ \sigma \geq 3/2 $, for any $ v, w \in H^3(\Omega) $ such that $ \| v \|_{H^3} \leq M $, $ \| w \|_{H^3} \leq M $, we have
	\begin{equation}\label{lem:diffB_H3}
		\| B(v) - B(w) \|_{H^3} \leq C(\| V \|_{H^3}, M) \| v - w \|_{H^3}. 
	\end{equation}
	In particular, when $w = 0$, we have
	\begin{equation}\label{lem:B_H3}
		\| B(v) \|_{H^3} \leq C(\| V \|_{H^3}, M). 
	\end{equation}
\end{lemma}

\begin{proof}
	With \cref{lem:f_H3}, the proof follows the same way as the proof of \cref{lem:B3} and we shall omit it for brevity. 
\end{proof}

\begin{lemma}\label{lem:diff_dB_H1}
	Under the assumptions $ V \in W^{1,4}(\Omega) $ and $ \sigma \geq 3/2 $, for any $ v_j, w_j \in H^2(\Omega) $ satisfying $ \| v_j \|_{H^2} \leq M $ and $ \| w_j \|_{H^2} \leq M $ with $ j = 1, 2 $, we have 
	\begin{equation*}
		\| dB(v_1)[w_1] - dB(v_2)[w_2] \|_{H^1} \leq C(\| V \|_{W^{1, 4}}, M) ( \| v_1 - v_2 \|_{H^1} + \| w_1 - w_2 \|_{H^1}). 
	\end{equation*}
\end{lemma}

\begin{proof}
	By \cref{lem:diff_dB}, it suffices to prove 
	\begin{align}\label{reduced}
		&\left \| \partial_x (dB(v_1)[w_1] - dB(v_2)[w_2]) \right \|_{L^2} \notag \\ 
		&\leq C(\| V \|_{W^{1, 4}}, M) ( \| v_1 - v_2 \|_{H^1} + \| w_1 - w_2 \|_{H^1}). 
	\end{align}
	Recalling \cref{eq:dB_def}, we have
	\begin{align}\label{Y_decomp}
		&\partial_x \left (dB(v_1)[w_1] - dB(v_2)[w_2] \right ) \notag \\
		&= -i \partial_x (V(w_1 - w_2)) -i(1+\sigma) \partial_x \left ( f(|v_1|^2)w_1 - f(|v_2|^2)w_2 \right ) \notag \\
		&\quad - i  \partial_x \left ( G(v_1)w_1 - G(v_2)w_2 \right ) = : Y_1 + Y_2 + Y_3. 
	\end{align}
	For $Y_1$, by Sobolev embedding and H\"older's inequality, we have
	\begin{equation}\label{Y1}
		\| Y_1 \|_{L^2} \leq \| \partial_x V \|_{L^4}\| w_1 - w_2 \|_{L^4} + \| V \|_{L^\infty} \| w_1 - w_2 \|_{H^1} \lesssim \| V \|_{W^{1, 4}} \| v-w \|_{H^1}. 
	\end{equation}
	For $Y_2$, we have
	\begin{align}\label{Y2_decomp}
		\| Y_2 \|_{L^2} 
		&\lesssim \| f^\prime(|v_1|^2)v_1 w_1 \partial_x \overline{v_1} - f^\prime(|v_2|^2)v_2 w_2 \partial_x \overline{v_2} \|_{L^2} \notag \\
		&\quad  + \| f^\prime(|v_1|^2)\overline{v_1} w_1 \partial_x v_1 - f^\prime(|v_2|^2)\overline{v_2} w_2 \partial_x v_2 \|_{L^2} \notag \\
		&\quad + \| f(|v_1|^2) \partial_x w_1 - f(|v_2|^2) \partial_x w_2 \|_{L^2} =: \| Y_2^1 \|_{L^2} + \| Y_2^2 \|_{L^2} + \| Y_2^3 \|_{L^2}. 
	\end{align}
	One can easily check that when $\sigma \geq 3/2$, for $z_1, z_2 \in \C$ satisfying $|z_1| \leq M_0$ and $|z_2| \leq M_0$, 
	\begin{equation}\label{diff_fz2z}
		\begin{aligned}
			&\left |f^\prime(|z_1|^2)z_1 - f^\prime(|z_2|^2)z_2 \right | \lesssim M_0^{2\sigma - 2} |z_1 - z_2|, \\ 	
			&\left |f^\prime(|z_1|^2)\overline{z_1} - f^\prime(|z_2|^2)\overline{z_2} \right | \lesssim M_0^{2\sigma - 2} |z_1 - z_2|. 
		\end{aligned}
	\end{equation}
	Using \cref{diff_fz2z}, Sobolev embedding and H\"older's inequality and noting that $|f^\prime(|z|^2)z| + |f^\prime(|z|^2)\overline{z}| \lesssim |z|^{2\sigma - 1} $ for all $ z \in \C$, we have, for $Y^1_2$ in \cref{Y2_decomp}, 
	\begin{align}\label{Y21}
		\| Y_2^1 \|_{L^2} 
		&\leq \| f^\prime(|v_1|^2)v_1 - f^\prime(|v_2|^2)v_2 \|_{L^4}\| w_1 \|_{L^\infty} \| \partial_x \overline{v_1} \|_{L^4} \notag \\
		&\quad + \| f^\prime(|v_2|^2)v_2 \|_{L^\infty} \| w_1 -  w_2 \|_{L^4} \| \partial_x \overline{v_1} \|_{L^4} \notag \\
		&\quad + \| f^\prime(|v_2|^2)v_2 \|_{L^\infty} \| w_2 \|_{L^\infty} \| \partial_x \overline{v_1} - \partial_x \overline{v_2} \|_{L^2} \notag \\
		&\leq C(M) \left (\| v_1 - v_2 \|_{L^4} + \| w_1 - w_2 \|_{L^4} + \| \partial_x v_1 - \partial_x v_2 \|_{L^2} \right ) \notag \\
		&\leq C(M) ( \| v_1 - v_2 \|_{H^1} + \| w_1 - w_2 \|_{H^1}). 
	\end{align}
	Similar to \cref{Y21}, we have, for $Y^2_2$ in \cref{Y2_decomp},
	\begin{equation}\label{Y22}
		\| Y_2^2 \|_{L^2} \leq C(M) ( \| v_1 - v_2 \|_{H^1} + \| w_1 - w_2 \|_{H^1}). 
	\end{equation}
	For $Y_2^3$ in \cref{Y2_decomp}, we have, by using \cref{eq:diff_f}, Sobolev embedding and H\"older's inequality and noting \cref{fbound}, 
	\begin{align}\label{Y23}
		\| Y_2^3 \|_{L^2} 
		&\leq \| f(|v_1|^2) - f(|v_2|^2) \|_{L^4} \| \partial_x w_1 \|_{L^4} + \| f(|v_2|^2) \|_{L^\infty} \| \partial_x w_1 - \partial_x w_2 \|_{L^2} \notag \\
		&\leq C(M) \left (\| v_1 - v_2 \|_{L^4} + \| \partial_x w_1 - \partial_x w_2 \|_{L^2} \right ) \notag \\
		&\leq C(M) ( \| v_1 - v_2 \|_{H^1} + \| w_1 - w_2 \|_{H^1}). 
	\end{align}
	Combining \cref{Y21,Y22,Y23} and noting \cref{Y2_decomp}, we have 
	\begin{equation}\label{Y2}
		\| Y_2 \|_{L^2} \leq C(M) ( \| v_1 - v_2 \|_{H^1} + \| w_1 - w_2 \|_{H^1}). 
	\end{equation}
	Similar to the estimate of $Y_2$, we have
	\begin{equation}\label{Y3}
		\| Y_3 \|_{L^2} \leq C(M) ( \| v_1 - v_2 \|_{H^1} + \| w_1 - w_2 \|_{H^1}). 
	\end{equation}
	Inserting \cref{Y1,Y2,Y3} into \cref{Y_decomp} yields \cref{reduced} and completes the proof. 
\end{proof}

\begin{lemma}\label{lem:dB_H3}
	Under the assumptions $ V \in H^3(\Omega) $ and $ \sigma \geq 3/2 $, for any $ v, w \in H^3(\Omega) $ satisfying $ \| v \|_{H^3} \leq M $, $ \| w \|_{H^3} \leq M $, we have
	\begin{equation*}
		\| dB(v)[w] \|_{H^3} \leq C(\| V \|_{H^3}, M). 
	\end{equation*}
\end{lemma}

\begin{proof}
	Recalling \cref{eq:dB_def} and using \cref{lem:f_H3} and the algebra property of $H^3(\Omega)$, we obtain the desired result. 
\end{proof}

\begin{lemma}\label{lem:dG_H1}
	When $ \sigma \geq 1 $, for any $ v \in H^2(\Omega) $ and $w \in H^1(\Omega)$ satisfying $ \| v \|_{H^2} \leq M $, we have
	\begin{equation*}
		\| dG(v)[w] \|_{H^1} \leq C(M) \| w \|_{H^1}. 
	\end{equation*}
\end{lemma}

\begin{proof}
	By \cref{lem:dG_L2}, it suffices to prove
	\begin{equation}\label{reduced_dG}
		\| \partial_x \left (dG(v)[w]\right ) \|_{L^2} \leq C(M) \| w \|_{H^1}. 
	\end{equation}
	Recalling \cref{eq:dG_def}, we have the point-wise estimate
	\begin{align}\label{eq:par_dG}
		\left |\partial_x \left (dG(v)[w]\right ) \right | 
		&\lesssim \left ( \left |f^{\prime\prime\prime}(|v|^2) |v|^4 \right | +  \left |f''(|v|^2) |v|^2 \right| + \left | f^\prime(|v|^2) \right | \right )|\partial_x v||w| \notag\\
		&\quad + \left ( \left |f^{\prime\prime}(|v|^2)|v|^3\right | + \left |f^\prime(|v|^2)v \right | \right )|\partial_x w|. 
	\end{align} 
	From \cref{eq:par_dG}, noting \cref{f_bound2} and 
	\begin{equation}
		\left |f^{\prime\prime\prime}(|z|^2)|z|^4 \right | + \left |f^{\prime\prime}(|z|^2)|z|^2 \right | + \left | f^\prime(|z|^2) \right | \lesssim |z|^{2\sigma-2}, \quad z \in \C, \quad \sigma \geq 1, 
	\end{equation}
	and using Sobolev embedding and H\"older's inequality, we have
	\begin{equation*}
		\| \partial_x \left (dG(v)[w]\right ) \|_{L^2} \lesssim \| v \|_{L^\infty}^{2\sigma - 2} \| \partial_x v \|_{L^4}\| w \|_{L^4} + \| v \|_{L^\infty}^{2\sigma -1} \| \partial_x w\|_{L^2} \leq C(M) \| w \|_{H^1}, 
	\end{equation*}
	which proves \cref{reduced_dG} and completes the proof. 
\end{proof}

With \cref{lem:f_H3,lem:B4,lem:diff_dB_H1,lem:dB_H3,lem:dG_H1} and the product estimate
\begin{equation}\label{eq:product_est}
	\| vw \|_{H^1} \lesssim \| v \|_{W^{1, 4}} \| w \|_{H^1}, \quad v \in W^{1, 4}(\Omega), \quad w \in H^1(\Omega), 
\end{equation}
similar to the proof of \cref{prop:dominant_error_ST_L2}, we can obtain the following local error decomposition. 
\begin{proposition}\label{prop:dominant_error_ST_H1}
	Assuming that $ V \in H^3_\text{per}(\Omega) $, $ \sigma \geq 3/2 $ and $\psi \in C([0, T]; H_\text{per}^5(\Omega)) \cap C^1([0, T]; H^3(\Omega)) \cap C^2([0, T]; H^1(\Omega))$, for the local truncation error defined in \cref{eq:barE_def}, we have
	\begin{equation*}
		\mcalL^n = \mcalL_1^n + \mcalL_2^n, \quad 0 \leq  n \leq T/\tau -1, 
	\end{equation*}
	where 
	\begin{equation*}
		\| \mcalL_1^n \|_{H^1} \lesssim \tau^3 + \tau h^4, \quad \mcalL_2^n = \tau^3 \Delta^2 \int_0^1 \ker(\theta) e^{i (1-\theta)\tau\Delta} P_N B(e^{i\theta \tau \Delta} P_N \psi(t_n)) \rmd \theta. 
	\end{equation*}
\end{proposition}

\begin{proposition}\label{prop:stability_ST_H1}
	Let $ v, w \in X_N $ such that $ \| e^{i \frac{\tau}{2}\Delta} w \|_{L^\infty} \leq M $ and $ \| v \|_{H^2} \leq M_1 $ Then we have
	\begin{align*}
		&\left \| \left( \mathcal{S}_2^\tau(v) - e^{i \tau \Delta} v \right) - \left(\mathcal{S}_2^\tau (w) -  e^{i \tau \Delta} w \right) \right \|_{H^1} \\
		&\leq C(\| V \|_{W^{1, 4}}, \| v \|_{H^2}, \| e^{i \frac{\tau}{2} \Delta} w \|_{L^\infty}) \tau \| v - w \|_{H^1}. 
	\end{align*}
\end{proposition}
\begin{proof}
	Recalling \cref{stab_S2}, using \cref{prop:stability_LT_H1}, the isometry property of $e^{i t \Delta}$ and Sobolev embedding, we have
	\begin{align*}
		&\left \| \left( \mathcal{S}_2^\tau(v) - e^{i \tau \Delta} v \right) - \left(\mathcal{S}_2^\tau (w) -  e^{i \tau \Delta} w \right) \right \|_{H^1} \\
		&\leq \left \| P_N \left (\Phi_B^\tau( e^{i \frac{\tau}{2}\Delta} v) - e^{i \frac{\tau}{2}\Delta} v \right) - P_N \left (\Phi_B^\tau( e^{i \frac{\tau}{2}\Delta} w) - e^{i \frac{\tau}{2}\Delta} w \right) \right \|_{H^1} \\
		&\leq C(\| V \|_{W^{1, 4}}, \| e^{i \frac{\tau}{2}\Delta} v \|_{L^\infty}, \| e^{i \frac{\tau}{2}\Delta} w \|_{L^\infty}, \| e^{i \frac{\tau}{2}\Delta} v \|_{H^2}) \| e^{i \frac{\tau}{2}\Delta} v - e^{i \frac{\tau}{2}\Delta} w \|_{H^1} \\
		&\leq C(\| V \|_{W^{1, 4}}, \| v \|_{H^2}, \| e^{i \frac{\tau}{2}\Delta} w \|_{L^\infty}) \| v -  w \|_{H^1}, 
	\end{align*}
	which completes the proof. 
\end{proof}	

By \cref{prop:dominant_error_ST_H1,prop:stability_ST_H1}, we can establish the proof for \cref{eq:ST_H1} in a manner analogous to the proof of \cref{eq:ST_L2} in \cref{sec:proof_S2_L2}. To maintain brevity, we will not detail this process here.

\section{Numerical results}\label{sec:numerical experiments}
In this section, we present some numerical results for the NLSE with either low regularity potential or nonlinearity. In the following, we fix $ \Omega = (-16, 16) $, $ T=1 $ and $ d=1 $. Here, we only focus on temporal errors. Standard convergence orders of the Fourier spectral method can be observed as in \cite{bao2023_EWI}, and we omit them for brevity. To quantify the error, we introduce the following error functions:
\begin{align*}
	e_{L^2}(t_n) := \| \psi(\cdot, t_n) - \psi^n  \|_{L^2}, \quad  e_{H^1}(t_n) := \| \psi(\cdot, t_n) - \psi^n \|_{H^1}, \quad 0 \leq n \leq T/\tau. 
\end{align*}

\subsection{For the NLSE with low regularity potential}
In this subsection, we only consider the cubic NLSE with low regularity potential and a Gaussian initial datum as
\begin{equation}\label{eq:NLSE_low_reg_poten}
	\begin{aligned}
		&i \partial_t \psi(x, t) = -\Delta \psi(x, t) + V(x)\psi(x, t) - |\psi(x, t)|^2 \psi(x, t), \quad x \in \Omega, \quad t>0, \\
		&\psi(x, 0) = e^{-x^2/2}, \quad x \in \Omega, 
	\end{aligned}
\end{equation}
where $ V $ is chosen as $V_j$ with $ j  = 1, 2, 3, 4 $ defined as
\begin{equation}\label{eq:poten}
	\begin{aligned}
		&V_1(x) = \left\{
		\begin{aligned}
			&-4, &x \in (-2, 2) \\
			&0, &\text{otherwise}
		\end{aligned}
		\right.
		, \quad &&V_2(x) = |x|^{0.76}, \\
		&V_3(x) = |x|^{1.51} \left (1- \frac{x^2}{16^2} \right )^2, \quad &&V_4(x) = |x|^{2.51} \left (1- \frac{x^2}{16^2} \right )^3, 
	\end{aligned}
	\qquad x \in \Omega. 
\end{equation}
Note that the potential functions $V_j (1 \leq j \leq 4)$ defined in \cref{eq:poten} satisfy $ V_1 \in L^\infty(\Omega) $, $V_2 \in W^{1, 4}(\Omega) \cap H^1_\text{per}(\Omega)$, $ V_3 \in H^2_\text{per}(\Omega) $ and $ V_4 \in H^3_\text{per}(\Omega) $. 

We shall test the convergence orders of the LTFS \cref{full_discrete_LT} and the STFS \cref{full_discrete_ST} for the NLSE \cref{eq:NLSE_low_reg_poten} with $V = V_j (1 \leq j \leq 4)$. 
%Here, we only show the results obtained by using the extended Fourier pseudospectral method for spatial discretization, i.e., we consider the LTEFP scheme \cref{LTEFP_scheme} and the STEFP scheme \cref{STEFP_scheme}. The results are almost the same when using the Fourier spectral method for spatial discretization. However, 
We remark here that the results shown in \cref{fig:LT_poten,fig:ST_poten} below cannot be observed if one use the standard Fourier pseudospectral method for spatial discretization since the spatial errors will become dominant and thus hide the temporal errors when $\tau \lesssim h^2$. The 'exact' solutions are computed by the STFS with $ \tau = \tau_\text{e} := 10^{-6} $ and $ h = h_\text{e} := 2^{-9} $. 

We start with the LTFS method and choose $V = V_1$ and $V=V_2$ for optimal first-order $L^2$- and $H^1$-norm error bounds, respectively. The numerical results are shown in \cref{fig:LT_poten}, where the marker on the line corresponding to $h = 2^{-2}/2^k$ is placed at $\tau = \tau_0/4^k$ for $k = 0, 1, \cdots, 4$ to highlight the data points satisfying $\tau \sim h^2$. We can observe that the optimal first-order $L^2$- and $H^1$-norm error bounds are only valid when $\tau \lesssim h^2$, and there is order reduction when $ \tau \gg h^2 $ (approximate half order as observed in the numerical results). This confirms our error bounds in \cref{thm:LT} for the NLSE with low regularity potential and indicates that the step size restriction $\tau \lesssim h^2$ is necessary and optimal. 
\begin{figure}[htbp]
	\centering
	{\includegraphics[width=0.475\textwidth]{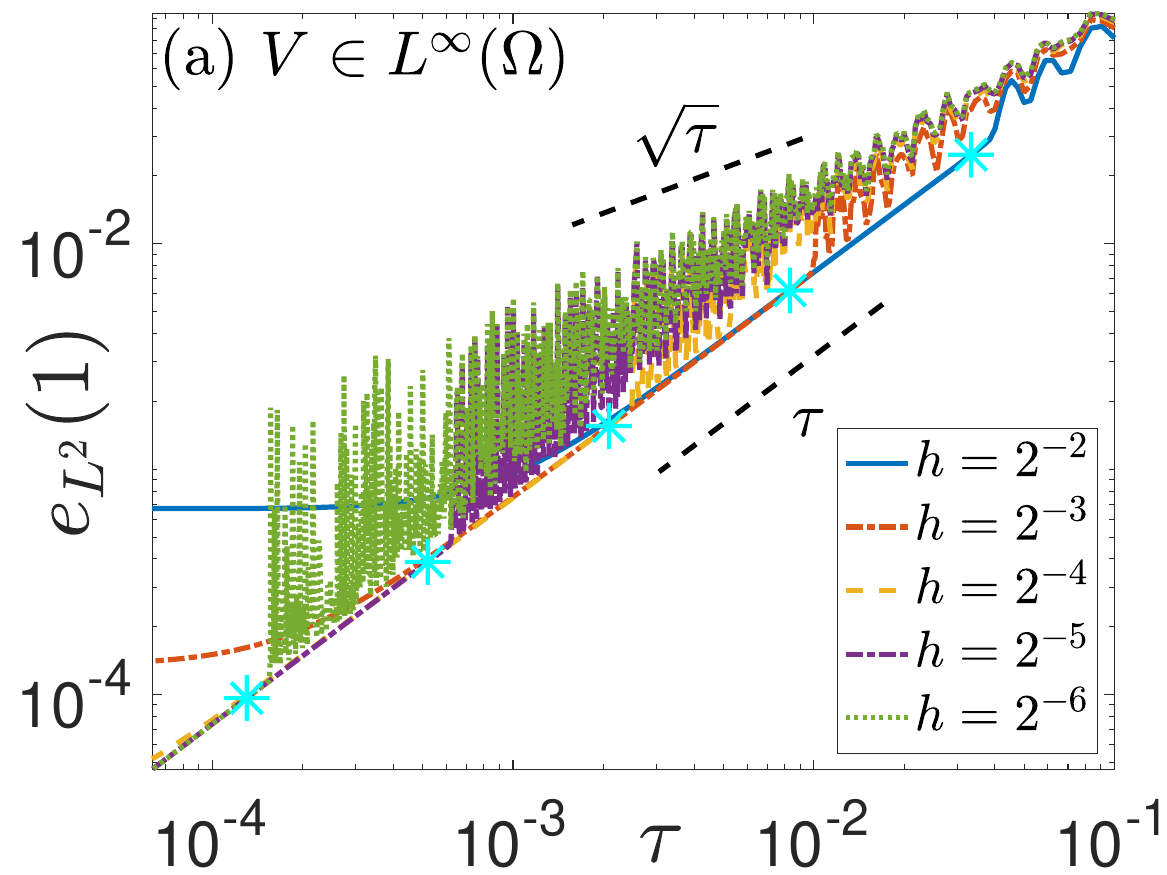}}\hspace{1em}
	{\includegraphics[width=0.475\textwidth]{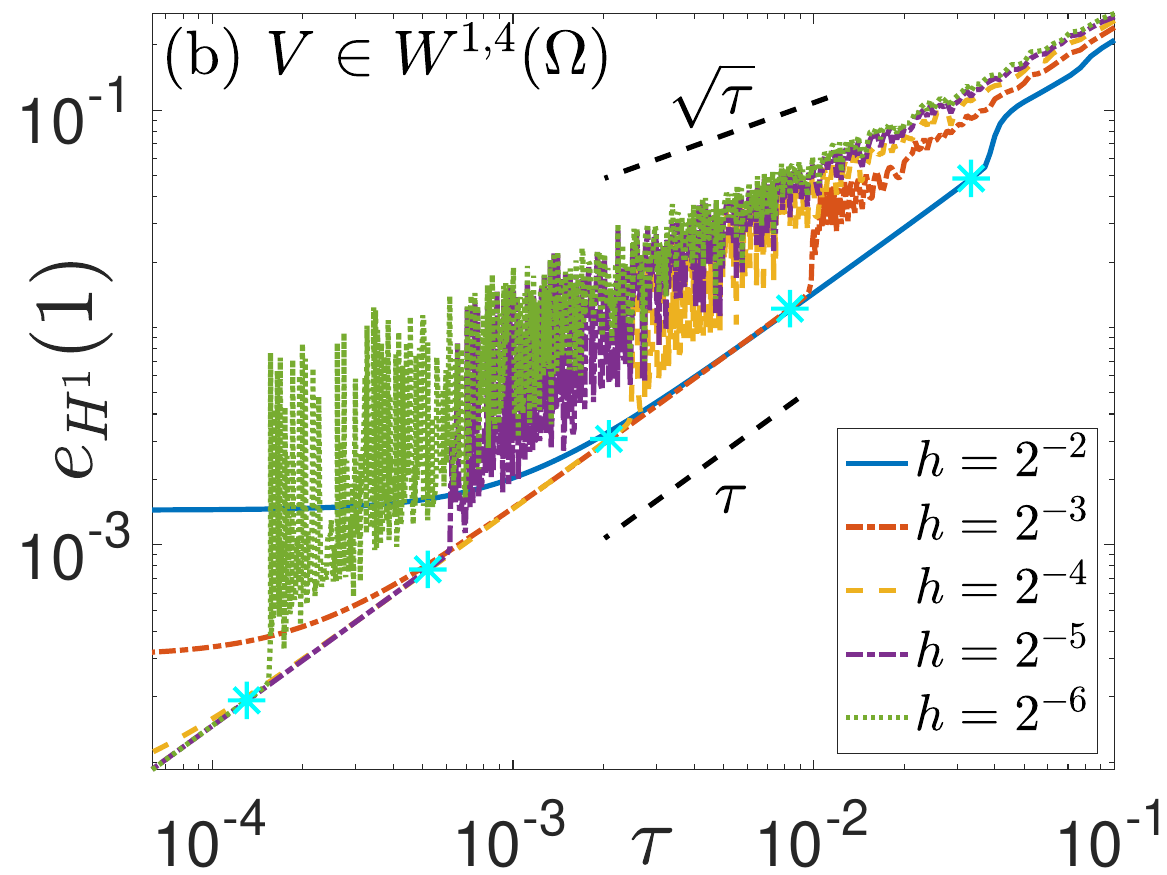}}
	\caption{Temporal errors of the LTFS method with different mesh sizes for \cref{eq:NLSE_low_reg_poten}: (a) $ V=V_1 \in L^\infty(\Omega) $ and (b) $ V = V_2 \in W^{1, 4}(\Omega) \cap H^1_\text{per}(\Omega) $.}
	\label{fig:LT_poten}
\end{figure}

Then we present the results of the STFS method with $V = V_3$ and $V = V_4$ for optimal second-order $L^2$- and $H^1$-norm error bounds, respectively. \cref{fig:ST_poten} exhibits that the optimal second-order $L^2$- and $H^1$-norm error bounds can be observed only when $\tau \lesssim h^2$, and there is order reduction when $\tau \gg h^2$ (approximately first order as observed in the numerical results). This observation validates our error bounds in \cref{thm:ST} for the NLSE with low regularity potential and indicates that the step size restriction $\tau \lesssim h^2$ is necessary and optimal. 

\begin{figure}[htbp]
	\centering
	{\includegraphics[width=0.475\textwidth]{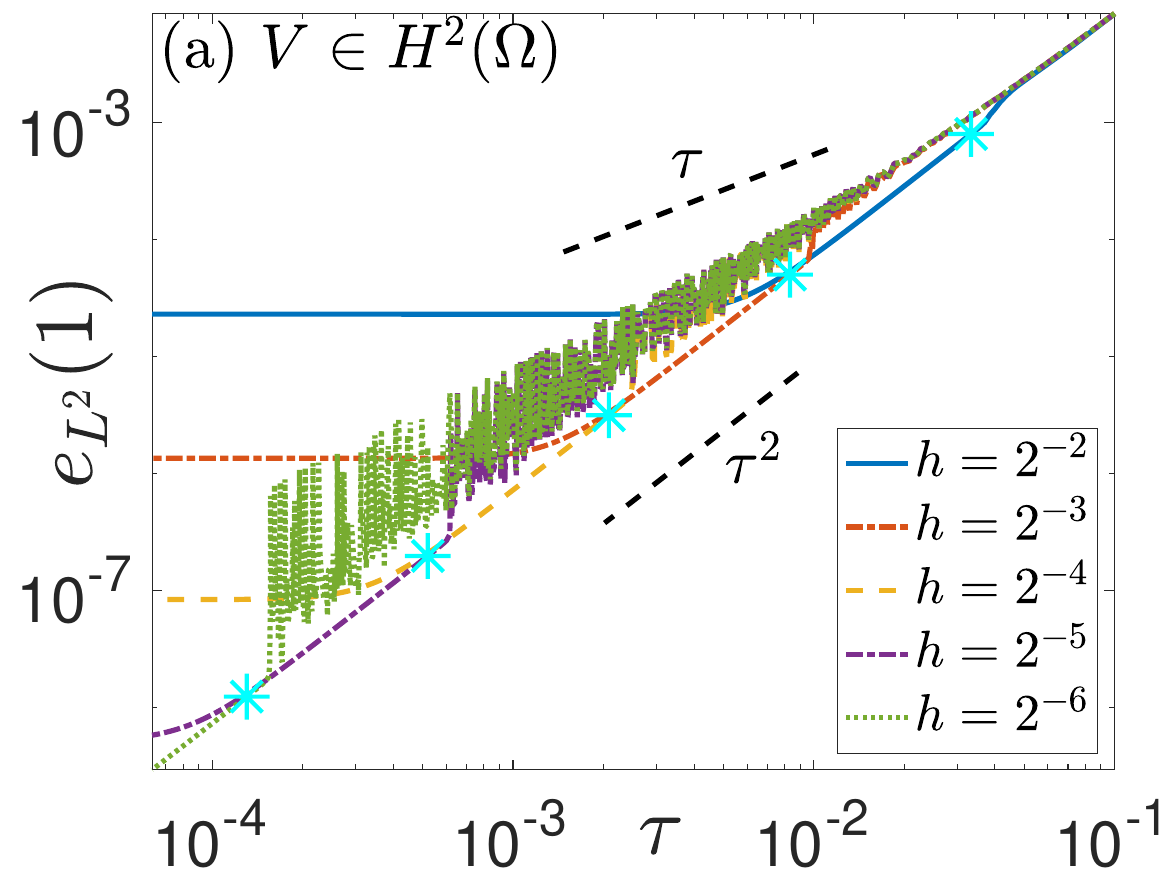}}\hspace{1em}
	{\includegraphics[width=0.475\textwidth]{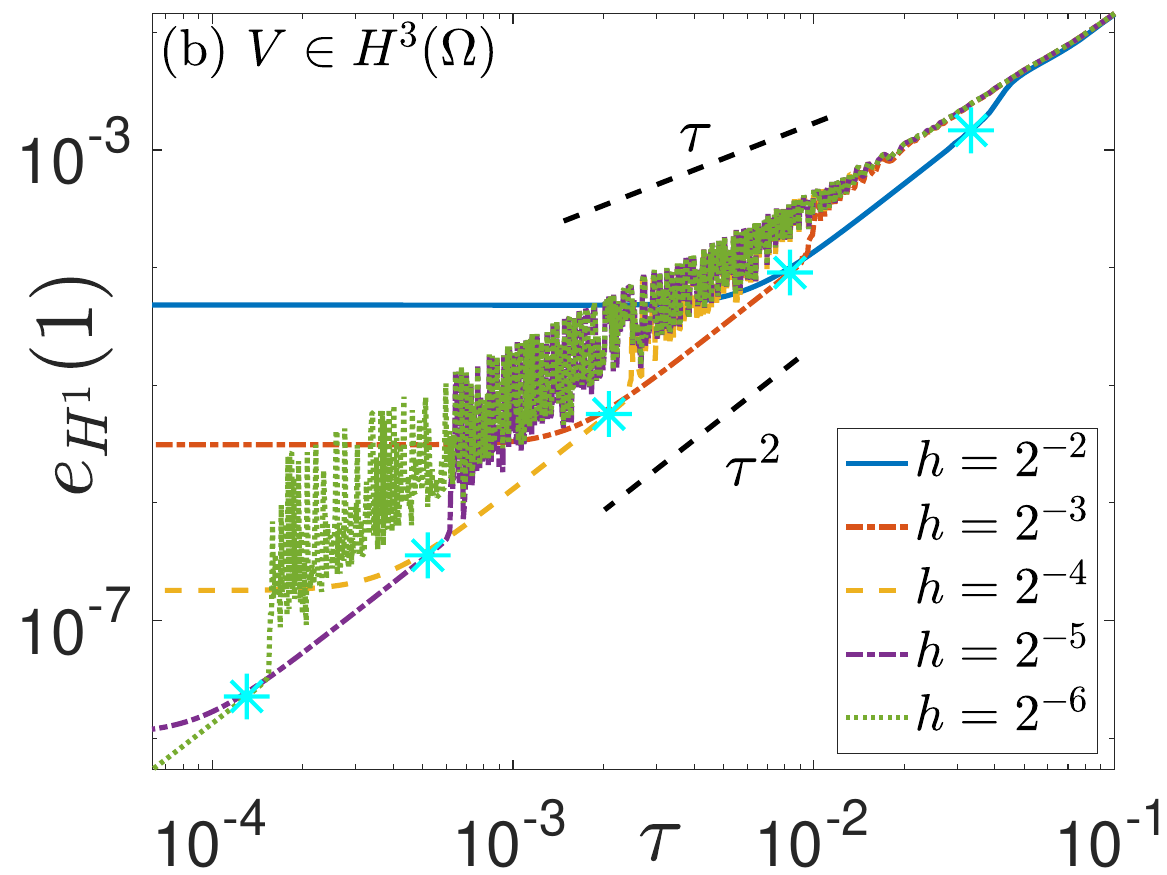}}
	\caption{Temporal errors of the STFS method with different mesh sizes for \cref{eq:NLSE_low_reg_poten}: (a) $ V=V_3 \in H^2_\text{per}(\Omega) $ and (b) $ V = V_4 \in H^3_\text{per}(\Omega) $.}
	\label{fig:ST_poten}
\end{figure}

\subsection{For the NLSE with low regularity nonlinearity}
In this subsection, we consider the NLSE with power-type nonlinearity and without potential as
\begin{equation}\label{eq:NLSE_low_reg_nonl}
	\begin{aligned}
		&i \partial_t \psi(x, t) = -\Delta \psi(x, t) - |\psi(x, t)|^{2\sigma} \psi(x, t), \quad x \in \Omega, \quad t>0, \\
		&\psi_0(x) = xe^{-\frac{x^2}{2}}, \quad x \in \Omega. 
	\end{aligned}
\end{equation}
The initial data is chosen as an odd function such that the solution to \cref{eq:NLSE_low_reg_nonl} will satisfy $\psi(0, t) \equiv 0$ for all $t \geq 0$ to demonstrate the influence of the low regularity of the nonlinearity at the origin. Moreover, we numerically check that the exact solution $ \psi(t) \in H^{3.5+2\sigma}(\Omega) $ and thus the assumptions on the exact solution in \cref{thm:LT,thm:ST} are satisfied. 

In this subsection, we shall test the convergence orders of the LTFS \cref{full_discrete_LT} and the STFS \cref{full_discrete_ST} for the NLSE \cref{eq:NLSE_low_reg_nonl} with different $\sigma > 0$. The 'exact' solutions are computed by the STFS \cref{full_discrete_ST} with $ \tau = \tau_\text{e} := 10^{-6} $ and $ h = h_\text{e} := 2^{-9} $. 

\cref{fig:LT_nonl_diff_h} exhibits the temporal errors in $L^2$- and $H^1$-norm of the LTFS method with different mesh size $h$ for $\sigma = 0.1\ (\sigma>0)$ and $\sigma = 0.5 \ (\sigma \geq 1/2)$, respectively. \cref{fig:LT_nonl_diff_h} (a) shows that the temporal convergence is first order in $L^2$-norm when $\sigma = 0.1$ and \cref{fig:LT_nonl_diff_h} (b) shows that the temporal convergence is first order in $H^1$-norm when $\sigma = 0.5$. The results in \cref{fig:LT_nonl_diff_h} confirm our optimal error bounds in \cref{thm:LT} for the NLSE with low regularity nonlinearity. 

\cref{fig:ST_nonl_diff_h} displays the temporal errors in $L^2$- and $H^1$-norm of the STFS method with different mesh size $h$ for $\sigma = 1.1 \ (\sigma \geq 1)$ and $\sigma = 1.5 \ (\sigma \geq 3/2)$, respectively. \cref{fig:ST_nonl_diff_h} (a) shows that the temporal convergence is second order in $L^2$-norm when $\sigma = 1.1$ and \cref{fig:ST_nonl_diff_h} (b) shows that the temporal convergence is second order in $H^1$-norm when $\sigma = 1.5$. The results in \cref{fig:ST_nonl_diff_h} confirm our optimal error bounds in \cref{thm:ST} for the NLSE with low regularity nonlinearity. 

Nevertheless, the numerical results in \cref{fig:LT_nonl_diff_h,fig:ST_nonl_diff_h} indicate that the temporal convergence order seems to be independent of the mesh size $h$, which suggests that the step size restriction $\tau \lesssim h^2$ may be relaxed in cases of purely low regularity nonlinearity. This phenomenon will be further investigated in our future work. It is worth noting that this step size restriction remains necessary and optimal in the presence of low regularity potential, as discussed in the preceding subsection. 

\begin{figure}[htbp]
	\centering
	{\includegraphics[width=0.475\textwidth]{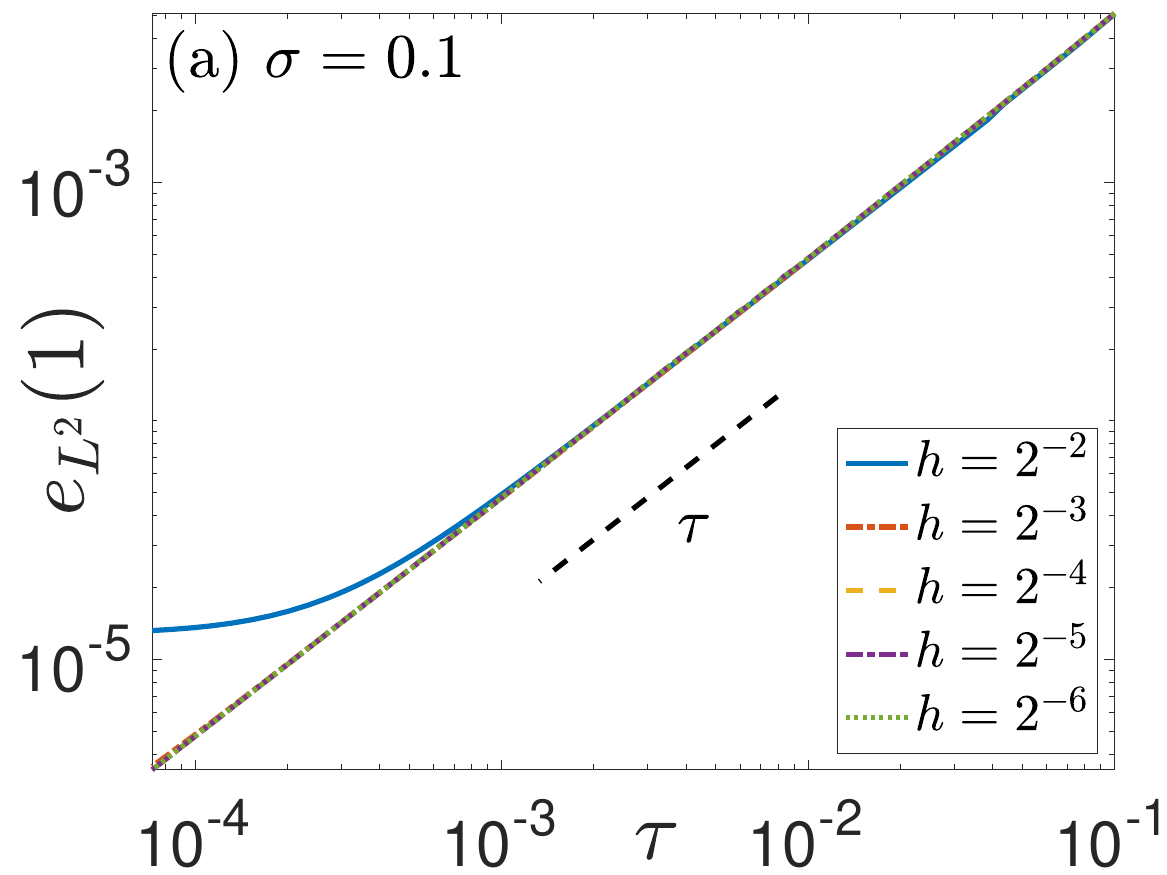}}\hspace{1em}
	{\includegraphics[width=0.475\textwidth]{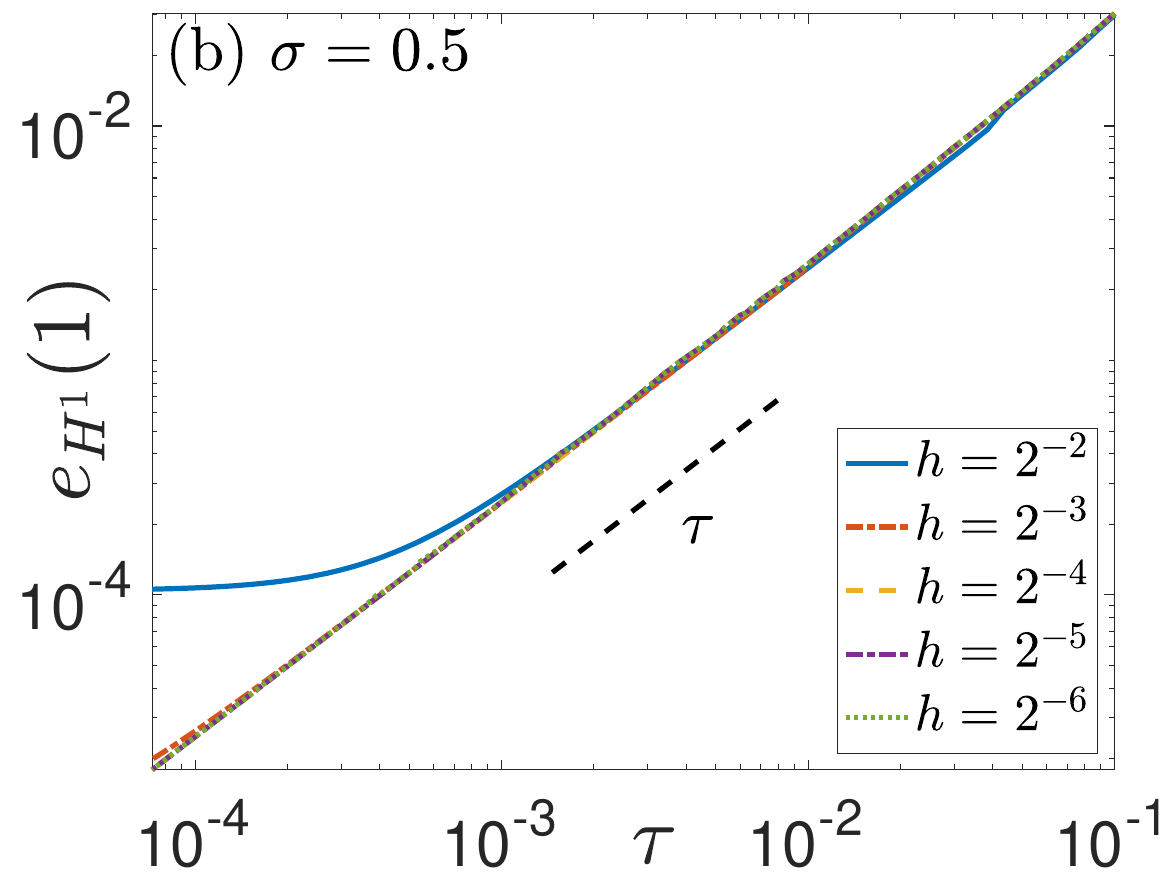}}
	\caption{Temporal errors of the LTFS method with different mesh sizes for \cref{eq:NLSE_low_reg_nonl}: (a) $\sigma = 0.1$ and (b) $ \sigma = 0.5 $. }
	\label{fig:LT_nonl_diff_h}
\end{figure}

\begin{figure}[htbp]
	\centering
	{\includegraphics[width=0.475\textwidth]{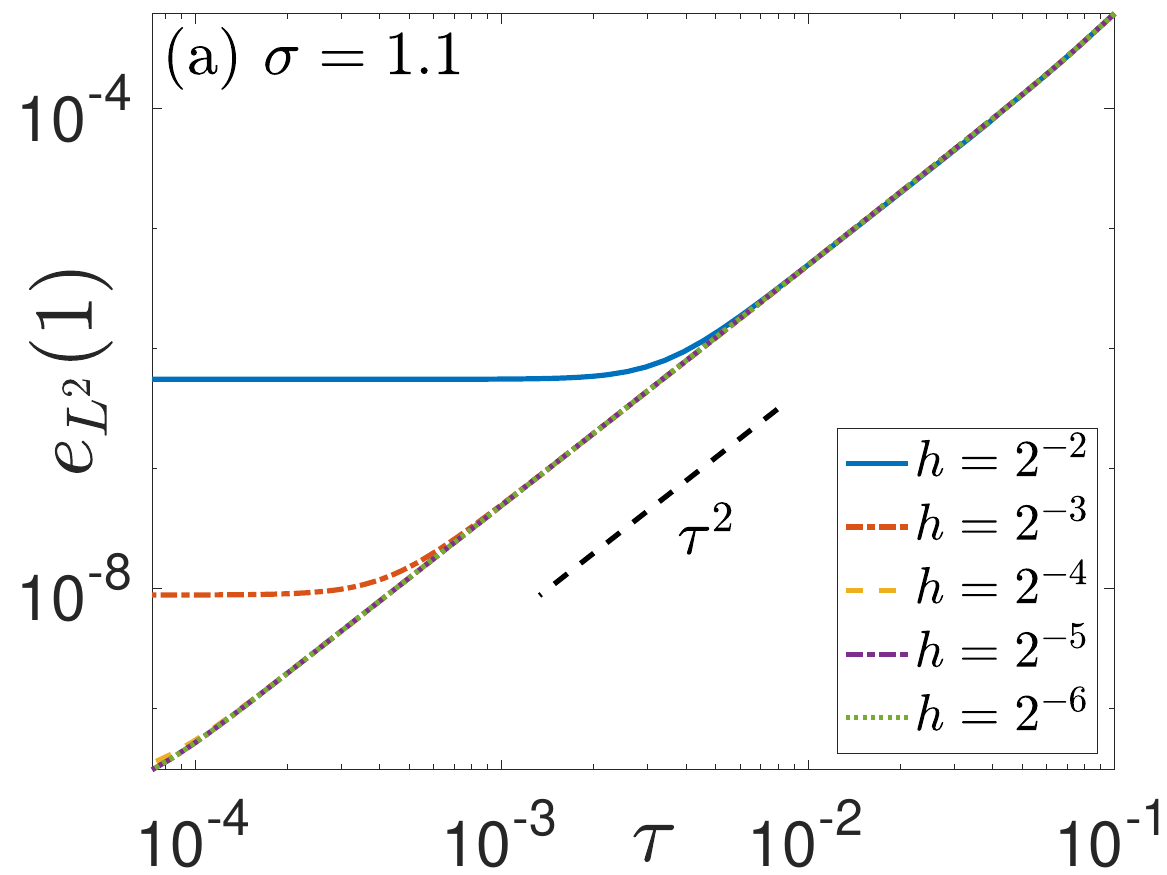}}\hspace{1em}
	{\includegraphics[width=0.475\textwidth]{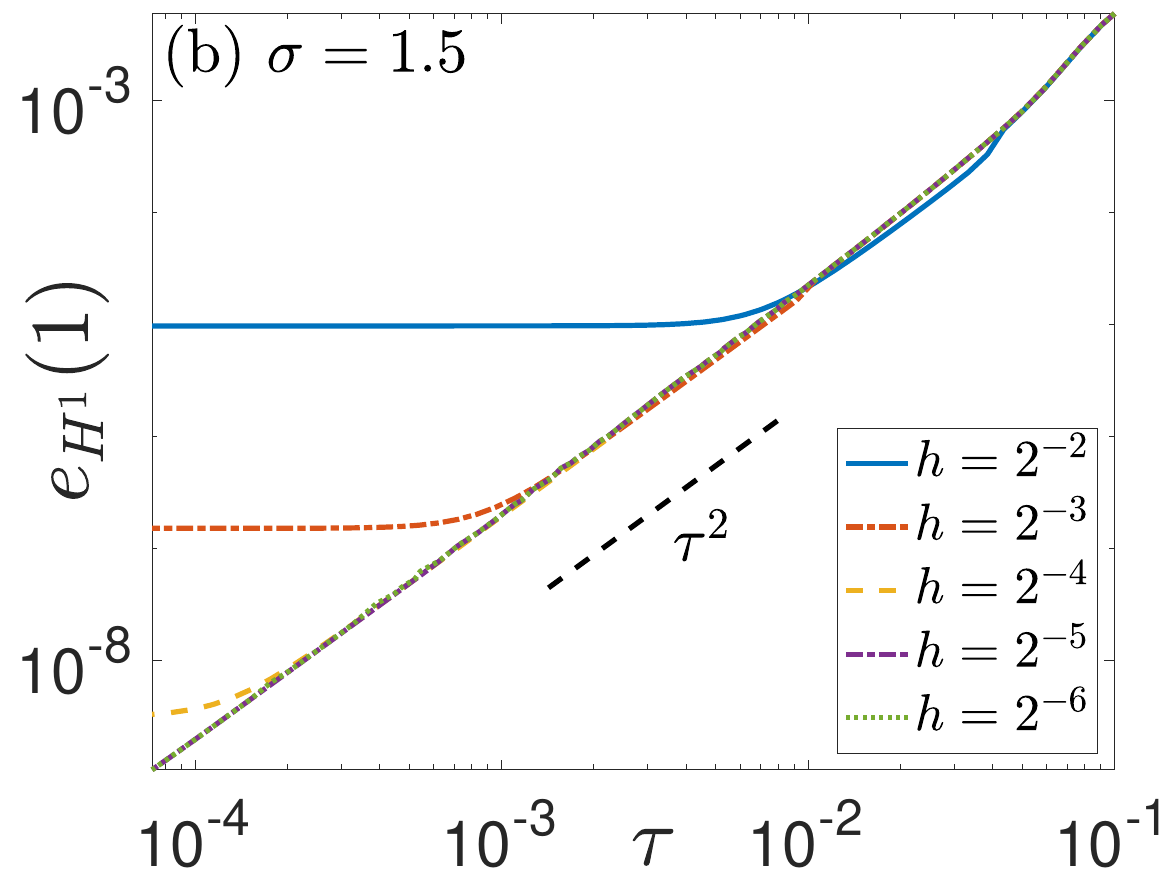}}
	\caption{Temporal errors of the STFS method with different mesh sizes for \cref{eq:NLSE_low_reg_nonl}: (a) $ \sigma = 1.1 $ and (b) $ \sigma = 1.5 $. }
	\label{fig:ST_nonl_diff_h}
\end{figure}

Then we show the temporal errors of the LTFS and the STFS methods for different $\sigma>0$ with a fixed $h = h_e = 2^{-9}$, which can be regarded as testing the convergence at the time semi-discrete level. We choose several time steps $\tau$ from $10^{-4}$ to $10^{-1}$. From \cref{fig:LT_no_step_size_restriction,fig:ST_no_step_size_restriction}, we observe that: when there is no step size restriction, first-order $L^2$-norm error bound holds for any $\sigma>0$ while the first-order $H^1$-norm error bound only holds for $\sigma \geq 1/2$; and second-order $L^2$- and $H^1$-norm error bounds can only hold for $\sigma\geq1$ and $\sigma \geq 3/2$, respectively. This observation implies that the threshold values of $\sigma$ in \cref{thm:LT,thm:ST} for optimal convergence orders are indeed sharp at semi-discrete level. 
\begin{figure}[htbp]
	\centering
	{\includegraphics[width=0.475\textwidth]{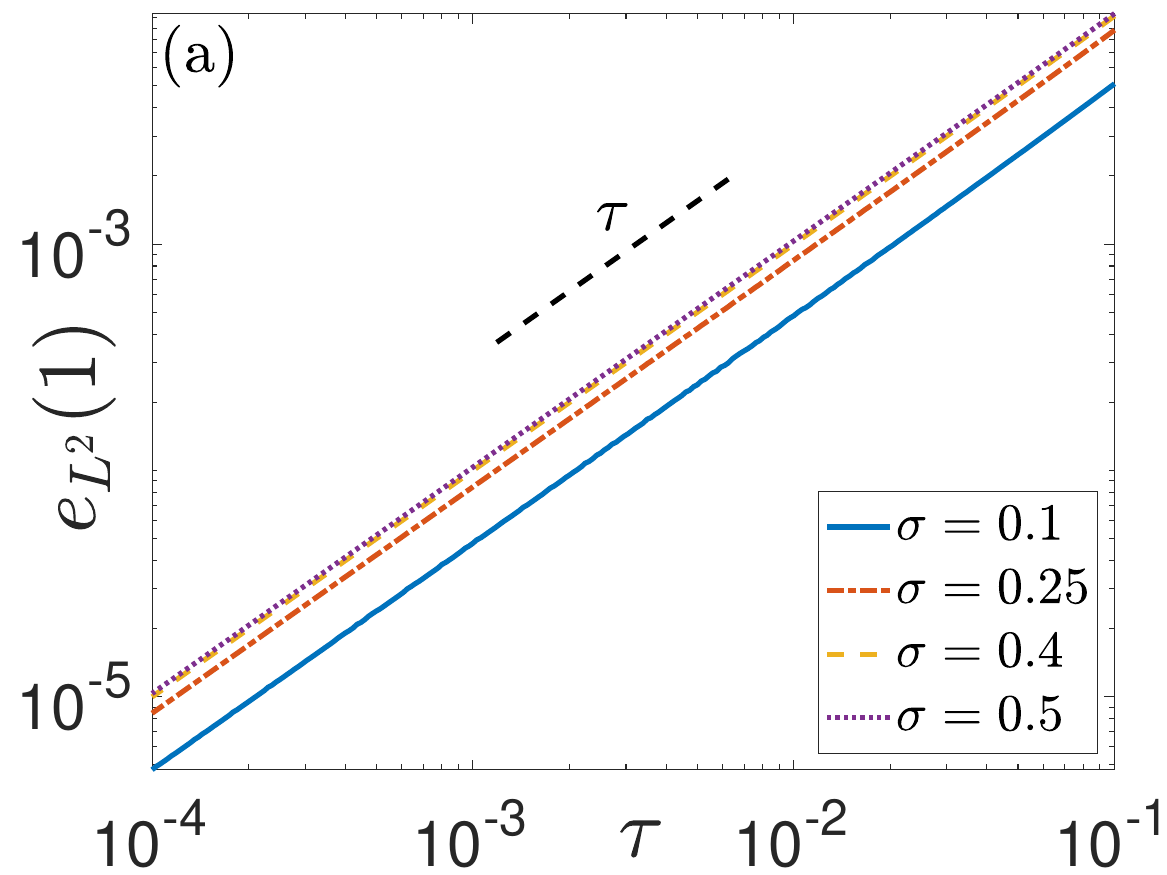}}\hspace{1em}
	{\includegraphics[width=0.475\textwidth]{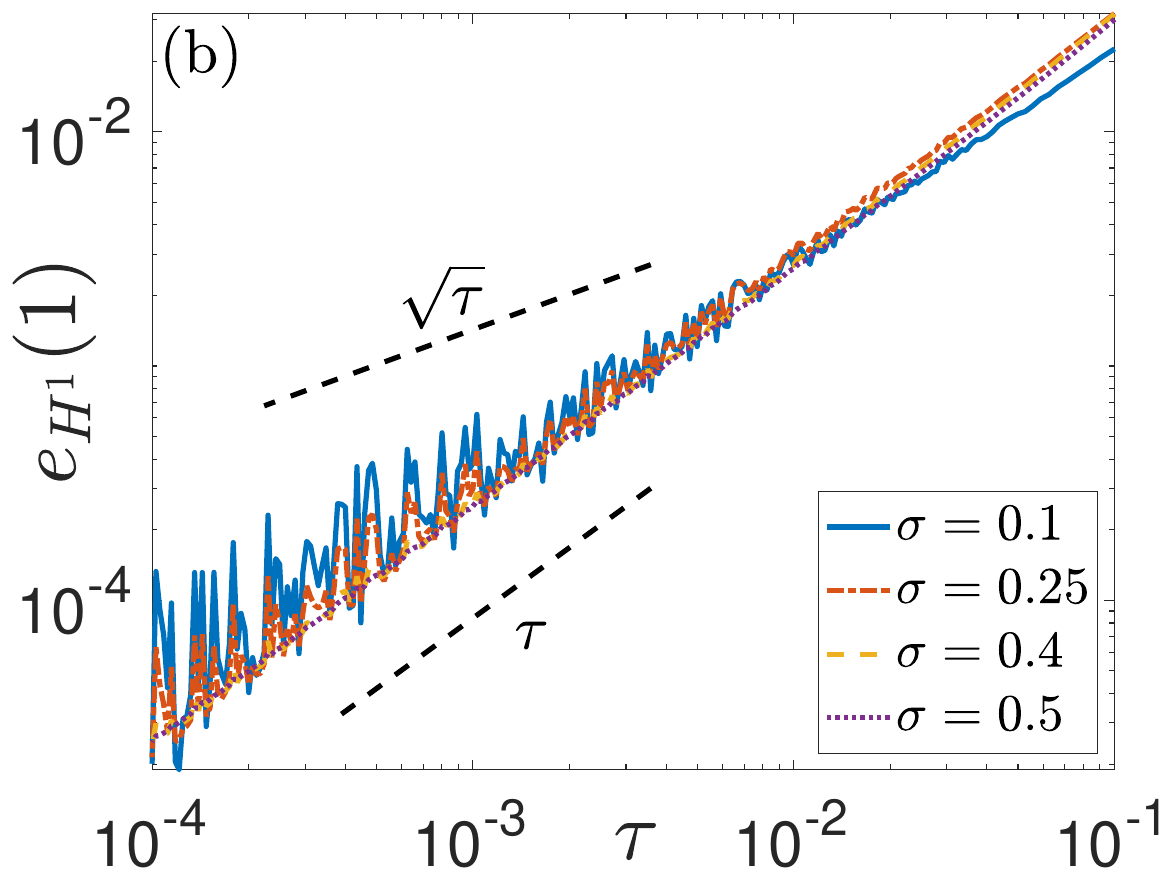}}
	\caption{Temporal errors of the LTFS method for \cref{eq:NLSE_low_reg_nonl} with fixed mesh size $ h=h_e $ and different values of $\sigma$: (a) $L^2$-error and (b) $H^1$-error. }
	\label{fig:LT_no_step_size_restriction}
\end{figure}

\begin{figure}[htbp]
	\centering
	{\includegraphics[width=0.475\textwidth]{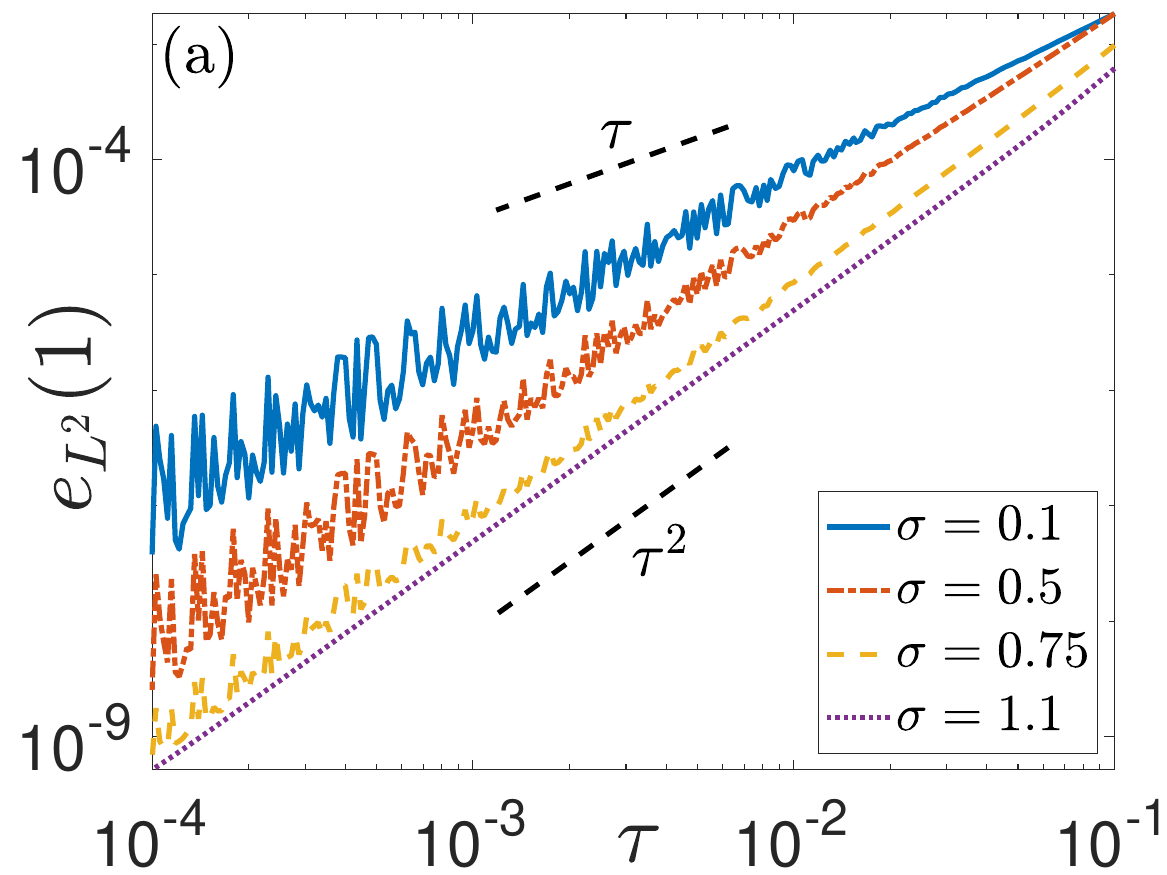}}\hspace{1em}
	{\includegraphics[width=0.475\textwidth]{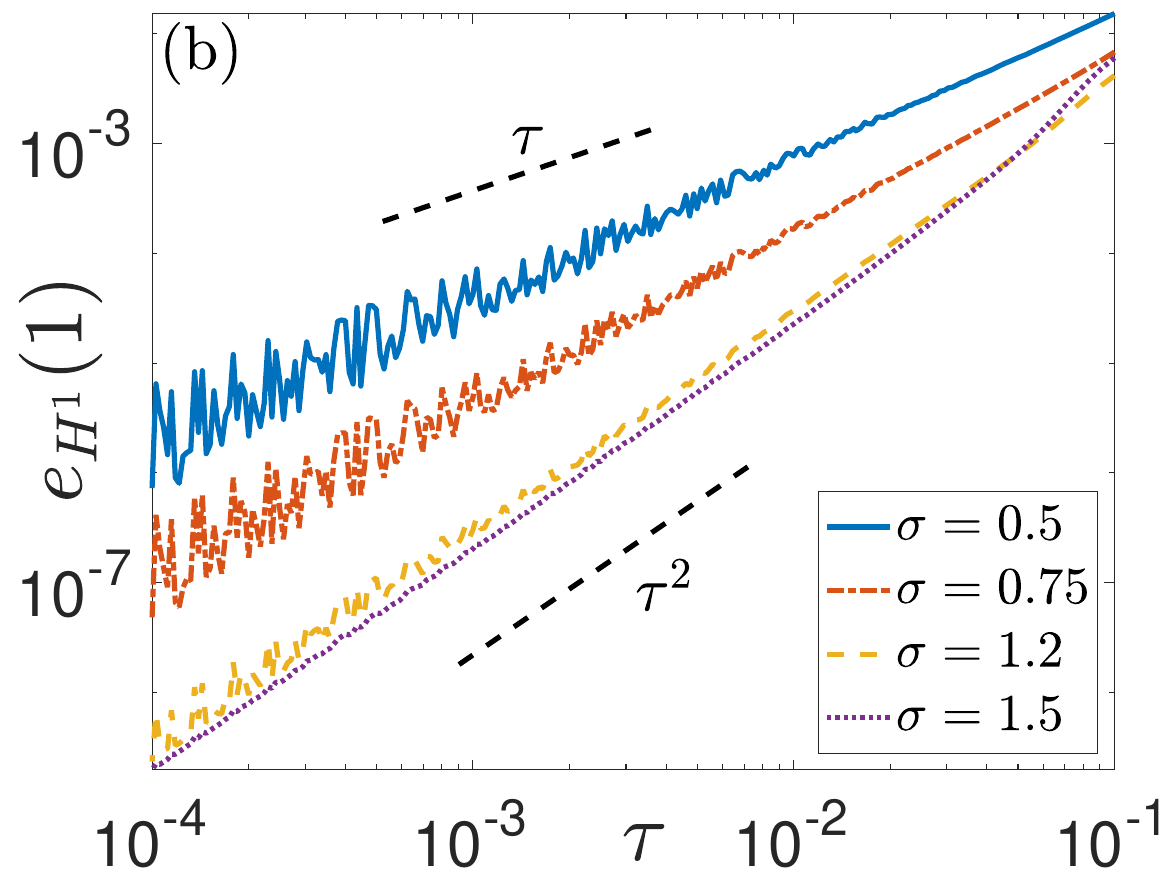}}
	\caption{Temporal errors of the STFS method for \cref{eq:NLSE_low_reg_nonl} with fixed mesh size $ h=h_e $ and different values of $\sigma$: (a) $L^2$-error and (b) $H^1$-error. }
	\label{fig:ST_no_step_size_restriction}
\end{figure}

\subsection{Comparison with the first-order Gautschi-type EWI}
In this subsection, we shall compare the performance of the first-order Lie-Trotter time-splitting method and the first-order Gautschi-type EWI (also known as the exponential Euler scheme \cite{ExpInt}) analyzed in \cite{bao2023_EWI} when applied to the NLSE with low regularity potential and nonlinearity. Since the EWI is a first-order scheme, here, we only present the comparison between the EWI and the first-order LTFS method \cref{full_discrete_LT}. 

According to our analysis before, for the time-splitting method, we only present the results computed at the fully discrete level by using Fourier spectral method for spatial discretization subject to the CFL-type time step size restriction $\tau < h^2/\pi$. While for the EWI, we present the results computed with fixed $h = h_e$ since the mesh size $h$ has almost no influence on the performance of the EWI as long as it is small enough (such that the spatial error is smaller than the temporal error). The comparison between the time-splitting method and the EWI both at the semi-discrete level (i.e. $\tau \gg h^2$) can be found in Section 5.3 of \cite{bao2023_EWI}. 

In \cref{fig:comp_poten} (a), we show the $L^2$-errors of using the LTFS and the EWI to solve the NLSE \cref{eq:NLSE_low_reg_poten} with $V(x) = V_1(x)$ defined in \cref{eq:poten}. In \cref{fig:comp_poten} (b), we show the $H^1$-errors of using the LTFS and the EWI to solve the NLSE \cref{eq:NLSE_low_reg_poten} with $V(x) = V_2(x)$ defined in \cref{eq:poten}. We can observe that both the LTFS and the EWI show optimal first-order convergence in each case, while the value of errors of the LTFS is smaller than those of the EWI. 

In \cref{fig:comp_nonl} (a), we show the $L^2$-errors of using the LTFS and the EWI to solve the NLSE \cref{eq:NLSE_low_reg_nonl} with $\sigma = 0.1$. In \cref{fig:comp_nonl} (b), we show the $H^1$-errors of using the LTFS and the EWI to solve the NLSE \cref{eq:NLSE_low_reg_nonl} with $\sigma = 0.5$ defined in \cref{eq:poten}. We can observe that both the LTFS and the EWI show optimal first-order convergence in each case, while the value of errors of the LTFS is smaller than those of the EWI. 

Based on the observation above, we can conclude that the time-splitting methods perform better than the EWI at fully discrete level under the CFL-type time step size restriction $\tau < h^2/\pi$. 
%This may due to the structure-preserving property of the time-splitting methods (see \cite{Ant} for more detailed discussion), and necessitates the design and analysis of structure preserving EWIs. 

\begin{figure}[htbp]
	\centering
	{\includegraphics[width=0.475\textwidth]{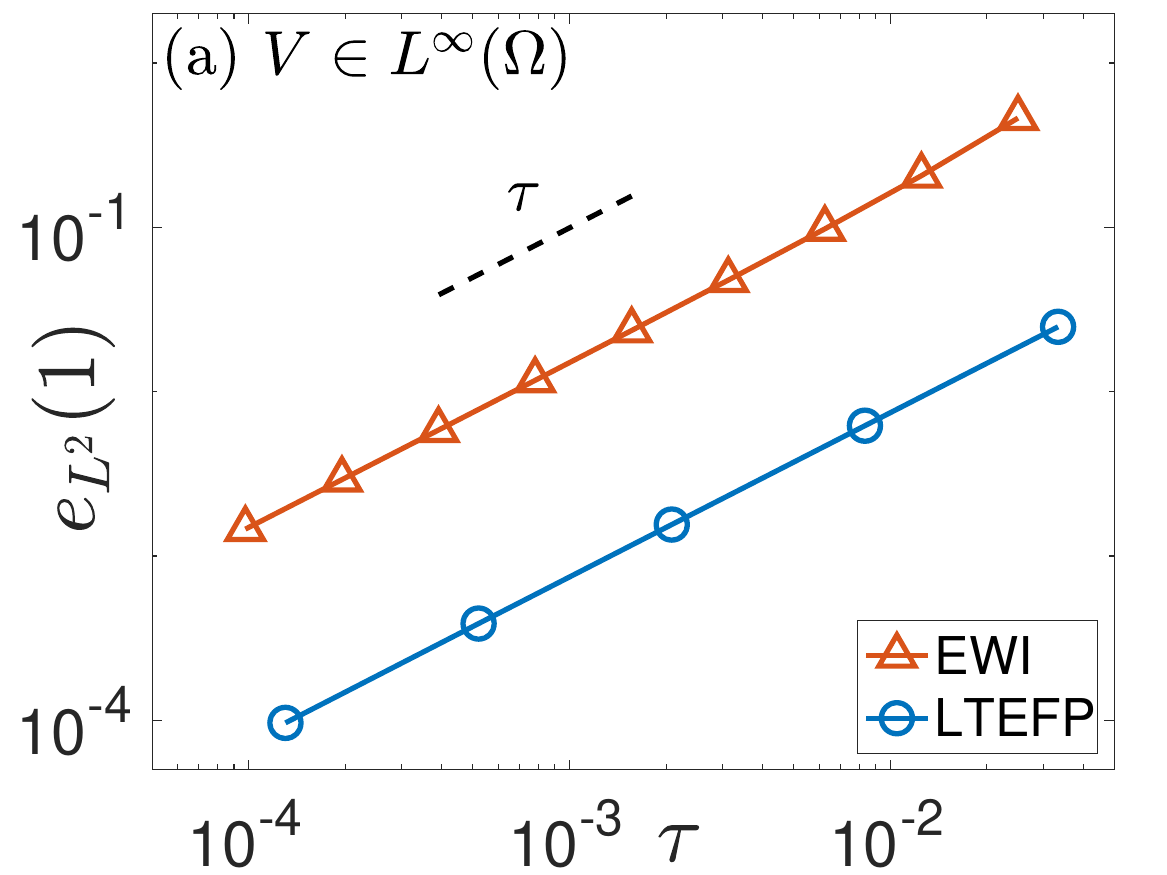}}\hspace{1em}
	{\includegraphics[width=0.475\textwidth]{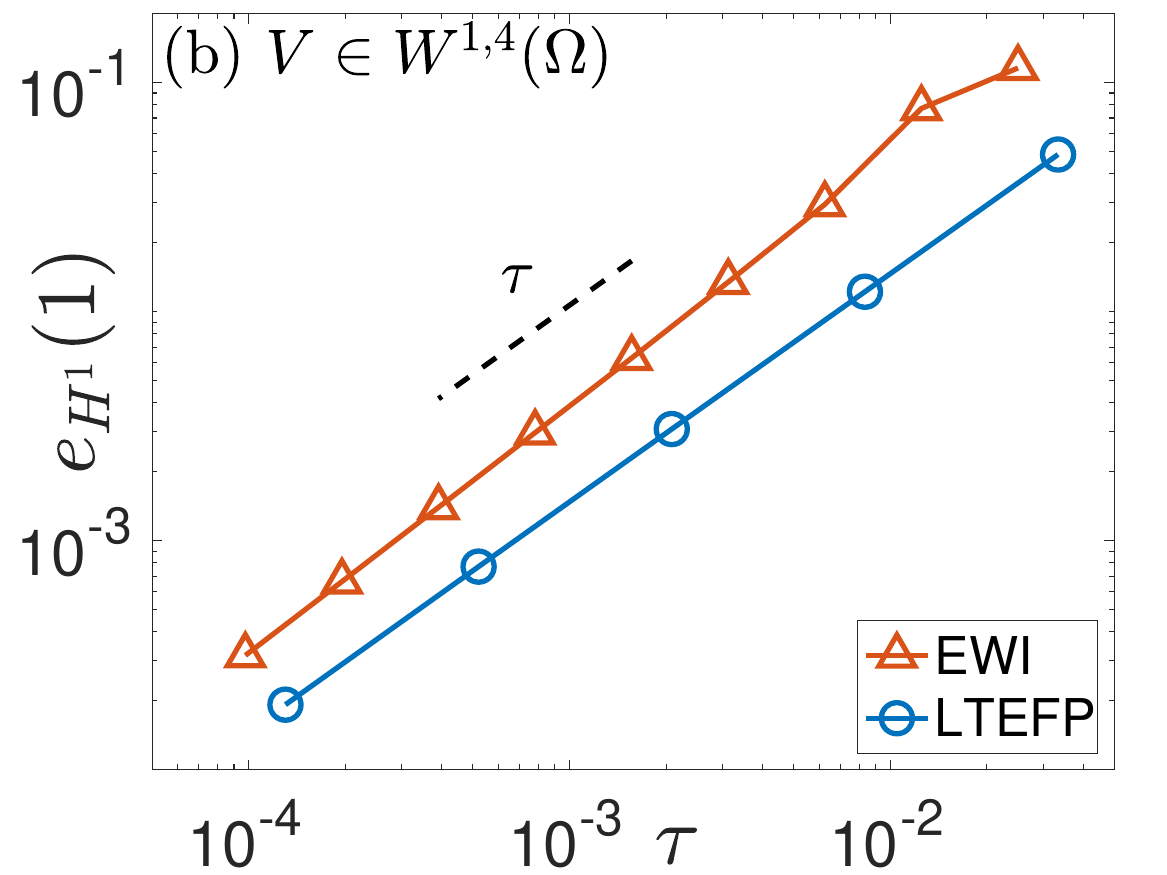}}
	\caption{Temporal errors of the LTFS method and the EWI for solving \cref{eq:NLSE_low_reg_poten}: (a) $V = V_1 \in L^\infty(\Omega)$ and (b) $V = V_2 \in W^{1, 4} (\Omega)$. }
	\label{fig:comp_poten}
\end{figure}

\begin{figure}[htbp]
	\centering
	{\includegraphics[width=0.475\textwidth]{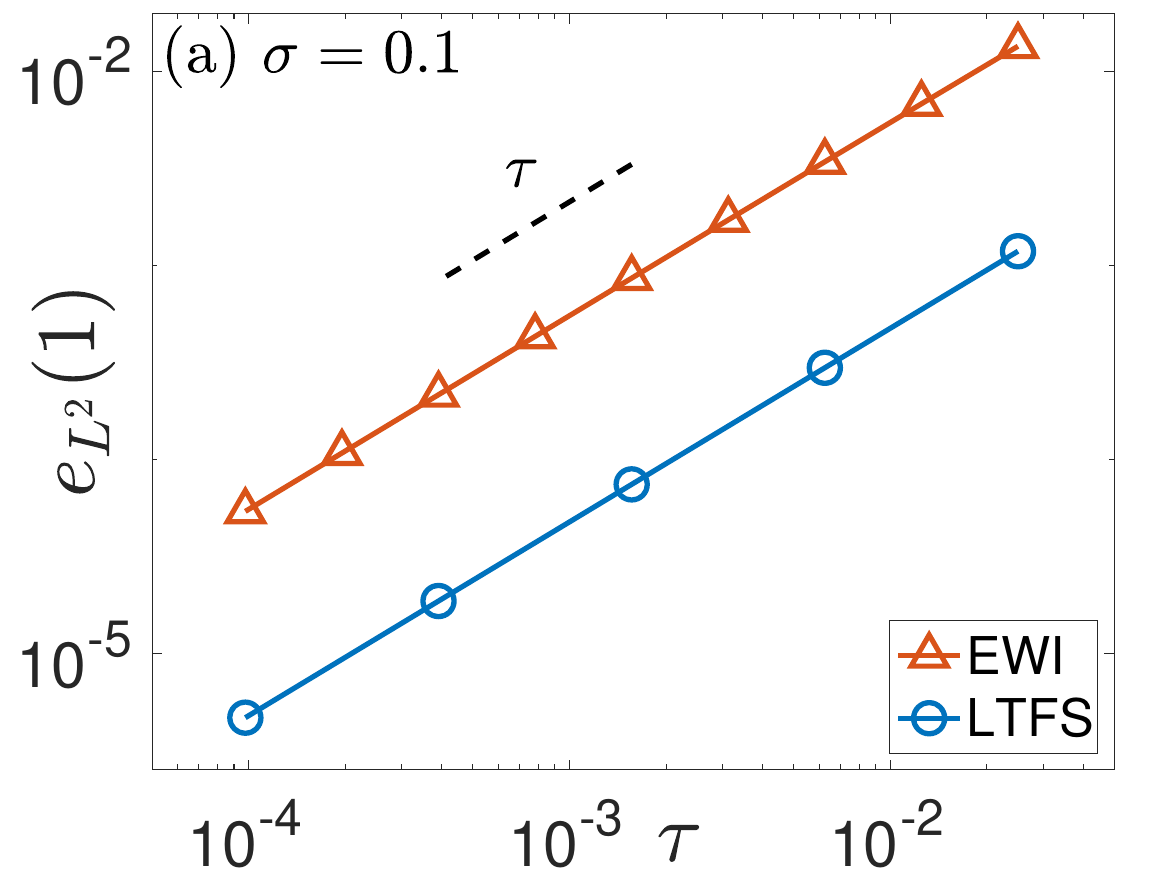}}\hspace{1em}
	{\includegraphics[width=0.475\textwidth]{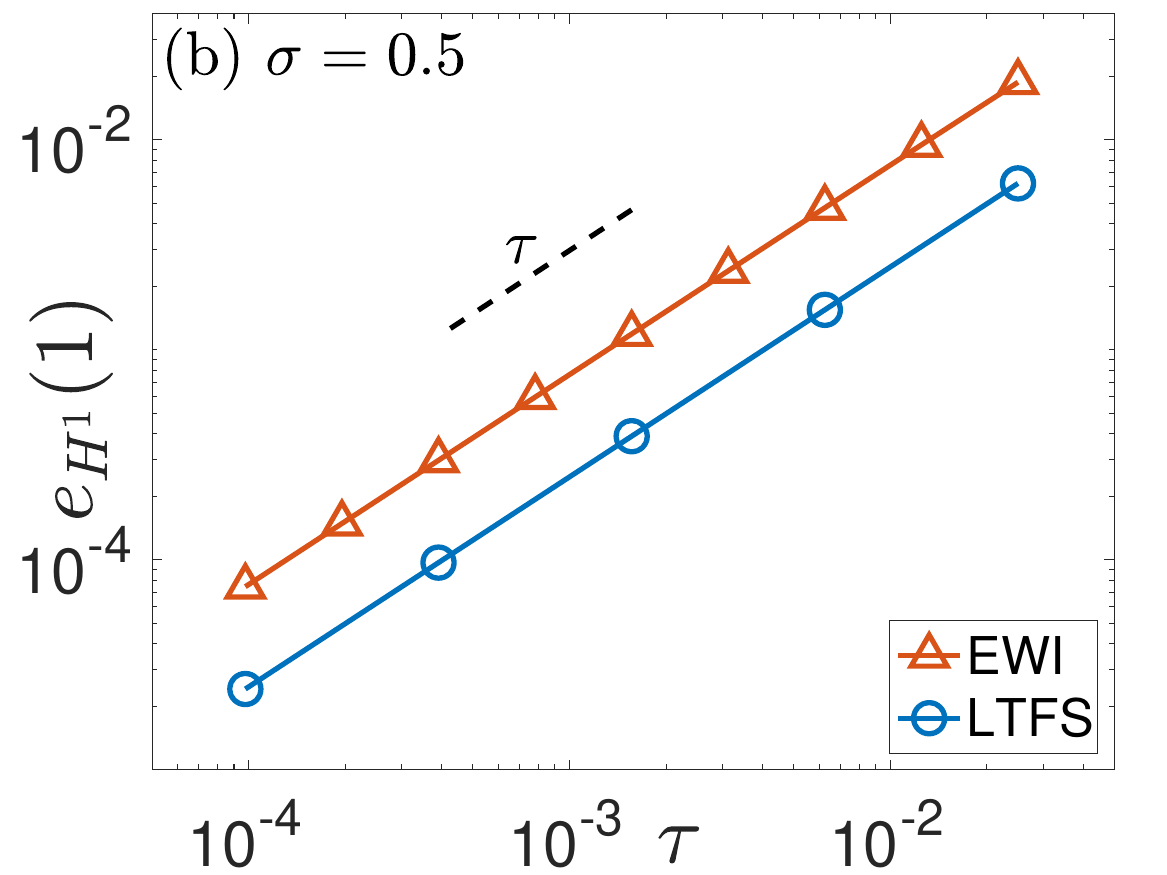}}
	\caption{Temporal errors of the LTFS method and the EWI for solving \cref{eq:NLSE_low_reg_nonl}: (a) $\sigma =  0.1$ and (b) $\sigma = 0.5$. }
	\label{fig:comp_nonl}
\end{figure}

\section{Conclusion}
We established optimal error bounds on time-splitting Fourier spectral methods for the nonlinear Schr\"odinger equation with low regularity potential $V$ and typical power-type nonlinearity $ f(\rho) = \rho^\sigma(\sigma>0) $. For the first-order Lie-Trotter time-splitting method, optimal $ L^2 $- and $ H^1 $-norm error bounds were established for $V \in L^\infty, \sigma > 0 $ and $ V \in W^{1, 4}, \sigma \geq 1/2 $, respectively. For the second-order Strang time-splitting method, optimal $ L^2 $- and $H^1$-norm error bounds were established for $V \in H^2, \sigma \geq 1 $ and $V \in H^3, \sigma \geq 3/2 (\text{or } \sigma=1) $, respectively. Compared to the error estimates of the time-splitting methods in the literature, our optimal error bounds require much weaker regularity on potential and nonlinearity for optimal convergence rates. 
%We also proposed an extended Fourier pseudospectral method for spatial discretization, which has computational cost similar to the standard Fourier pseudospectral method, and proved its optimal error bounds under slightly stronger regularity assumptions on the nonlinearity. 
Extensive numerical results were reported to validate our error estimates and to show that they are sharp. In particular, a CFL-type time step size restriction which is needed in our proof can be observed numerically in the presence of low regularity potential, indicating that it cannot be removed or improved. Furthermore, numerical results for various values of $\sigma>0$ suggest that the above threshold values of $ \sigma $ for optimal convergence orders are also sharp.

\end{document}